\documentclass[12pt]{article}
\usepackage{amssymb,amsmath}
\usepackage{graphicx}
\newtheorem{thm}{Theorem}
\newtheorem{lem}{Lemma}
\newtheorem{clm}{Claim}

\newcommand\RR{{\mathbb R}}
\newcommand\myitemsep{\setlength{\itemsep}{-2pt}}
\newenvironment{proof}%
{\noindent{\emph{Proof.}\ }}%
{\hfill$\square$\par\bigskip}%

\begin{document}
\setlength{\unitlength}{1mm}

\title{Minimal obstructions for 1-immersions and hardness of 1-planarity testing}

\author{ \makebox[10cm]{}
\\ Vladimir P. Korzhik\thanks{This paper was done while the author
visited Simon Fraser University.}\\
National University of Chernivtsi, Chernivtsi\\
and\\
Institute of Applied Problems of Mechanics and Mathematics\\
of National Academy of Science of Ukraine, Lviv\\
Ukraine
 \and \makebox[10cm]{}
\\ Bojan Mohar\thanks{Supported in part by the
  Research Grant P1--0297 of ARRS (Slovenia), by an NSERC Discovery Grant (Canada)
  and by the Canada Research Chair program.}~\thanks{On leave from:
  IMFM \& FMF, Department of Mathematics, University of Ljubljana, Ljubljana,
  Slovenia.} \\
Department of Mathematics \\
Simon Fraser University \\
Burnaby, B.C. V5A 1S6 \\ Canada}

\date{}

\maketitle

\newpage
\begin{abstract}
A graph is \emph{$1$-planar} if it can be drawn on the
plane so that each edge is crossed by no more than
one other edge (and any pair of crossing edges cross only once). A non-1-planar graph $G$ is
\emph{minimal} if the graph $G-e$ is 1-planar for
every edge $e$ of $G$. We construct two infinite
families of minimal non-1-planar graphs and show that
for every integer $n\geq 63$, there are at least
$2^{(n-54)/4}$ nonisomorphic minimal
non-1-planar graphs of order $n$.
It is also proved that testing 1-planarity is NP-complete.
\end{abstract}

\vspace{1cm}

\noindent\textbf{Running head:}\\
Obstructions for 1-immersions\\

\vspace{1cm}

\noindent\textbf{Corresponding author:} Bojan Mohar\\

\vspace{1cm}

\noindent\textbf{AMS classifications:}\\
05C10, 05B07.

\vspace{1cm}

\noindent\textbf{Keywords:}\\
Topological graph, crossing edges, 1-planar graph, 1-immersion.

\newpage

\section{Introduction}\label{Intr}

A graph drawn in the plane is \emph{$1$-immersed\/} in the plane
if any edge is crossed by at most one other edge (and any pair of crossing edges cross only once). A graph is
\emph{$1$-planar} if it can be 1-immersed into the plane.
It is easy to see that if a graph has 1-immersion in which
two edges $e,f$ with a common endvertex cross, then the
drawing of $e$ and $f$ can be changed so that these two
edges no longer cross. Consequently, we may assume that
adjacent edges are never crossing each other and that
no edge is crossing itself. We take this assumption as
a part of the definition of 1-immersions since this limits
the number of possible cases when discussing 1-immersions.

The notion of 1-immersion
of a graph was introduced by Ringel \cite{R} when
trying to color the vertices and faces of a plane
graph so that adjacent or incident elements receive
distinct colors. In the last two decades this class of
graphs received additional attention because of its
relationship to the family of \emph{map graphs}, see
\cite{CGP,CGP2} for further details.

Little is known about 1-planar graphs. Borodin
\cite{B1,B2} proved that every 1-planar graph is
6-colorable. Some properties of maximal 1-planar
graphs are considered in \cite{S}. It was shown in
\cite{BKRS} that every 1-planar graph is acyclically
20-colorable. The existence of subgraphs of bounded
vertex degrees in 1-planar graphs is investigated in
\cite{FM}. It is known (see \cite{BSW,Ch1,Ch2}) that a
1-planar graph with $n$ vertices has at most $4n-8$
edges and that this upper bound is tight. In the
paper \cite{CK} it was observed that the class of
1-planar graphs is not closed under the operation of
edge-contraction.

Much less is known about non-1-planar graphs.
The basic question is how to recognize 1-planar graphs.
This problem is clearly in NP, but it is not clear at all
if there is a polynomial time recognition algorithm.
We shall answer this question by proving that 1-planarity
testing problem is NP-complete.

The recognition problem is closely related to the study of
minimal obstructions for 1-planarity.
A graph $G$ is said to be a \emph{minimal} non-1-planar graph
(\emph{{\rm MN}-graph}, for short) if $G$ is not 1-planar, but
$G-e$ is 1-planar for every edge $e$ of $G$.
An obvious question is:

\begin{quote}
How many MN-graphs are there? Is their number finite?
If not, can they be characterized?
\end{quote}

\noindent
The answer to the first question is not hard: there are
infinitely many.
This was first proved in \cite{K}.
Here we present two additional simple arguments implying the
same conclusion.

\medskip

{\bf Example 1.}
Let $G$ be a graph such that
$t = \lceil{\rm cr}(G)/|E(G)|\rceil - 1\geq 1$,
where ${\rm cr}(G)$ denotes the crossing
number of $G$. Let $G_t$ be the graph obtained from
$G$ by replacing each edge of $G$ by a path of length
$t$. Then $|E(G_t)| = t|E(G)| < {\rm cr}(G) = {\rm
cr}(G_t)$. This implies that $G_t$ is not 1-planar.
However, $G_t$ contains an MN-subgraph $H$. Clearly,
$H$ contains at least one subdivided edge of $G$ in
its entirety, so $|V(H)| > t$. Since $t$ can be
arbitrarily large (see, for example, the well-known lower bound on ${\rm cr}(K_n)$), this shows that there are
infinitely many MN-graphs.

\medskip

Before giving the next example, it is worth noticing that 3-cycles must be embedded in a planar way in every 1-immersion of a graph in the plane.

\medskip

{\bf Example 2.}
Let $K\in\{ K_5, K_{3,3} \}$ be one of the Kuratowski graphs.
For each edge $xy\in E(K)$, let $L_{xy}$ be a 5-connected triangulation
of the plane and $u,v$ be adjacent vertices of $L_{xy}$ whose degree
is at least 6. Let $L'_{xy} = L_{xy}-uv$. Now replace each edge $xy$
of $K$ with $L'_{xy}$ by identifying $x$ with $u$ and $y$ with $v$.
It is not hard to see that the resulting graph $G$ is not 1-planar
(since two of graphs $L'_{xy}$ must ``cross each other'', but that is
not possible since they come from 5-connected triangulations).
Again, one can argue that they contain large MN-graphs.

\medskip

The paper \cite{K} and the above examples prove the existence of infinitely
many MN-graphs but do not give any concrete examples.
They provide no information on properties of MN-graphs.
Even the most basic question if there are infinitely many MN-graphs
whose minimum degree is at least three cannot be answered by considering
these constructions. In \cite{K},
two specific MN-graphs of order 7 and 8, respectively, are given.
One of them, the graph $K_7-E(K_3)$, is the unique
7-vertex MN-graph and since all 6-vertex graphs are
1-planar, the graph $K_7-E(K_3)$ is the MN-graph with
the minimum number of vertices. Surprisingly enough, the
two MN-graphs in \cite{K} are the only explicit MN-graphs known
in the literature.

The main problem when trying to construct 1-planar
graphs is that we have no characterization of
1-planar graphs. The set of 1-planar graphs is not
closed under taking minors, so 1-planarity can not be
characterized by forbidding some minors.

In the present paper we construct two explicit infinite
families of MN-graphs whose minimum degree is at least three
and, correspondingly, we give
two different approaches how to prove that a graph
has no plane 1-immersion.

In Sect.~\ref{Sec:2} we construct MN-graphs based
on the Kuratowski graph $K_{3,3}$. To obtain them,
we replace six edges of $K_{3,3}$ by some
special subgraphs. The minimality of these examples is easy
to verify, but their non-1-planarity needs long and delicate
arguments. Using these MN-graphs, we show that for
every integer $n\geq 63$, there are at least $2^{(n-54)/4}$
nonisomorphic minimal non-1-planar graphs of order $n$.
In Sect.~\ref{Sec:3} we describe a class of
3-connected planar graphs that have no plane
1-immersions with at least one crossing point
(\emph{{\rm PN}-graphs}, for short). Every
PN-graph has a unique plane 1-immersion, namely, its
unique plane embedding. Hence, if a
1-planar graph $G$ contains a PN-graph $H$ as a subgraph,
then in every plane 1-immersion of $G$ the subgraph
$H$ is 1-immersed in the plane in the same way.

Having constructions of PN-graphs, we can construct
1-planar and non-1-planar graphs with some desired
properties: 1-planar graphs that have exactly $k>0$
different plane 1-immersions; MN-graphs, etc.

In Sect.~\ref{Sec:4} we construct MN-graphs based
on PN-graphs. Each of these MN-graphs $G$ has as a
subgraph a PN-graph $H$ and the unique plane
1-immersion of $H$ prevents 1-immersion of the remaining part
of $G$ in the plane.

Despite the fact that minimal obstructions for 1-planarity
(i.e., the MN-graphs) have diverse structure, and despite
the fact that discovering 1-immer\-sions of specific graphs
can be very tricky, it turned out to be a hard problem
to establish hardness of 1-planarity testing.
A solution is given in Sect.~\ref{sect:NPC}, where
we show that 1-planarity testing is NP-complete,
see . The proof is geometric in the sense
that the reduction is from 3-colorability of planar graphs
(or similarly, from planar 3-satisfiability).

In Sect.~\ref{Sec:6} we show how the proof of Theorem~\ref{thm:NPC1} can be modified to obtain a proof that $k$-planarity testing for multigraphs is NP-complete.

An extended abstract of this paper was published in Graph Drawing 2008 \cite{GD08}.

\section{Chain graphs based on $K_{3,3}$}
\label{Sec:2}

Two cycles of a graph are \emph{adjacent} if they share
a common edge. If a graph $G$ is drawn in the plane, then we say that
a vertex $x$ lies \emph{inside} (resp.~\emph{outside}) an
embedded (that is, non-self-intersecting) cycle $C$, if $x$
lies in the interior (resp.\ exterior) of $C$, and
does not lie on $C$. Having two embedded adjacent
cycles $C$ and $C'$, we say that $C$ lies inside
(resp. outside) $C'$ if every point of $C$ either
lies inside (resp. outside) $C'$ or lies on $C'$.
From this point on, by a 1-immersion of a graph we mean a plane
1-immersion. We assume that in 1-immersions,
adjacent edges do not cross each other and no edge
crosses itself. Thus, every 3-cycle of a 1-immersed graph is
embedded in the plane. Hence, given a 3-cycle of a
1-immersed graph, we can speak about its interior and
exterior. We say that an
embedded cycle \emph{separates} two vertices $x$ and
$y$ on the plane, if one of the vertices lies inside
and the other one lies outside the cycle. Two
edges $e$ and $e'$ of a graph $G$ \emph{separate}
vertices $x$ and $y$ of the graph if $x$ and $y$
belong to different connected components of the graph
$G-e-e'$.

Throughout the paper we will deal with 1-immersed graphs.
When an immersion of a graph $G$ is clear from the context,
we shall identify vertices, edges, cycles and subgraphs of
$G$ with their image in $\RR^2$ under the 1-immersion.
Then by a face of a 1-immersion of $G$ we mean any connected component
of $\RR^2\setminus G$.

By using M\"obius transformations combined with homeomorphisms
of the plane it is always possible to exchange the interior and
exterior of any embedded cycle and it is possible to change any
face of a given 1-immersion into the outer face of a 1-immersion.
Formally, we have the following
observation (which we will use without referring to it every time):

\medskip

 (A) Let $C$ be a cycle of a graph $G$. If $G$ has a 1-immersion $\varphi$
in which $C$ is embedded, then $G$ has a 1-immersion $\varphi'$ with the same
number of crossings as $\varphi$, in which
$C$ is embedded and all vertices of $G$, which lie inside $C$ in $\varphi$,
lie outside $C$ in $\varphi'$ and vice versa.

\medskip
Now we begin describing a family of MN-graphs based on the
graph $K_{3,3}$.

By a \emph{link} $L(x,y)$ connecting two vertices $x$
and $y$ we mean any of the graphs shown in
Fig.~\ref{Fig:2.1} where
$\{z,\overline{z}\}=\{x,y\}$. We say that the
vertices $x$ and $y$ are incident with the link.
The links in Figs.~\ref{Fig:2.1}(A) and (B) are called
\emph{A-link} and \emph{B-link}, respectively, and the one
in Fig.~\ref{Fig:2.1}(C) is called a \emph{base link}.
Every link has a \emph{free cycle}: both 3-cycles in
an A-link are its free cycles, while every B-link or base
link has exactly one free cycle (the cycle indicated by thick
lines in Fig.~\ref{Fig:2.1}).

\begin{figure}
\centering
\includegraphics[width=0.5\textwidth]{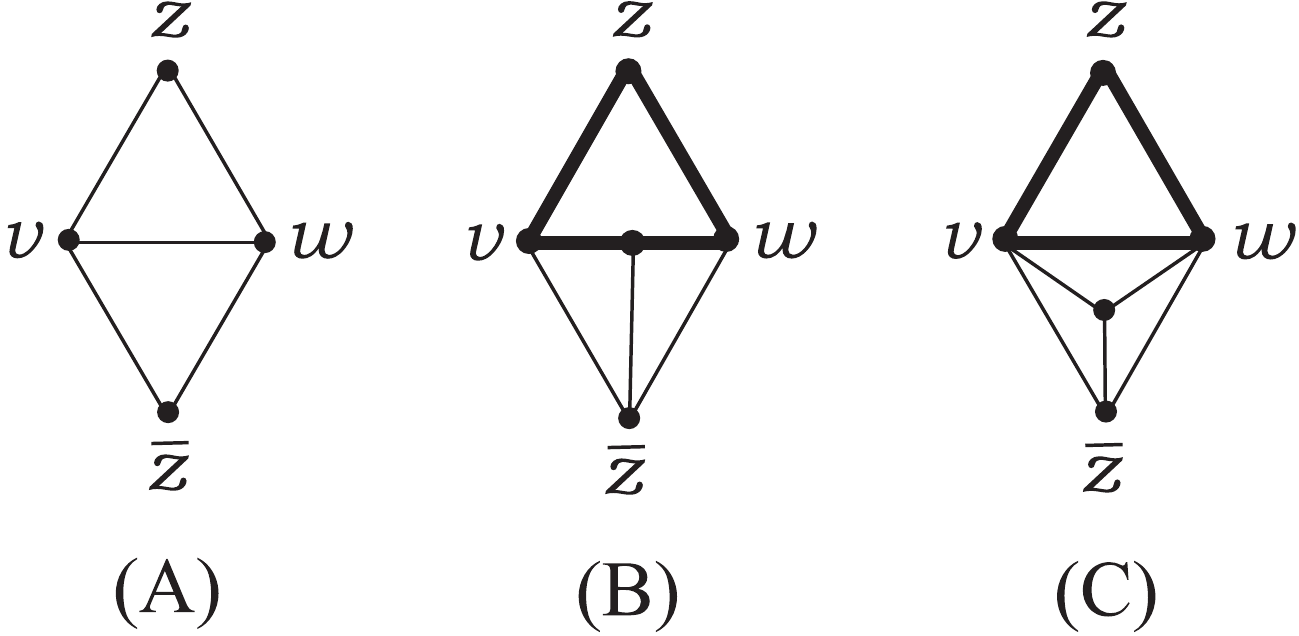}
\caption{Links.}
 \label{Fig:2.1}
  \end{figure}

By an A-\emph{chain} of length $n\geq 2$ we mean the
graph shown in Fig.~\ref{Fig:2.2}(a). By a
B-\emph{chain} of length $n\geq 2$ we mean the graph
shown in Fig.~\ref{Fig:2.2}(c) and, for $n\geq 3$, every graph
obtained from that graph in the following way: for
some integers $h_1,h_2,\ldots,h_t$, where $t\geq 1$
and $1\leq h_1<h_2<\cdots<h_t\leq n-2$, we replace the link
at the left of Fig.~\ref{Fig:2.2}(e) by the link shown
at the right, for $i=1,2,\ldots,t$. Note that, by definition,
A- and B-chains have length at least 2. We say that the
chains in Figs.~\ref{Fig:2.2}(a) and (c) connect the
vertices $v(0)$ and $v(n)$ which are called the
\emph{end vertices} of the chain. Two chains are
\emph{adjacent} if they share a common end vertex.
A-chains and B-chains will be represented as shown
in Figs.~\ref{Fig:2.2}(b) and (d), respectively,
where the arrow points to the end vertex incident
with the base link. The vertices
$v(0),v(1),v(2),\ldots,v(n)$ are the \emph{core
vertices} of the chains. Every free cycle of a link
contains exactly one core vertex. The two edges of a
free cycle $C$ incident to the core vertex
are the \emph{core-adjacent} edges of $C$.
It is easy to see that two edges $e$ and $e'$ of a
chain separate the end vertices of the chain if and
only if the edges are the core-adjacent edges of a
free cycle of a link of the chain.

\begin{figure}
\centering
\includegraphics[width=0.9\textwidth]
{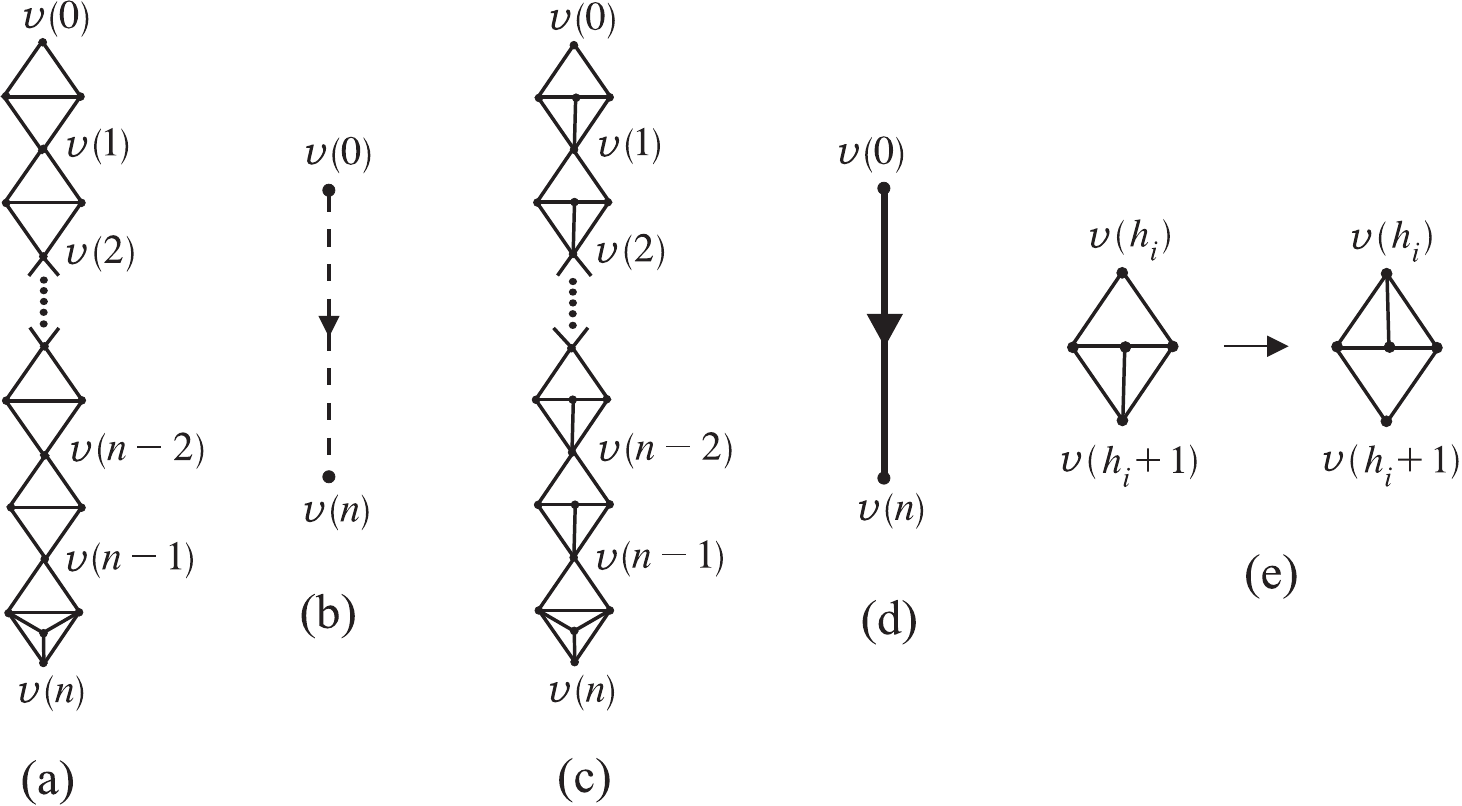} \caption{A- and B-chains.}
 \label{Fig:2.2}
  \end{figure}

By a \emph{subchain} of a chain shown in
Figs.~\ref{Fig:2.2}(a) and (c) we mean a subgraph of
the chain consisting of links incident with $v(i)$
and $v(i+1)$ for all $i=m,m+1,\ldots,m'-1$ for some
$0\leq m<m'\leq n$. We say that the subchain \emph{connects}
the vertices $v(m)$ and $v(m')$.

A \emph{chain graph} is any graph obtained from
$K_{3,3}$ by replacing three of its edges incident with
the same vertex by A-chains and three edges incident
with another vertex in the same chromatic class by B-chains, where the chains can have
arbitrary lengths $\geq 2$. These changes are to be
made as shown in Fig.~\ref{Fig:2.3}(a). The vertices
$\Omega(1)$, $\Omega (2)$, and $\Omega (3)$ are the
\emph{base vertices} of the chain graph. The edges
joining the vertex $\Omega$ to the base vertices are
called the $\Omega$-\emph{edges}.

\begin{figure}[t!]
\centering
\includegraphics[width=0.7\textwidth]
{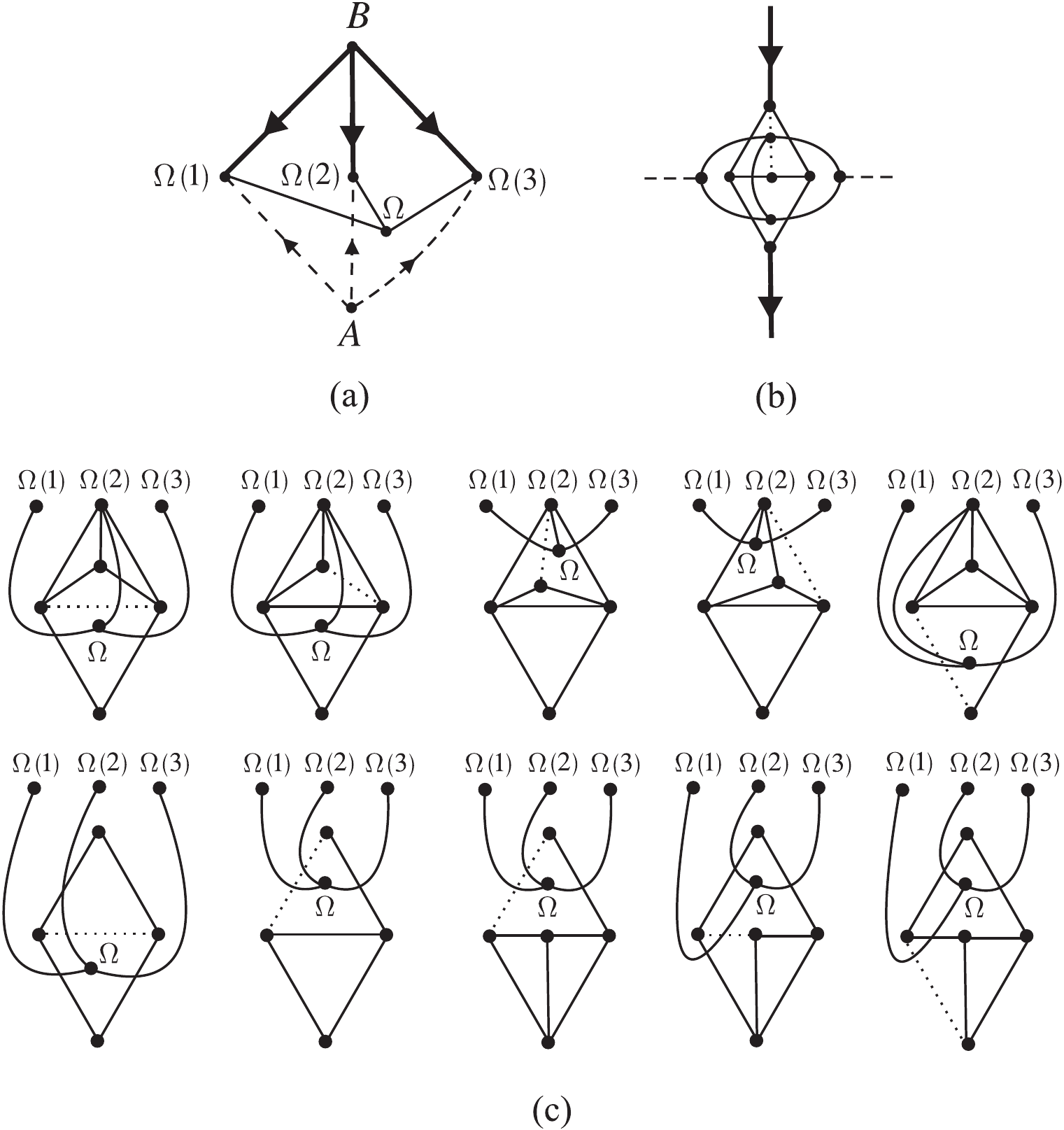} \caption{A chain graph $G$ and 1-planarity of the graph $G-e$.}
 \label{Fig:2.3}
  \end{figure}

We will show that every chain graph is an MN-graph.

\begin{lem}
\label{lem:1.1}
Let\/ $G$ be a chain graph and $e\in E(G)$.
Then $G-e$ is\/ $1$-planar.
\end{lem}

\begin{proof}
If $e$ is an $\Omega$-edge, then $G-e$ is planar and hence 1-planar.
Suppose now that $e$ is not an $\Omega$-edge.
By symmetry, we may assume that $e$ belongs
to an A- or B-chain incident to $\Omega(2)$. If $e$ is the ``middle''
edge of a B-link, then Fig.~\ref{Fig:2.3}(b) shows that the corresponding
B-chain can be crossed by an A-chain, and it is easy to see that this can be
made into a 1-immersion of $G-e$.
In all other cases, 1-immersions are
made by crossing the link $L$ whose edge $e$ is deleted with the edges incident
with the vertex $\Omega$. The upper row in Fig.~\ref{Fig:2.3}(c)
shows the cases when $L$ is a base link. The lower row covers the cases
when $L$ is an A-link or a B-link.
The edge $e$ is shown in all cases as the dotted edge.
\end{proof}

Our next goal is to show that chain graphs are not 1-planar.
In what follows we let $G$ be a chain graph and $\varphi$
a (hypothetical) 1-immersion of $G$.

\begin{lem}
\label{lem:2.1}
Let $\varphi$ be a $1$-immersion of a chain graph\/ $G$ such that
the number of crossings in $\varphi$ is minimal among all\/ $1$-immersions of $G$.
If\/ $L$ is a link in an A- or B-chain of\/ $G$, then no two edges of\/ $L$
cross each other in $\varphi$.
\end{lem}

\begin{proof}
The first thing to observe is that whenever edges $ab$ and $cd$ cross,
there is a disk $D$ having $a,c,b,d$ on its boundary, and $D$ contains these
two edges but no other points of $G$. In 1-immersions with minimum number of
crossings this implies that no other edges between the vertices $a,c,b,d$
are crossed. Similarly, if $L=L(z,\bar z)$ is a link in a chain, and an edge
incident with $z$ crosses an edge incident with $\bar z$, the whole link $L$
can be drawn in $D$ without making any crossings. This shows that the only
possible cases for a crossing of two edges $e,f$ in $L$ are the following ones,
where we take the notation from Figure~\ref{Fig:2.1} and we let $u$ be
the vertex of $L$ that is not labeled in the figure:

(a) $L$ is a B-link and $e=zv$, $f=uw$.

(b) $L$ is a B-link and $e=\bar zv$, $f=uw$.

(c) $L$ is a base link and $e=zv$, $f=uw$.

(d) $L$ is a base link and $e=vw$, $f=\bar zu$.

(e) $L$ is a base link and $e=\bar zw$, $f=vu$.

\smallskip
\noindent
Let $D$ be a disk as discussed above corresponding to the crossing of $e$ and $f$.
In cases (b), (d) and (e), the 3-valent vertex $u$ has all neigbors on the boundary of $D$,
so the crossing between $e$ and $f$ can be eliminated by moving $u$ inside $D$
onto the other side of the edge $e$ (see Fig.~\ref{Fig:2.3E}(a)).

\begin{figure}[t!]
\centering
\includegraphics[width=0.5\textwidth]
{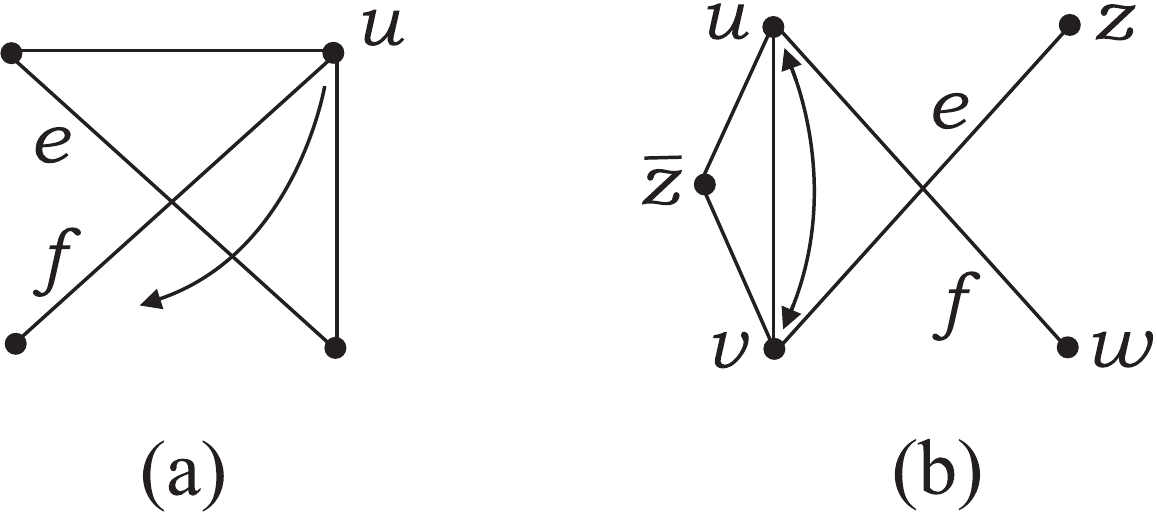} \caption{Eliminating the crossing between the edges $e$ and $f$.}
 \label{Fig:2.3E}
  \end{figure}

It remains to consider cases (a) and (c). Observe that the boundary of $D$
contains vertices $z,u,v,w$ in this order and that $u$ and $v$ both have
precisely one additional neighbor $\bar z$ outside of $D$. Therefore, we
can turn $\varphi$ into another 1-immersion of $G$ by swapping $u$ and $v$
and only redraw the edges inside $D$ (see Fig.~\ref{Fig:2.3E}(b)). However, this eliminates the crossing in $D$ and yields a 1-immersion with fewer crossings, a contradiction.
\end{proof}

\begin{lem}
\label{lem:2.2}
Let $\varphi$ be a $1$-immersion of a chain graph\/ $G$ such that
the number of crossings in $\varphi$ is minimal among all\/ $1$-immersions of $G$.
If\/ $\Pi$ and\/ $\Pi'$ are nonadjacent A- and B-chains, respectively, then in
$\varphi$ the following holds for every 3-cycle $C$ of\/ $\Pi$:
\begin{itemize}
    \item [\rm (i)] The core vertices of\/ $\Pi'$ either
    all lie inside or all lie outside $C$.
    \item [\rm (ii)] If all core vertices of\/ $\Pi'$ lie
    inside $($resp. outside$)$ $C$, then at most one
    vertex of\/ $\Pi'$ lies outside $($resp.\ inside$)$ $C$.
\end{itemize}
\end{lem}

\begin{proof}
First we show (i). If $C$ does not contain the vertex $A$, then every
two core vertices of $\Pi'$ are connected by four
edge-disjoint paths not passing through the vertices of $C$,
hence (i) holds for $C$.

Suppose now that $C$ contains the vertex $A$ and that
core vertices of $\Pi'$ lie inside and outside $C$.
Then there is a link $L(z,\overline{z})$ of $\Pi'$
such that the vertex $z$ lies inside and the vertex
$\overline{z}$ lies outside $C$. We may assume
without loss of generality that $\Pi$ and $\Pi'$ are
incident to the base vertices $\Omega(1)$ and
$\Omega(2)$, respectively, and (taking (A) into account) that the vertex $z$
(if $z\neq B$) separates the vertices $B$ and
$\overline{z}$ in $\Pi'$ (see Fig.~\ref{Fig:2.4}(a), where in
$L(z,\overline{z})$ the dotted line indicates that
the link has either edge $\varepsilon z$ or
$\varepsilon \overline{z}$; also if $\bar z = \Omega(2)$,
then the link indicated in Fig.~\ref{Fig:2.4}(a) is a base link).

The 3-cycle $C$ crosses at least two edges of
$L(z,\overline{z})$. The vertex $z$ (resp.
$\overline{z}$) is connected to each of the vertices
$\Omega(1)$ and $\Omega(3)$ (resp. to the vertex
$\Omega(2)$) by two edge-disjoint paths not passing through
$V(C)$ or through the noncore vertices of
$L(z,\overline{z})$. Hence, $\Omega(1)$
and $\Omega(3)$ lie inside $C$ (resp. $\Omega(2)$
lies outside $C$). It follows that the vertex
$\Omega$ lies inside $C$ and the edge $(\Omega,\Omega(2))$
is the third edge that crosses $C$. We conclude that $C$
crosses exactly two edges of $L(z,\overline{z})$ and
the two edges separate $z$ from $\overline{z}$ in
$L(z,\overline{z})$. Thus, the two edges are the
core-adjacent edges of the free cycle of
$L(z,\overline{z})$.
Hence, in $\varphi$, the link
$L(z,\overline{z})$ is 1-immersed as shown
in Fig.~\ref{Fig:2.4}(b), where the dotted edges indicate
alternative possibilities for the position of $z$ (at top) or $\overline{z}$ (at bottom).

\begin{figure}[t!]
\centering
\includegraphics[width=0.7\textwidth]
{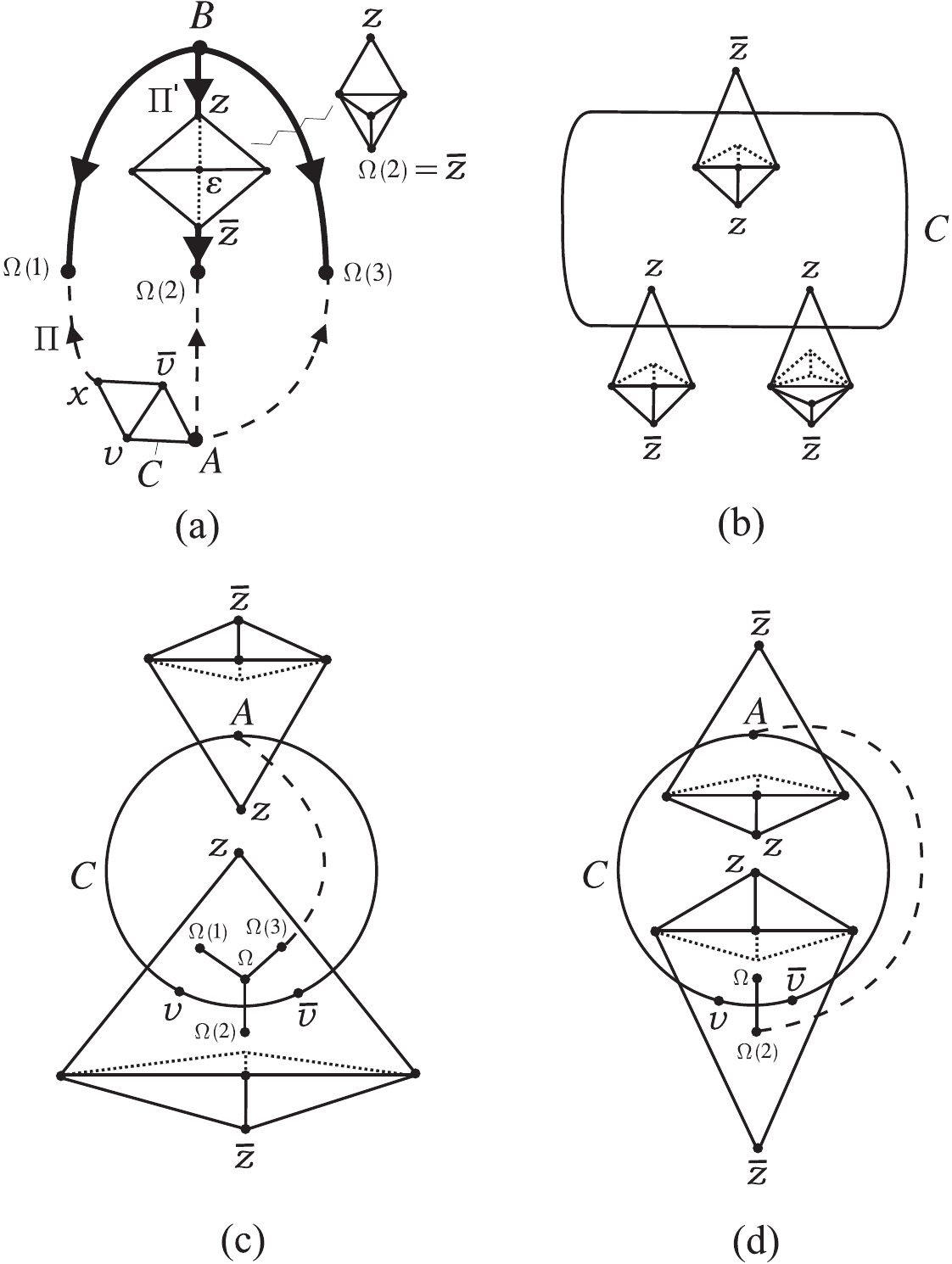} \caption{Cases in the proof of Lemma \ref{lem:2.2}.}
 \label{Fig:2.4}
  \end{figure}

Let $v,\bar v$ be the vertices of $C$ different from $A$
and let $x$ be the fourth vertex of the link containing $C$.
The vertex $x$ is connected to $\Omega(1)$ by two
edge-disjoint paths not passing through the vertices of $C$,
hence $x$ lies inside $C$. At most two vertices of
$C$ lie inside the free cycle of $L(z,\overline{z})$. Suppose exactly one of $v$ and $\overline{v}$ is inside the free cycle. If we are in the case of the bottom of Fig.~\ref{Fig:2.4}(b), then the path $vx\overline{v}$ can not lie inside $C$, a contradiction. In the case of the top of Fig.~\ref{Fig:2.4}(b), if the path $vx\overline{v}$ lies inside $C$, then $x$ must lie
inside a 3-cycle $Q$ of $L(z,\overline{z})$ incident
to $z$, whereas $A$ lies outside $Q$, a
contradiction, since $Q$ is not incident to B and in
$G$ there are two edge-disjoint paths connecting $x$
to $A$ and not passing through the vertices $v$,
$\overline{v}$, and the vertices of $Q$.

If either both $v$ and $\overline{v}$ or none of them
lie inside the free 4-cycle, then in the case of
Fig.~\ref{Fig:2.4}(c) (resp. (d)), where we depict
the two possible placements of the nonbase link
$L(z,\overline{z})$, there are two edge-disjoint
paths of $G$ connecting $A$ and $\Omega(3)$ (resp. $A$ and
$\Omega(2)$) and not passing through $z$ (resp.
$\overline{z}$), a contradiction. Reasoning exactly
in the same way, we also obtain a contradiction
when $L(z,\overline{z})$ is a base link.

Now we prove (ii). By (A), we may assume that all core
vertices of $\Pi'$ lie inside $C$. By inspecting Fig.~\ref{Fig:2.1},
it is easy to check
that for every link $L$, for every set $W$ of noncore
vertices of $L$ such that $|W|\geq 2$, there are at least four edges
joining $W$ with of $V(L)\setminus W$. Hence, if at least two noncore
vertices belonging to the same link of $\Pi'$ lie outside $C$,
then at least four edges join them with the vertices of $\Pi'$
lying inside $C$, a contradiction. Every noncore
vertex of $\Pi'$ has valence at least 3. Hence if exactly $n$ ($n\geq 2$)
noncore vertices of $\Pi'$ lie outside $C$ and if they all belong to
different links, then at least $3n\geq 6$ edges join them with
the vertices of $\Pi'$ lying inside $C$, a contradiction.
\end{proof}

Two chains \emph{cross} if an edge of one crosses an edge of the other.

\begin{lem}
\label{lem:2.3}
Let $\varphi$ be a $1$-immersion of a chain graph\/ $G$ such that
the number of crossings in $\varphi$ is minimal among all\/ $1$-immersions of $G$.
If\/ $\Pi$ and\/ $\Pi'$ are nonadjacent A- and B-chains,
respectively, then $\Pi$ does not cross $\Pi'$ in
$\varphi$.
\end{lem}

\begin{proof}
Suppose, for a contradiction, that $\Pi$ crosses $\Pi'$.
Then an edge  of a link $L(z,\overline{z})$ of $\Pi'$ crosses
a 3-cycle $C=xv\bar v$ of a link $L$ of $\Pi$.
Let $\overline{C}=\bar x v \bar v$ be a 3-cycle that is adjacent
to $C$ in $L$. (If $L$ is not a base link, then $L=C\cup\overline{C}$
and $x,\bar x$ are the core vertices of $L$.)
By Lemma~\ref{lem:2.2}, we may assume that all core vertices
of $\Pi'$ lie outside $C$ and that exactly one vertex
$u$ of $\Pi'$ lies inside $C$. The vertex $u$ is
3-valent and is not a core vertex. The three edges incident with $u$ cross all three edges of $C$, hence $\overline{C}$ does not cross $C$. If $\overline{C}$ lies inside $C$, then one of the three edges incident with $u$ crosses $C$ and $\overline{C}$, a contradiction. If $C$ lies inside $\overline{C}$, then we consider the plane as the complex plane and
apply the M\"obius transformation $f(z)=1/(z-a)$ with the point $a$ taken inside $\overline{C}$ but outside $C$. This yields a 1-immersion of $G$ such that
\vspace{1mm}

(a) $C$ lies outside $\overline{C}$, $\overline{C}$ lies outside $C$, and exactly one vertex $u$ of $\Pi'$ lies inside $C$.

\vspace{1mm}

\noindent Therefore, we may assume that in $\varphi$ we have (a). Since the three edges incident with $u$ cross all three edges of $C$, at least two vertices of $\Pi'$ lie outside $\overline{C}$, hence, by Lemma~\ref{lem:2.2}, all core vertices of $\Pi'$ lie outside $C$ and $\overline{C}$. Since the edge $v\bar v$ in $C\cap\overline{C}$ is crossed by
an edge of $L(z,\bar z)$, also $\overline C$ contains precisely one
vertex $u'$ of $L(z,\bar z)$ and $u'$ has degree 3 and is not a core
vertex.

\begin{figure}
\centering
\includegraphics[width=0.7\textwidth]
{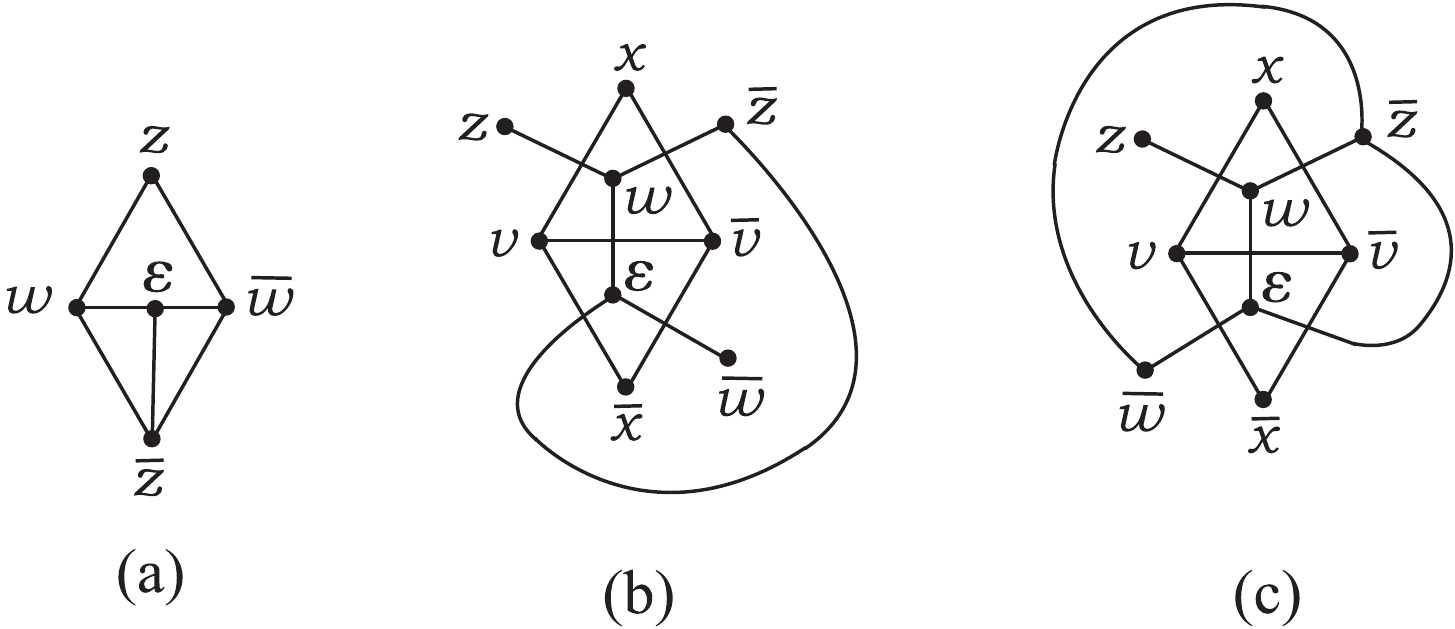} \caption{Cases in the proof of Lemma \ref{lem:2.3}.}
 \label{Fig:2.5}
  \end{figure}

Adjacent trivalent vertices $u,u'$ cannot be contained in a base
link. Therefore, $L(z,\overline{z})$ is not a base link.
Let $L(z,\overline{z})$ be depicted as shown in Fig.~\ref{Fig:2.5}(a).
Because of symmetry, we may assume that $u=w$ and
$u'=\varepsilon$, and that the crossings are as shown in
Fig.~\ref{Fig:2.5}(b) and (c).

 In the case of Fig.~\ref{Fig:2.5}(b) the adjacent vertices $z$ and $\overline{w}$
 of $L(z,\overline{z})$ are separated by the 3-cycle $\overline{z} w \varepsilon$,
 whose edges are crossed by three edges different from the edge
 $z\overline{w}$, a contradiction.

 Consider the case in Fig.~\ref{Fig:2.5}(c). If $x$ and $\overline{x}$
 are core vertices of $\Pi$, then they are separated by the 3-cycle
 $\overline{z} \overline{w} \varepsilon$ of $\Pi'$, a contradiction, since there are 4 edge-disjoint paths between $x$ and $\overline{x}$ that avoid this 3-cycle (the 3-cycle does not contain the vertex $B$).

 Suppose that $x$ and $\overline{x}$ are not two core vertices. This is possible only when $L$ is a base link. The 3-cycle $\overline{z} w \varepsilon$ of $L(z,\overline{z})$ crosses the three edges joining the vertex $\overline{v}$ of $L$ with three vertices $x$, $v$, and $\overline{x}$ of $L$. The fifth vertex of $L$ is adjacent to at least one vertex from $x$, $v$, and $\overline{x}$, hence it lies outside $\overline{z} w \varepsilon$ and is not adjacent to $\overline{v}$. Thus, $\overline{v}$ has valence 3 in $L$. If $\overline{v}$ is a core vertex, then, since the 3-cycle $\overline{z} w \varepsilon$ does not contain the vertex $B$, $\overline{v}$ is connected to one of the vertices $x$, $v$, and $\overline{x}$ by a path passing through $B$ and not passing through the vertices of $\overline{z} w \varepsilon$, a contradiction. Hence $\overline{v}$ is not a core vertex. The link $L$ has
 exactly one noncore vertex $\overline{v}$ of valence 3 and the vertices $x$
 and $\overline{x}$ are adjacent. Hence the 3-cycle $x \overline{v} \overline{x}$
 separates two core vertices $z$ and $\overline{z}$ of $\Pi'$, a contradiction, since $z$ and $\overline{z}$ are connected by a path not passing through the vertices of the 3-cycle $x \overline{v} \overline{x}$.
\end{proof}

\begin{thm}
\label{thm:2.1}
Every chain graph is an MN-graph.
\end{thm}

\noindent
\emph{Proof}. Let $G$ be a chain graph. By Lemma \ref{lem:1.1}, it suffices
to prove that $G$ is not 1-planar. Consider, for a contradiction, a
1-immersion $\varphi$ of $G$ and suppose that $\varphi$ has minimum
number of crossings among all 1-immersions of $G$.

We know by Lemma~\ref{lem:2.3} that non-adjacent chains
do not cross each other. In the sequel we will consider
possible ways that the $\Omega$-edges cross with one of the
chains. Let us first show that such a crossing is inevitable.

\begin{clm}
\label{clm:2.0}
At least one of the chains contains a link $L=L(x,y)$ such that
every $(x,y)$-path in $L$ is crossed by an $\Omega$-edge.
\end{clm}

\begin{proof}
Suppose that for every link $L=L(x,y)$, an $(x,y)$-path in $L$ is
not crossed by any $\Omega$-edge. Then every chain contains a path
joining its end vertices that is not crossed by the $\Omega$-edges.
All six such paths plus the $\Omega$-edges form a subgraph of $G$
that is homeomorphic to $K_{3,3}$. By Lemma~\ref{lem:2.3},
the only crossings between subdivided edges of this $K_{3,3}$-subgraph
are among adjacent paths. However, it is easy to eliminate crossings
between adjacent paths and obtain an embedding of $K_{3,3}$ in
the plane. This contradiction completes the proof of the claim.
\end{proof}

Let $L=L(x,y)$ be a link in an A- or B-chain $\Pi$ whose $(x,y)$-paths are
all crossed by the $\Omega$-edges. We may assume that $L$ is contained in
a chain connecting the vertex $\Omega(1)$ with $A$ or $B$ and that $x$ separates $y$ and $\Omega(1)$ in $\Pi$.
By Lemma \ref{lem:2.1}, the induced 1-immersion of $L$ is an embedding. The vertex $\Omega$ lies inside a face of $L$ and all $\Omega$-edges that cross $L$ cross the edges of the boundary of the face. Considering the possible embeddings of $L$, it is easy to see that all $(x,y)$-paths are crossed by $\Omega$-edges only in the case when $\Omega$ lies inside a face of $L$ whose boundary contains two core-adjacent edges of a free cycle $C$ of $L$, and two $\Omega$-edges cross the two core-adjacent edges. By (A), we may assume that $\Omega$ lies inside $C$.

If $C$ is a $k$-cycle, $k\in\{3,4\}$, then $L$ has another cycle $C'$ that shares with $C$ exactly $k-2$ edges and contains a core vertex not belonging to $C$. If $C$ lies inside $C'$, then we consider the plane as the complex plane and
apply the M\"obius transformation $f(z)=1/(z-a)$ with the point $a$ taken inside $C'$ but outside $C$. This yields a 1-immersion of $G$ such that $C$ does not lie inside $C'$ and $\Omega$ lies inside $C$. Hence, we may assume that $C$ does not lie inside $C'$, that $\Omega$ lies inside $C$ and two $\Omega$-edges $h$ and $h'$ cross two core-adjacent edges of $C$. Note that any two among the vertices $A,B,\Omega(2),\Omega(3)$ are joined
by four edge-disjoint paths not using any edges in the chain $\Pi$ containing $L$. Therefore, these four vertices of $G$ are all immersed in the same face of $L$.

Let the $\Omega$-edges $h$ and $h'$ join the vertex $\Omega$ with basic vertices $\Omega(i)$ and $\Omega(j)$, respectively. If the third basic vertex $\Omega(\ell)$ is $\Omega(2)$ or $\Omega(3)$, then $\Omega(2)$ and $\Omega(3)$ lie inside different faces of $L$, a contradiction, hence $\Omega(\ell)=\Omega(1)$.

The vertex $\Omega(1)$ is connected to one of the vertices $A$ and $B$ by two edge-disjoint paths, not passing the vertices of $C$. Hence, if $C$ is a 3-cycle, then $\Omega(1)$ is not inside $C$.

Now the embeddings of possible links $L$ (so that we can join the vertices $\Omega$ and $\Omega(1)$ by an edge not violating the 1-planarity) are shown in Figure \ref{Fig:2.6}.

\begin{figure}
\centering
\includegraphics[width=0.8\textwidth]%
{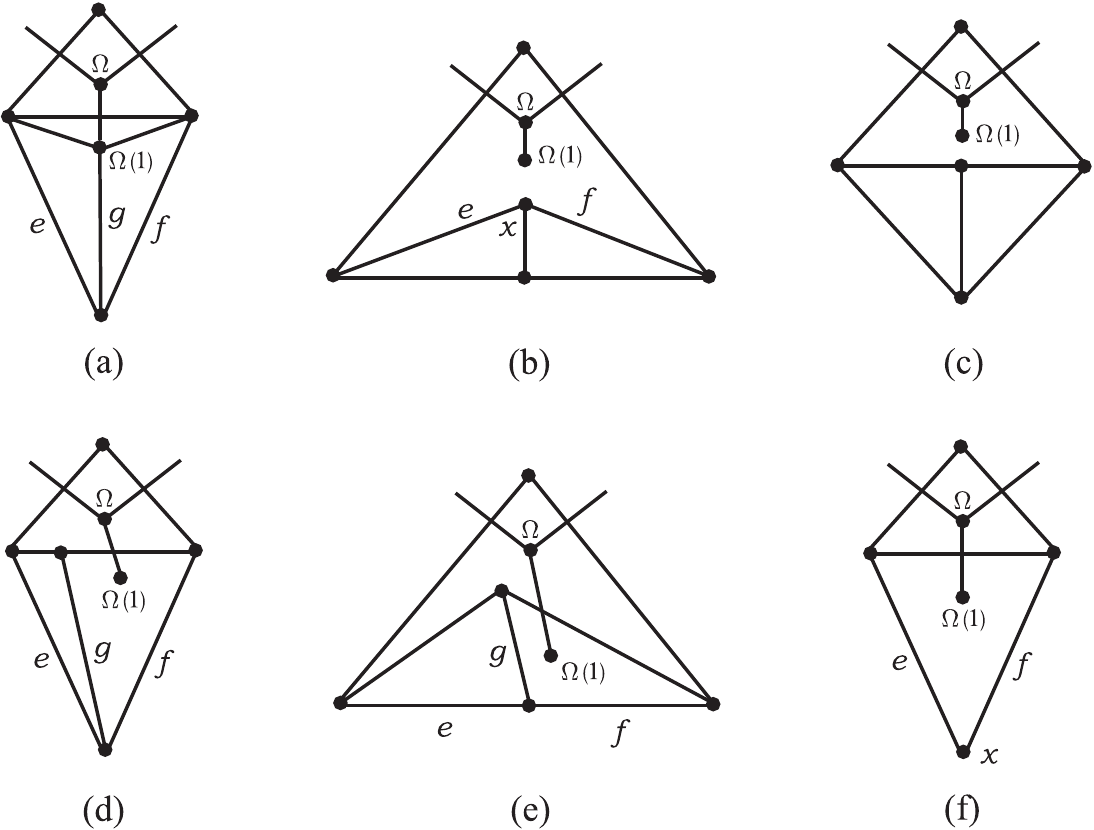}
 \caption{The $\Omega$-edges crossing a link}
 \label{Fig:2.6}
\end{figure}

Let us now consider particular cases (a)--(f) of Figure \ref{Fig:2.6}.

(a): In this case, $L$ is a base link and $x=\Omega(1)$. Consider two edge-disjoint paths in a chain $\overline{\Pi}$ joining $\Omega(1)$ with the vertex $A$ or $B$ which is not incident with $\Pi$. These paths must cross the edges $e$ and $f$ indicated in the figure. Let $a$ be the edge crossing $e$ and $b$ be the edge crossing $f$. It is easy to see that $a$ and $b$ cannot be both incident with $\Omega(1)$ since $\Omega(1)$ is incident with three edges of the base link in the chain $\overline{\Pi}$. The 1-planarity implies that the edges $a,b$ and the vertex $\Omega(1)$
separate the graph $G$. Therefore, $a$ and $b$ are core-adjacent edges of a link in the chain $\overline{\Pi}$. If the edge $g$ (shown in the figure) is crossed by an edge $c$ of $\overline{\Pi}$, then also $a,c$ and $\Omega(1)$ separate the graph. Thus $a,c$ would be core-adjacent edges in a link in $\overline{\Pi}$ as well, a contradiction. But if $g$ is not crossed and $a$ is not incident with $\Omega(1)$, then the edge $a$ and the vertex $\Omega(1)$ separate $G$, a contradiction. If $a$ is incident with $\Omega(1)$, then $b$ is not (as proved above). Now we get a contradiction by considering the separation of $G$ by the edge
$b$ and the vertex $\Omega(1)$.

(b): In this case, $L$ is a B-link, the cycle $C'$ lies inside $C$ and $\Omega(1)$ is inside $C$ but outside $C'$ (see Fig.~\ref{Fig:2.6} (b)). In the figure, the vertex labeled $x$ is actually $x$ and not $y$ because the vertex is not $B$ and is connected to $\Omega(1)$ by two edge-disjoint paths not passing through the vertices of $C$, and $y$ is connected to the vertex $B$ (lying outside $C$) by two edge-disjoint paths not passing through the vertices of $C$. Again, consider two edge-disjoint paths in the A-chain $\overline{\Pi}$ joining $\Omega(1)$ with the vertex $A$. Their edges $a$ and $b$ (say) must cross the edges $e$ and $f$, respectively. As in the proof of case (a), we see that the edges $a$, $b$ and the vertex $\Omega(1)$ separate $G$, thus $a$ and $b$ are core-adjacent edges of a free cycle of a link of $\overline{\Pi}$. This free cycle has length 3 and separates $x$ from $\Omega(1)$, a contradiction, since there is a B-subchain from $x$ to $\Omega(1)$ disjoint from the free 3-cycle and not containing the edges of $L$ incident with $x$.

(c): In this case, $L$ is a B-link, the cycle $C'$ lies outside $C$ and $\Omega(1)$ is inside $C$. The vertex $\Omega(1)$ is connected by 4 edge-disjoint paths with vertices $x$ and $A$ such that the paths do not pass through noncore vertices of $L$, a contradiction.

(d) and (e): In these cases, $L$ is a B-link and $\Omega(1)$ lies inside a 3-cycle of $L$. Two edge-disjoint paths from $\Omega(1)$ to $A$ cross the edges $e$ and $f$ of $L$. One of the paths also crosses the edge $g$. Then the edges crossing $e$ and $g$ separate $G$, a contradiction, since $G$ is 3-edge-connected.

(f): In this case, $L$ is an A-link and the chain $\Pi$ is joining $\Omega(1)$
with the vertex $A$. Again, consider two edge-disjoint paths in the B-chain $\overline{\Pi}$ joining $\Omega(1)$ with $B$. Their edges $a$ and $b$ (say) must cross
the edges $e$ and $f$, respectively.
As in the proof of case (a), we see that the edges $a,b$ and the vertex
$\Omega(1)$ separate $G$, thus $a,b$ are core-adjacent edges
in a link $L'=L(z,z')$ in $\overline{\Pi}$. Let $z$ be the core vertex of $L'$ incident
with the edges $a=zp$ and $b=zq$. Note that in $L'$, vertices $p,q$ have
two common neighbors, a vertex $r$ of degree 3 in $G$ and the core vertex $z'$,
and that $z'$ is adjacent to $r$. Inside $L$ we have the subchain $\Pi'$ of the A-chain $\Pi$ connecting $x$ with $\Omega(1)$. Hence, the free cycle of $L'$ containing the edges $a$ and $b$ must have length at least 4 (that is, $L'$ is not a base link) and  $z$ is immersed outside $L$,
while $p,q,r,z'$ are inside. The subchain $\Pi'$ has two
edge-disjoint paths that are crossed by the paths $prq$ and $pz'q$.
Each of the paths $prq$ and $pz'q$ crosses core-adjacent edges in $\Pi$ since the crossed edges
and $\Omega(1)$ separate $G$. Thus, they enter faces of these links
that are bounded by free cycles of the links. However, the edge $rz'$
would need to cross one edge of each of these two free cycles, hence the two free cycles are adjacent, that is, they are the two free cycles of an A-link of $\Pi$. Then the vertex $z'$ lies inside one of the two free cycles and the free cycle separates $z'$ from $\Omega(1)$, a contradiction, since $z'$ is connected with $\Omega(1)$ by a B-subchain. \hfill$\square$

\bigskip

The following theorem shows how chain graphs can be used to
construct exponentially many nonisomorphic MN-graphs of order $n$.

\begin{thm}\label{thm:2}
For every integer $n\geq 63$, there are at least
$2^{(n-54)/4}$ nonisomorphic MN-graphs of order $n$.
\end{thm}

\begin{proof}
The A-chain of length $t$ has $3t+2$
vertices and a B-chain of length $t$ has $4t+1$
vertices. Consider a chain graph whose three A-chains
have length 2, 2, and $\ell\geq 2$, respectively, and
whose B-chains have length 2, 3, and $t\geq 4$,
respectively. The graph has $35+3\ell+4t$ vertices.
Applying the modification shown in
Fig.~\ref{Fig:2.2}(e) to links of the two B-chains of the
graph which have length at least 3, we obtain $2^{t-1}$ nonisomorphic chain graphs
of order $35+3\ell+4t$, where $\ell\geq 2$ and $t\geq
4$. We claim that for every integer $n\geq 63$, there
are integers $2\leq \ell\leq 5$ and $t\geq 4$ such
that $n=35+3\ell+4t$. Indeed, if $n\equiv
0,1,2,3\pmod{4}$, put $\ell=3,2,5,4$, respectively.
If $n=35+3\ell+4t$, where $2\leq \ell\leq 5$, then
$t\geq n/4-50/4$. Hence, there are at
least $2^{\frac{n}{4}-\frac{54}{4}}$ nonisomorphic
chain graphs of order $n\geq 63$. Since every chain
graph is a MN-graph, the theorem follows.
\end{proof}

\section{PN-graphs} \label{Sec:3}

By a \emph{proper} 1-immersion of a graph we mean a
1-immersion with at least one crossing point.
Let us recall that a PN-\emph{graph} is a planar graph
that does not have proper 1-immersions.
In this section we describe a class of PN-graphs and
construct some graphs of the class. They will
be used in Section~\ref{Sec:4} to construct MN-graphs.

For
every cycle $C$ of $G$, denote by $N(C)$ the set
of all vertices of the graph not belonging to $C$ but
adjacent to $C$. Two disjoint edges $vw$ and $v'w'$ of a graph $G$
are \emph{paired} if the four vertices $v,w,v',w'$
are all four vertices of two adjacent 3-cycles (two cycles are \emph{adjacent} if they share an edge).

Following Tutte, we call a cycle $C$ of a graph $G$ \emph{peripheral}
if it is an induced cycle in $G$ and $G-V(C)$ is connected.
If $G$ is 3-connected and planar, then the face boundaries in its
(combinatorially unique) embedding in the plane are
precisely its peripheral cycles.

\begin{thm}\label{thm:3.1}
Suppose that a 3-connected planar graph $G$ satisfies
the following conditions:
\begin{itemize}
\myitemsep
    \item [{\rm (C1)}] Every vertex has degree at least\/ $4$
    and at most\/ $6$.
    \item [{\rm (C2)}] Every edge belongs to at least one
    $3$-cycle.
    \item [{\rm (C3)}] Every $3$-cycle is peripheral $($in other words, there are no separating 3-cycles$)$.
    \item [{\rm (C4)}] Every $3$-cycle is adjacent to at most
    one other $3$-cycle.
    \item [{\rm (C5)}] No vertex belongs to three mutually
    edge-disjoint $3$-cycles.
    \item [{\rm (C6)}] Every $4$-cycle is either peripheral or is
    the boundary of two adjacent triangular faces $($this means that there are no separating 4-cycles$)$.
    \item [{\rm (C7)}] For every $3$-cycle $C$, any two
    vertices of $V(G)\setminus (V(C)\cup N(C))$ are
    connected by four edge-disjoint paths not passing through
    the vertices of $C$.
    \item [{\rm (C8)}] If an edge $vw$ of a
    nontriangular peripheral cycle $C$ is paired with an
    edge $v'w'$ of a nontriangular peripheral cycle
    $C'$, then:
    \begin{enumerate}\myitemsep
       \item [{\rm (i)\,}]
        $C$ and $C'$ have no vertices in common;
        \item [{\rm (ii)}]
        any two vertices $a$ and $a'$ of
        $C$ and $C'$, respectively, such that
        $\{a,a'\}\not\subseteq\{v,w,v',w'\}$ are
        non-adjacent and are not connected by a path
        $aba'$ of length\/ $2$, where $b$ does not
        belong to $C$ and $C'$.
    \end{enumerate}
    \item [{\rm (C9)}] $G$ does not contain the subgraphs
    shown in Fig.~\ref{Fig:3.1} $($in this figure, $4$-valent
    $($resp.\ $5$-valent$)$ vertices of $G$ are encircled
    $($resp.\ encircled twice$)$ and the two starred vertices can be the same vertex$)$.
\end{itemize}
Then $G$ has no proper $1$-immersion.
\end{thm}

\begin{figure}
\centering
\includegraphics[width=0.7\textwidth]
{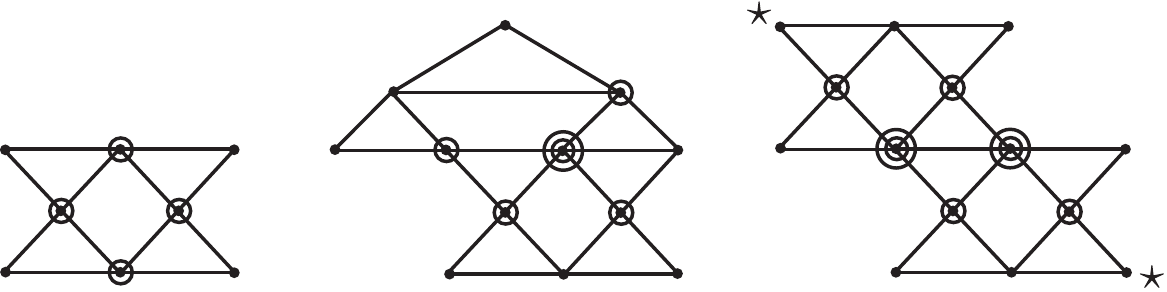} \caption{Forbidden subgraphs.}
 \label{Fig:3.1}
  \end{figure}

\begin{proof}
Denote by $f$ the unique plane
embedding of $G$. Suppose, for a contradiction, that there is a
proper 1-immersion $\varphi$ of $G$. Below we consider the 1-immersion
and show that then $G$ has a subgraph which is excluded by (C8) and (C9),
thereby obtaining a contradiction. In the figures below, the encircled
letter $f$ (resp. $\varphi$) at the top left of a figure means that the
figure shows a fragment of the plane embedding $f$
(resp.\ 1-immersion $\varphi$).

\begin{lem}\label{lem:3.1}
In $\varphi$, there is a $3$-cycle such that there is a vertex inside
and a vertex outside the cycle.
\end{lem}

\begin{proof} The 1-immersion $\varphi$ has crossing
edges $e$ and $e'$. By (C2), the crossing edges
belong to different 3-cycles. If the 3-cycles are nonadjacent, then we
apply the following obvious observation:

(a) If two nonadjacent 3-cycles $D$ and $D'$ cross each other,
    then there is a vertex of $D$ inside and outside $D'$.

If $e=xy$ and $e'=x'y'$ belong to adjacent 3-cycles $xyy'$
and $x'yy'$, respectively (see Fig.~\ref{Fig:3.2}), then, by (C4),
there are nontriangular peripheral cycles $C$ and $C'$ containing $e$
and $e'$, respectively. The cycles $C$ and $C'$ intersect  at some point
$\delta$ different
from the intersection point of edges $e$ and $e'$. By
(C8)(i), the two cycles do not have a common vertex,
hence $\delta$ is the intersection point of two edges. By (C2),
these two edges belong to some 3-cycles, $D$ and $D'$.
Property (C8)(ii) implies that $D$ and $D'$ are nonadjacent 3-cycles.
By (a), the proof is complete.
\end{proof}

\begin{figure}
\centering
\includegraphics[width=0.2\textwidth]
{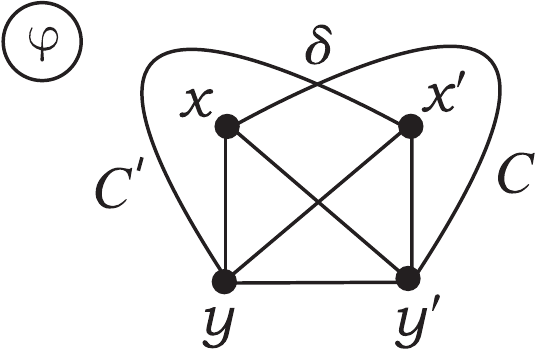} \caption{Crossing two adjacent 3-cycles.}
 \label{Fig:3.2}
  \end{figure}

\begin{lem}\label{lem:3.3}
If\/ $C=u_1u_2u_3$ is a $3$-cycle such that there is a vertex
inside and a vertex outside $C$, then there is
only one vertex inside $C$ or only one vertex outside
$C$, and this vertex belongs to $N(C)$.
\end{lem}

\begin{proof} By (C7), we may assume that all vertices of
$V(G)\setminus (V(C)\cup N(C))$ lie outside $C$.
Then there can be only vertices of $N(C)$ inside $C$.
To prove the lemma, it suffices to show the following:

(a) For every $Q\subseteq N(C)$, $|Q|\geq 2$,
at least four edges join vertices of
$Q$ to vertices in $V(G)\setminus(V(C)\cup Q)$.

By (C1), every vertex of $Q$ has valence at least 4. By (C4),
every vertex of $N(C)$ is adjacent to at most two vertices of $C$.
We claim that if a vertex $v\in N(C)$ is adjacent to two vertices
$u_1$ and $u_2$ of $C$, then $v$ is not adjacent to other vertices
of $N(C)$. Suppose, for a contradiction, that $v$ is adjacent to
a vertex $w\in N(C)$. Then, by (C4), the vertex $w$ can be adjacent
only to $u_3$ and the 4-cycle $vwu_3u_2$ is not the boundary of two adjacent 3-cycles, hence, by (C6), the 4-cycle is peripheral.
Then, by (C3), any two of the three edges $u_2v$, $u_2u_1$,
and $u_2u_3$ are two edges of a peripheral cycle, hence
$u_2$ has valence 3, contrary to (C1).

Now to prove (a) it suffices to prove the following claim:

\medskip

(b) For every $Q\subseteq N(C)$, $|Q|=1$ (resp. $|Q|\geq 2$), such that
every vertex of $Q$ is adjacent to exactly one vertex of $C$, at least
2 (resp. 4) edges join vertices of
$Q$ to vertices of $V(G)\setminus(V(C)\cup Q)$.

\medskip

The claim is obvious for $|Q|\in\{1,2\}$.
For $|Q|=3$, it suffices to show that the three vertices of $Q$ are
not all pairwise adjacent. Suppose, for a contradiction, that the vertices
$v_1$, $v_2$, and $v_3$ of $Q$ are pairwise adjacent. Then, by (C4),
the three vertices of $Q$ are not adjacent to the same vertex of $C$.
Let $v_1$ and $v_2$ be adjacent to the vertices $u_1$ and $u_2$ of $C$,
respectively. Then any two of the edges $v_1u_1$, $v_1v_2$, and
$v_1v_3$ are two edges of a 3-cycle (peripheral cycle) or a 4-cycle
which is not the boundary of two adjacent 3-cycles, so by (C6),
that 4-cycle is also peripheral. Hence, $v_1$ has valence 3,
contrary to (C1).

For $|Q|\geq 4$, it suffices to show that no vertex of $Q$ is
adjacent to three other vertices in $Q$. Suppose, for a contradiction,
that $v\in Q$ is adjacent to $w_1, w_2, w_3 \in Q$. If $v$ is adjacent to
$u_1$, then the edge $vu_1$ belongs to three cycles $D_1$, $D_2$, and $D_3$
such that for $i=1,2,3$,
the cycle $D_i$ contains edges $vu_1$ and $vw_i$, has length 3 or 4,
and if $D_i$ has length 4, then $D_i$ is not the boundary of two adjacent
3-cycles. By (C3) and (C6), these three cycles are peripheral.
This contradiction completes the proof of (b).
\end{proof}

\medskip

Suppose that a vertex $h$ belongs to two adjacent 3-cycles
$hvw$ and $hvw'$. Since $\deg(h)\ge4$, $h$ is adjacent to a vertex
$u\not\in\{v,w,w'\}$. By (C2), the edge $hu$
belongs to a 3-cycle $huu'$. By (C4),
$u'\not\in\{v,w,w',u\}$. Hence, we have the
following:

\medskip

(B) If an edge $e$ is contained in two 3-cycles of $G$, then both
endvertices of $e$ have valence at least 5.

\medskip

In the remainder of the proof of Theorem \ref{thm:3.1},
we will show that any two crossing
edges of the proper 1-immersion $\varphi$ belong to a subgraph
that is excluded by (C8) and (C9).

By Lemma~\ref{lem:3.1}, there is a 3-cycle
$C=xyz$ such that there is a vertex inside and a
vertex outside $C$. By Lemma~\ref{lem:3.3}, there is only one
vertex $v$ inside $C$ and $v$ is adjacent to $x$.

Now we show that there is a 3-cycle $B=vuw$
disjoint from $C$. Let $x,a_1,a_2,\ldots,a_t$
($t\geq 3$) be all vertices adjacent to $v$. Suppose
there is a 3-cycle $D=va_ib$, where
$b\in\{x,y,z\}$. If $b\in\{y,z\}$, then the 3-cycle
$xvb$ is adjacent to two 3-cycles $C$ and $D$,
contrary to (C4). Hence $D=va_ix$. By (C4), at
most two vertices of $\{a_1,a_2,\ldots,a_t\}$ are
adjacent to $x$. Hence, there is a vertex $a_j$ such
that a 3-cycle $B$ containing the edge $va_j$ is
disjoint from $C$.

Since there is only one vertex $v$ inside $C$,
exactly two edges of $B$ cross edges of $C$. First,
suppose that $B$ separates $x$ from $y$ and $z$
(Fig.~\ref{Fig:3.3}(a)). By (C2), the edge $xv$
belongs to a 3-cycle $R=xva$. If
$a\not\in\{y,z,u,w\}$, then two of the vertices
$y,z,u,w$ lie inside $R$ and the other two vertices
lie outside $R$ (see Fig.~\ref{Fig:3.3}(b)), contrary
to Lemma~\ref{lem:3.3}. So, we may assume, without
loss of generality, that $a=z$
(Fig.~\ref{Fig:3.3}(c)). Then the vertex $x$ belongs
to two adjacent 3-cycles, $C$ and $vxz$, hence, by (B),
$\deg(x)\ge5$. By (C4), $x$ is not adjacent to $u$ or $w$.
Since $x$ is the only vertex inside $B$, $x$ has valence at most 4,
a contradiction. Hence, $B$ can not separate $x$ from
$y$ and $z$.

\begin{figure}
 \centering
 \includegraphics[width=0.7\textwidth]{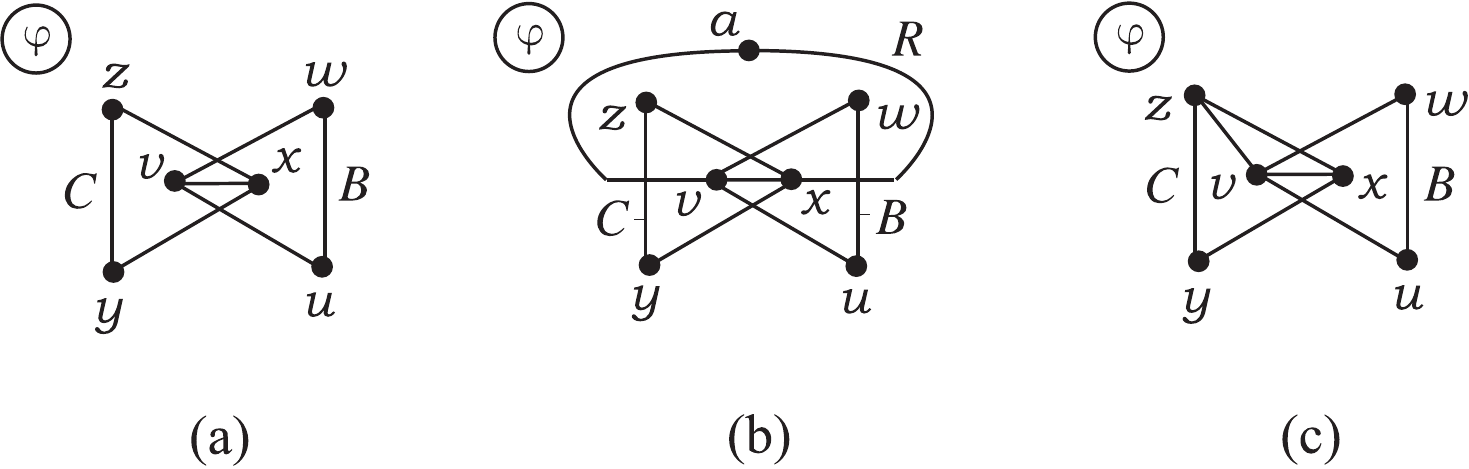}
 \caption{The 3-cycle $B$ separates $x$ from $y$ and $z$.}
 \label{Fig:3.3}
\end{figure}

Now suppose that $B$ separates the vertices $y$ and
$z$, and, without loss of generality, let $z$ lie
inside $B$ (Fig.~\ref{Fig:3.4}(a)). If $z$ is
adjacent to $v$, then, by (B), $z$ has valence at
least 5, hence $z$ is adjacent to a vertex
$b\in\{u,w\}$. But then the 3-cycle $xzv$ is
adjacent to two 3-cycles, $C$ and $vbz$, contrary
to (C4). If $z$ is adjacent to $u$ and $w$, then the
edge $xz$ belongs to three peripheral cycles, $C$,
$xvuz$, and $xvwz$, a contradiction, since
every edge belongs to at most two peripheral cycles.
Hence, since $z$ is the only vertex inside $B$, $z$
has valence 4, $z$ is not adjacent to $v$ and is
adjacent to exactly one of the vertices $u$ and $w$. The vertices $u$ and $w$ are not symmetric in Fig.~\ref{Fig:3.4}(a), so we have to consider two cases.

\emph{Case 1. The vertex $z$ is adjacent to $u$ $($Fig.~\ref{Fig:3.4}$($b$)$$)$.}

\begin{figure}
 \centering
 \includegraphics[width=0.9\textwidth]{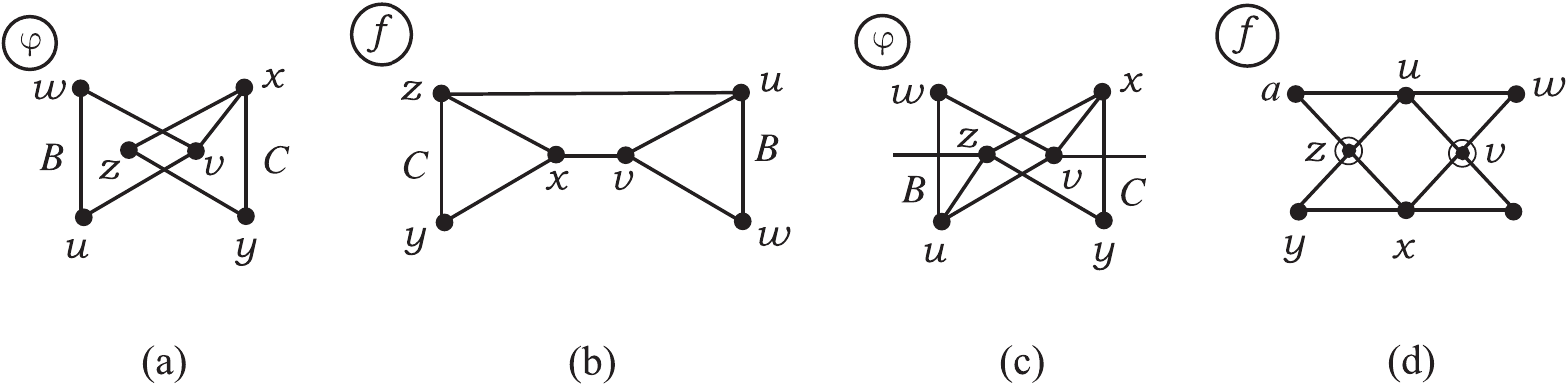}
 \caption{The 3-cycle $B$ separates $y$ and $z$.}
 \label{Fig:3.4}
\end{figure}

Consider the vertex $v$. If $v$ is adjacent to $y$,
then (see Fig.~\ref{Fig:3.4}(b)) $x$ is incident with
exactly three peripheral cycles and has valence 3, a
contradiction. Hence, $v$ is not adjacent to $y$ and
since $v$ is the only vertex inside $C$, $v$ has
valence 4 (see Fig.~\ref{Fig:3.4}(c)). By (C2),
we obtain a subgraph shown in
Fig.~\ref{Fig:3.4}(d). By (C9), at least one of the
vertices $u$ and $x$, say $u$, is not 4-valent.
Then, by (C5), at least one of the edges $au$ and $uw$ belongs
to two 3-cycles. Here we have two subcases to consider.

\emph{Subcase 1.1. The edge $uw$ belongs to two
3-cycles $uwv$ and $uwb$
$($Fig.~\ref{Fig:3.5}$($a$)$$)$.}

\begin{figure}
 \centering
 \includegraphics[width=0.9\textwidth]{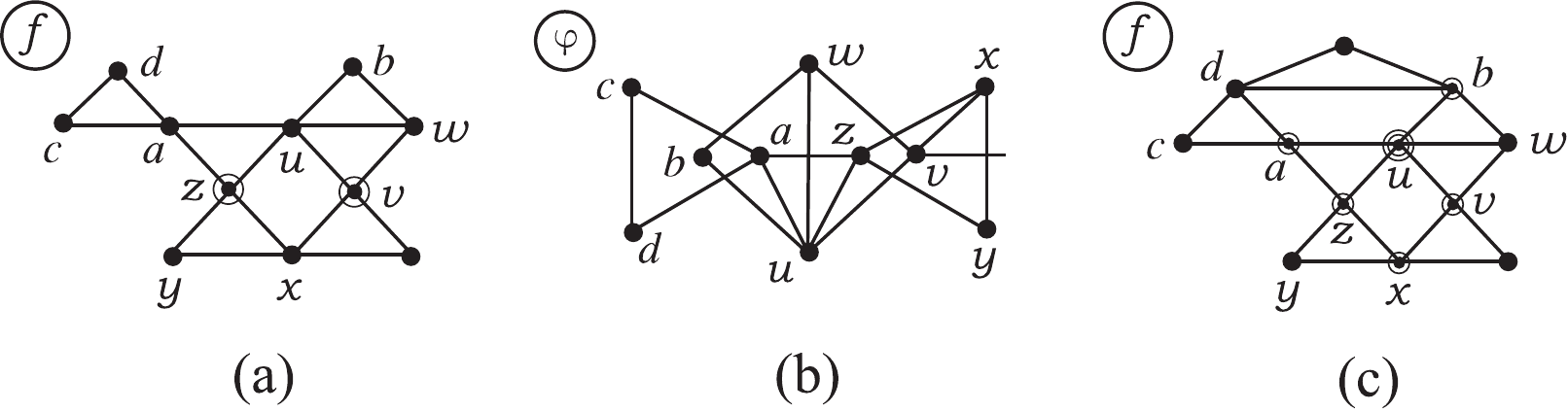}
 \caption{The edge $uw$ belongs to two 3-cycles.}
 \label{Fig:3.5}
\end{figure}

Now, $a$ is the only vertex inside the
3-cycle $uwb$ (Fig. \ref{Fig:3.5}(b)). By (C4),
$a$ is non-adjacent to $w$ and $b$, hence $a$ has
valence 4. The edges $yz$ and $za$
(see Fig. \ref{Fig:3.5}(a)) belong to a
nontriangular peripheral cycle $yzac\ldots$. The
edge $ac$ belongs to a 3-cycle $acd$ and
$b$ is the only vertex inside the 3-cycle
$acd$. The vertex $b$ is not adjacent to $c$,
since the 4-cycle $bcau$ can not be peripheral
(see Fig.~\ref{Fig:3.5}(a)). Since
$b$ has valence at least 4, $b$ is adjacent to $d$
and has valence 4. The 4-cycle $dbua$ is peripheral,
so we obtain a subgraph of $G$ shown in Fig.~\ref{Fig:3.5}(c). Note that, by (C3)--(C6), the vertex at the top of Figure~\ref{Fig:3.5}(c)
is different from all other vertices shown in the figure.
This contradicts (C9).

\emph{Subcase 1.2. The edge $au$ belongs to two
3-cycles $auz$ and $aub$ $($Fig.~\ref{Fig:3.6}$($a$)$$)$.}

Since $b$ has valence at least 4, $b$ lies outside
the 3-cycle $auz$ (Fig.~\ref{Fig:3.6}(b)). The
edges $yz$ and $za$ (resp.\ $wu$ and $ub$)
belong to a nontriangular peripheral cycle $C_1$
(resp.~$C_2$). The cycles $C_1$ and $C_2$ are paired. In
$\varphi$, the crossing point of edges $uw$ and
$az$ is an intersection point of $C_1$ and $C_2$.
The cycles $C_1$ and $C_2$ have at least one other
crossing point, denote this intersection point by
$\delta$. By (C8)(i), $C_1$ and $C_2$ are
vertex-disjoint, hence, $\delta$ is the crossing
point of $C_1$ and an edge $h_1h_2$ of $C_2$
(Fig.~\ref{Fig:3.6}(c)). The edge $h_1h_2$ is not
the edge $uw$ and belongs to a 3-cycle
$h_1h_2h_3$. If $h_3$ belongs to $C_2$, then in the embedding
$f$ the edge $h_1h_3$ is a chord of the embedded peripheral cycle $C_2$
and thus $\{h_1,h_3\}$ is a separating vertex set of $G$. But $G$
is 3-connected, a contradiction. Hence $h_3$ does not belong to $C_2$.

\begin{figure}[t!]
 \centering
 \includegraphics[width=0.95\textwidth]{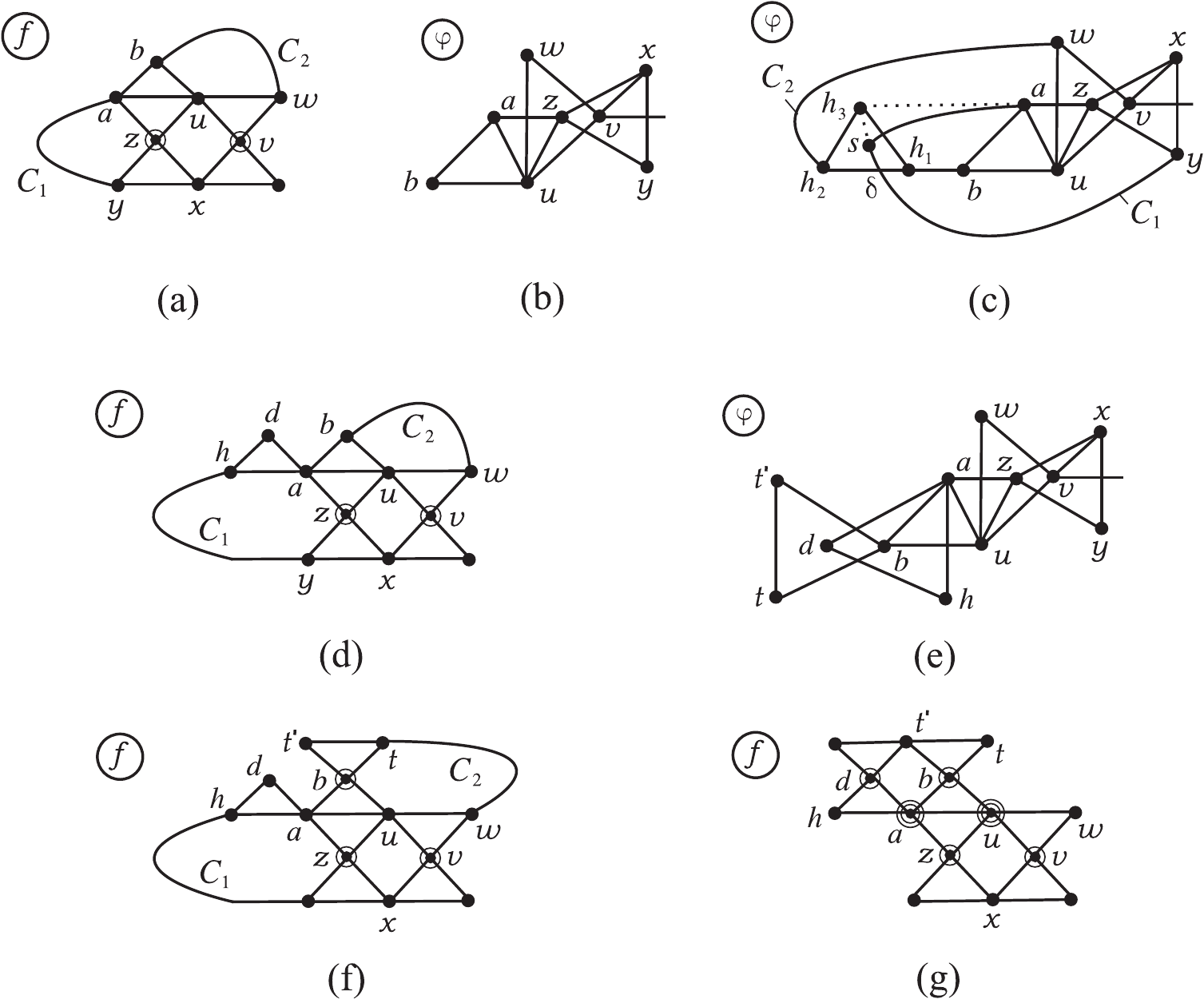}
 \caption{The edge $au$ belongs to two 3-cycles.}
 \label{Fig:3.6}
\end{figure}

Now suppose that $h_1h_2\neq bu$. We have
$h_2\not\in\{a,z,b,u\}$. By (C8)(ii), $h_3\neq a$,
$a$ is not adjacent to $h_2$ and $h_3$, and $C_1$
does not pass through $h_3$ (that is, $h_3$ does
not belong to $C_1$ and $C_2$). Hence, a vertex $s$ of $C_1$ lies
inside the 3-cycle $h_1h_2h_3$ (see Fig.~\ref{Fig:3.6}(c)). By
(B), $\deg(a)\ge5$. If $s=a$, then,
since $s$ is the only vertex inside the 3-cycle
$h_1h_2h_3$, $a$ is adjacent to at least one
of $h_2$ and $h_3$, a contradiction.
Hence, $s\neq a$. Since $z$ is the only vertex inside
the 3-cycle $B$, and the 3-cycle $h_1h_2h_3$ is not $B$
(since $h_1h_2 \neq uw$), we have $s\neq z$. Now, since
$s$ has valence at least 4, $s$ is adjacent to at
least one of the vertices $h_1$, $h_2$, and $h_3$,
contrary to (C8)(ii). Hence $h_1h_2=bu$ and $h_3=a$.

We have $C_1=yzah\ldots$ and the edge $ah$
belongs to a 3-cycle $ahd$
(see Figs.~\ref{Fig:3.6}(d) and (e)). Considering
Fig.~\ref{Fig:3.6}(e), if there is a vertex inside
the 3-cycle $auz$, the vertex has valence at most
3, a contradiction. Hence no edge crosses the edge
$au$. If $h$ lies inside the 3-cycle $aub$,
then, by (C4), $h$ is not adjacent to $b$ and
$u$, $h$ has valence at most 3, a contradiction.
Hence, $h$ lies outside the 3-cycle $aub$ and
$b$ is the only vertex inside the 3-cycle
$ahd$ (see Fig.~\ref{Fig:3.6}(e)). Note that the edge $ah$ belongs to $C_1$, hence $C_1$ does not cross both $bu$ and $ab$.

By (C4), $b$ is not adjacent to $d$ and $h$, hence
$b$ has valence 4 and belongs to a 3-cycle $btt'$
disjoint from $aub$ (Fig.~\ref{Fig:3.6}(f)),
where $C_2=wubt\ldots$. Now $d$ is the only
vertex inside the 3-cycle $btt'$
(Fig.~\ref{Fig:3.6}(e)). By (C8)(ii), $d$ is not
adjacent to $t$. Since $d$ has valence at least 4,
$d$ is adjacent to $t'$ and has valence 4. The
4-cycle $dt'ba$ is peripheral. We obtain a subgraph
of $G$ shown in Fig.~\ref{Fig:3.6}(g), contrary to
(C9). The obtained contradiction completes the proof in the Case 1.

\emph{Case 2. The vertex $z$ is adjacent to $w$.}

This case is dealt in much the same way as the Case 1. Here we describe only what figures will be in the Case 2 instead of the Figs.~\ref{Fig:3.4}-\ref{Fig:3.6} in the Case 1. We hope that the reader is familiar enough with the proof of the Case 1 to supply the missing details himself.

In Figs.~\ref{Fig:3.4}(a), \ref{Fig:3.4}(d), \ref{Fig:3.5}(a), and \ref{Fig:3.5}(c) interchange the letters $u$ and $w$. In Figs.~\ref{Fig:3.4}(c) and \ref{Fig:3.5}(b) replace the edge $uz$ by the edge $zw$. In Figs.~\ref{Fig:3.6}(a), (d), (f), and (g) interchange the letters $u$ and $w$. Figs.~\ref{Fig:3.6E}(b), (c), and (e) are replaced by the Figs.(a), (b), and (c), respectively.

\begin{figure}[t!]
 \centering
 \includegraphics[width=0.95\textwidth]{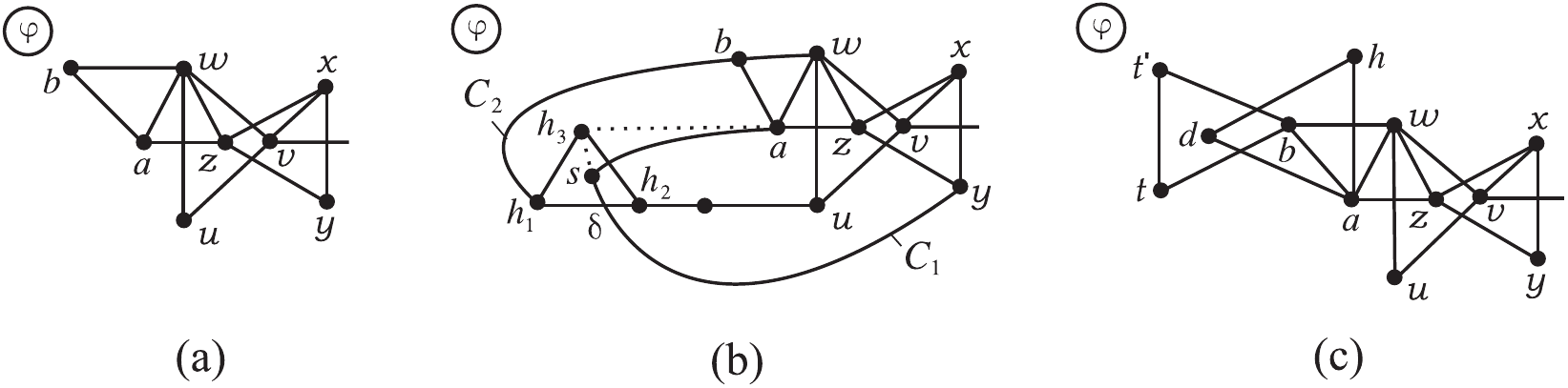}
 \caption{The edge $aw$ belongs to two 3-cycles.}
 \label{Fig:3.6E}
\end{figure}

\end{proof}

Denote by $\mathcal{A}$ the class of all 3-connected
plane graphs $G$ satisfying the conditions (C1)--(C9)
of Theorem~\ref{thm:3.1}. In what follows we show how
to construct some graphs in $\mathcal{A}$ and,
as an example, we shall give two infinite families of graphs in $\mathcal{A}$,
one of which will be used in Section~\ref{Sec:4} to construct MN-graphs.

First we describe a large family of 3-connected plane
graphs satisfying the conditions (C1)--(C6) and
(C8) of Theorem~\ref{thm:3.1}.

Given a 4-valent vertex $v$ of a 3-connected plane
graph, two peripheral cycles $C$ and $C'$ containing $v$ are
\emph{opposite} peripheral cycles at $v$ if $C$ and $C'$ have
 no edges incident with $v$ in common.

Denote by $\mathcal{H}$ the class of all 3-connected (simple) planar graphs $H$ satisfying
the following conditions (H1)--(H4):
\begin{itemize}\myitemsep
    \item [(H1)] Every vertex has valence 3 or 4.
    \item [(H2)] $H$ has no 3-cycles.
    \item [(H3)] Every 4-cycle is peripheral.
    \item [(H4)] For every 4-valent vertex $v$ and for
    any two opposite peripheral cycles $C$ and $C'$ at
    $v$, no edge joins a vertex of $C-v$ to a
    vertex of $C'-v$.
\end{itemize}

A plane graph $G$ is a \emph{medial extension} of a graph $H\in\mathcal{H}$\/ if $G$ is obtained from $H$ in the
following way. The vertex set of $G$ is the set
$\{v(e):e\in E(H)\}$. The edge set of $G$ is defined
as follows. For every 3-valent vertex $v$ of $H$, if
$e_1,e_2,e_3$ are the edges incident with $v$, then
in $G$ the vertices $v(e_1)$, $v(e_2)$, and $v(e_3)$
are pairwise adjacent (the three edges of $G$ are said to
be \emph{associated with\/} $v$).
For every 4-valent vertex $w$ of $H$, if
$(e_1,e_2,e_3,e_4)$ is the cyclic order of edges
incident to $w$ around $w$ in the plane, then $G$
contains the edges of the 4-cycle $v(e_1)v(e_2)v(e_3)v(e_4)$,
and contains either the edge
$v(e_1)v(e_3)$ or the edge $v(e_2)v(e_4)$; these
five edges of $G$ are said to be \emph{associated with\/} the
4-valent vertex of $H$. Note that $G$ can be obtained from
the medial graph of $H$ by adding a diagonal to every 4-cycle
associated with a 4-valent vertex of~$H$.

\begin{lem}\label{lem:3.4}
Every medial extension $G$ of any graph $H\in \mathcal{H}$ is a
3-connected planar graph satisfying the conditions
{\rm (C1)--(C6)} and\/ {\rm (C8)} of Theorem~\ref{thm:3.1}.
\end{lem}

\begin{proof}
By the construction of $G$, if
$\{v(e),v(e')\}$ is a separating vertex set of $G$,
then the graph $H-e-e'$ is disconnected, a
contradiction, since $H$ is 3-connected. Hence $G$ is
3-connected. Every peripheral cycle
of $H$ \emph{induces} a peripheral cycle of $G$. It is
easy to see that all peripheral cycles of $G$ that are not
induced by the peripheral cycles of $H$ are the 3-cycles formed by
the edges associated with the vertices of $H$.

It is easy to see that $G$ satisfies (C1)--(C5). To show that $G$
satisfies (C6), let
$J=v(e_1)v(e_2)v(e_3)v(e_4)$ be a 4-cycle of
$G$. By the construction of $G$, if vertices $v(e)$
and $v(e')$ of $G$ are adjacent, then the edges $e$
and $e'$ of $H$ are adjacent, too. Since in $H$ no three edges among $e_1$, $e_2$, $e_3$, $e_4$
form a cycle (by (H2)), these four edges either
form a 4-cycle (in this case $J$ is peripheral) or are the
edges incident to a 4-valent vertex of $H$ (in this
case, by the construction of $G$, $J$ is the boundary
of two adjacent faces of $G$). Hence, $G$ satisfies (C6).

It remains to show that $G$ satisfies (C8). Let $C$ and $C'$
be nontriangular peripheral cycles of $G$ such that an
edge $a$ of $C$ is paired with an edge $a'$ of $C'$.
Then, by the construction of $G$, the peripheral cycles
$C$ and $C'$ are induced by peripheral cycles
$\overline{C}$ and $\overline{C}'$ of $H$,
respectively, that are opposite at some 4-valent
vertex $u$. If $C$ and $C'$ have a common vertex
$v(e)$, then $\overline{C}$ and $\overline{C}'$ have
a common edge $e$, hence $H$ has a separating vertex
set $\{u,w\}$, where $w$ is a vertex incident to $e$,
a contradiction, since $H$ is 3-connected. Now
suppose that $G$ has an edge joining a vertex $v(e)$
of $C$ to a vertex $v(e')$ of $C'$ such that at least
one of the vertices $v(e)$ and $v(e')$ is not
incident to $a$ and $a'$. Then the edges $e$ and $e'$
are incident to the same vertex $w$ of $H$ and the
cycles $\overline{C}$ and $\overline{C}'$ pass through $w$. It follows that $\{u,w\}$ is a separating
vertex set of $H$, a contradiction. Next suppose that
$G$ has a path $v(e)v(b)v(e')$ connecting a vertex
$v(e)$ of $C$ to a vertex $v(e')$ of $C'$ such that
$v(b)$ does not belong to $C$ and $C'$, and at least
one of the vertices $v(e)$ and $v(e')$ is not
incident to $a$ or $a'$. If in $H$ the edges $e$,
$b$, and $e'$ are incident to the same vertex $w$,
then $\{u,w\}$ is a separating vertex set of $H$, a
contradiction. If in $H$ the edges $a$ and $b$ (resp.
$b$ and $e'$) are incident to a vertex $w$ (resp.
$w'$) such that $w\neq w'$, then the edge $b$ joins
the vertex $w$ of $\overline{C}$ with the vertex $w'$
of $\overline{C}'$, contrary to (H4). Hence $G$
satisfies (C8). The proof is complete.
\end{proof}

There are medial extensions of graphs in $\mathcal{H}$ that do
not satisfy conditions (C7) and (C9). In the sequel we shall
describe a way to verify the conditions (C7), and henceforth
give examples of graphs satisfying (C1)--(C9).
To show that a medial extension $G$ of $H\in \mathcal{H}$
satisfies (C7) it is convenient to proceed in the
following way. Subdivide every edge $e$ of $H$ by a
two-valent vertex $v(e)$ of $G$. We obtain a graph
$\overline{H}$ whose vertex set is $V(H)\cup V(G)$
where the vertices of $V(G)$ are all 2-valent
vertices of $\overline{H}$. We will consider paths of
$G$ associated with paths of $\overline{H}$
connecting 2-valent vertices.

Two paths $P$ and $P'$ of $\overline{H}$ are \emph{H-disjoint} if
$P \cap P'\cap V(H) = \emptyset$, i.e.,\ $P \cap P'\subseteq V(G)$.

\begin{figure}
\centering
\includegraphics[width=0.7\textwidth]{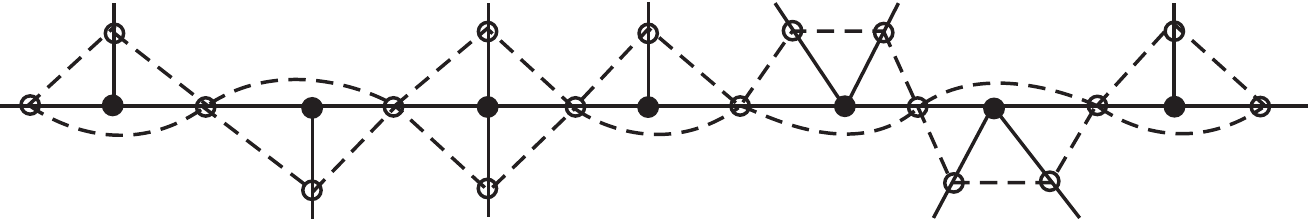}
\caption{Two paths of $G$ associated with a path of $\overline{H}$.}
 \label{Fig:3.7}
  \end{figure}

Consider a path $P=v(e_1)w_1 v(e_2)w_2 v(e_3)\ldots w_{n-1}v(e_n)$
in $\overline{H}$ where $w_1,\allowbreak w_2,\ldots,w_{n-1}$ are
the $H$-vertices on $P$. It is easy to see that the edges of $G$
associated with the vertices $w_1,w_2,\ldots,w_{n-1}$
of $H$ contain two edge-disjoint paths connecting in $G$
the vertices $v(e_1)$ and $v(e_n)$ (see
Fig.~\ref{Fig:3.7}); any two such paths in $G$ are said to be
\emph{associated} with the path $P$ of $\overline{H}$.
Since $H$ has no multiple edges, every edge of $G$ is
associated with exactly one vertex of $H$. Hence, if
$P$ and $P'$ are $H$-disjoint paths in
$\overline{H}$, each of which is connecting 2-valent
vertices, then every path in $G$ associated with $P$
is edge-disjoint from every path in $G$ associated
with $P'$. As a consequence, we have the following conclusion:
{\sloppy

}

\medskip

(C) If $\overline{H}$ has a cycle
    containing 2-valent vertices $v(e)$ and
    $v(e')$, then $G$ has four edge-disjoint paths
    connecting $v(e)$ and $v(e')$.

\medskip

The fact that a path in $H$ gives rise to two edge-disjoint paths
in $G$ (paths associated with the path of $H$) can be used to
check the property (C7) of~$G$.

For a 3-cycle $C$ of $G$, a path of $\overline{H}$ is
\emph{$C$-independent} if the path does not
contain vertices of $C$. When checking (C7) for a medial
extension $G$ of $H\in \mathcal{H}$, given a 3-cycle $C$ of $G$
and two 2-valent vertices $x,y\in V(G)\setminus (V(C)\cup
N(C))\subset V(\overline{H})$, four $C$-independent
edge-disjoint paths $P_1$, $P_2$, $P_3$, and $P_4$ of $G$ connecting $x$ and $y$ in $G$ will be represented in some
subsequent figures (see, e.g., Figure \ref{Fig:3.9})
in the following way. The edges of the paths incident to
vertices of $N(C)$ are depicted as dashed edges joining 2-valent vertices, the dashed edges are not
edges of $\overline{H}$ (see, for example, Fig.~\ref{Fig:3.9}(b),
where the edges of $H$ are given as solid lines). All other edges
of the paths are represented by paths in $\overline{H}$. If $X$
is a subpath of $P_i$ such that $X$ is associated with a path
$\overline{X}$ in $\overline{H}$, then $X$ is
represented by a dashed line passing near the edges
of $\overline{X}$ in $\overline{H}$. If $X_i$ and $X_j$ are
subpaths of $P_i$ and $P_j$ (where possibly $i=j$), respectively,
such that $X_i$ and $X_j$ are edge-disjoint paths associated with a path
$\overline{X}$ of $\overline{H}$, then $X_i$ and
$X_j$ are represented by two (parallel) dashed lines passing
near the edges of $\overline{X}$ in $\overline{H}$.
Using these conventions, the reader will be able to check that the
depicted dashed paths and edges in the figure of
$\overline{H}$ represent four $C$-independent
edge-disjoint paths of $G$ connecting $x$ and $y$.

Now we describe some graphs in $\mathcal{A}$. Let us recall that
graphs in $\mathcal{A}$ are precisely those 3-connected planar graphs
that satisfy conditions (C1)--(C9). To simplify the arguments, we construct
graphs with lots of symmetries so that, for example, to
check the condition (C7) we will have to consider only
two 3-cycles of a graph.

\begin{figure}[t!]
\centering
\includegraphics[width=0.9\textwidth]
{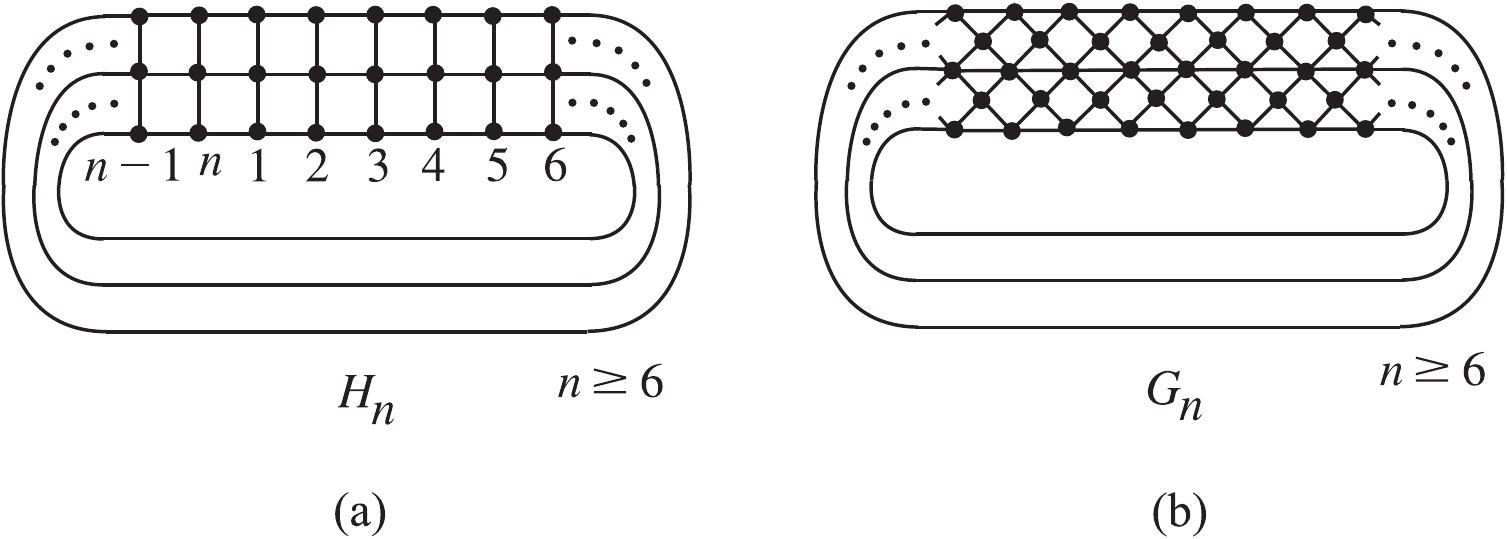} \caption{The graph $H_n\in \mathcal{H}$
and its extension $G_n$.}
 \label{Fig:3.8}
  \end{figure}

For $n\ge6$, let $H_n\in \mathcal{H}$ be the Cartesian product
of the path $P_3$ of length 2 and the cycle $C_n$ of length $n$
(see \ref{Fig:3.8}(a)). Let $G_n$ be the medial extension of $H_n$ shown in Fig.~\ref{Fig:3.8}(b).

\begin{lem}
\label{lem:Extra1}
Each graph $G_n$, $n\geq 6$, is a PN-graph.
\end{lem}
\begin{proof}
We show that $G_n$ satisfies (C1)--(C9) for every $n\geq 6$.

By Lemma~\ref{lem:3.4}, $G_n$
satisfies (C1)--(C6) and (C8). Every 4-gonal face of
$G_n$ is incident with a 6-valent vertex and $G_n$ has
no 5-valent vertices, hence $G_n$ satisfies (C9).

Now we show that $G_n$ satisfies (C7). Consider a
fragment of $\overline{H_n}$ shown in
Fig.~\ref{Fig:3.9}(a) (in the figure we introduce
notation for some vertices and also depict in dashed
lines some edges of $G_n$). Because of the symmetries of
$G_n$, it suffices to consider the following two
cases for a 3-cycle $C$, when checking (C7):

\begin{figure}
\centering
\includegraphics[width=1.0\textwidth]{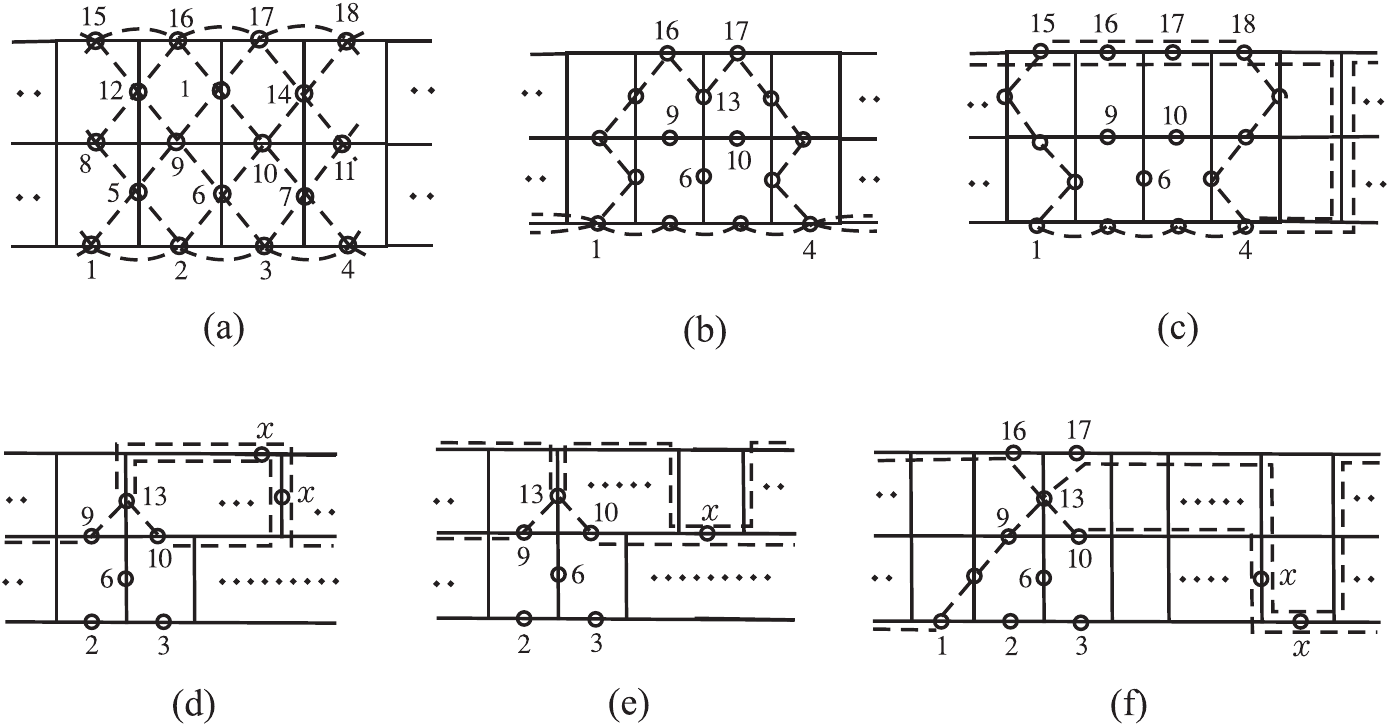}
\caption{Verifying (C7) for the graph $G_n$.}
 \label{Fig:3.9}
  \end{figure}

\medskip

\noindent
\emph{Case 1. $V(C)=\{6,9,10\}$.}

Then $N(C)=\{2,3,5,7,8,11,12,13,14\}$. If we delete from
$\overline{H_n}$ the vertices
$1,2,\ldots,14$, then the obtained graph has only one
connected component $U$ with a vertex in $H_n$, and
$U$ is 2-connected. Hence, any two 2-valent
vertices $x$ and $y$ of $U$ belong to a cycle in $U$
and then, by (C), $G_n$ has four $C$-independent
edge-disjoint paths connecting $x$ and $y$.
Fig.~\ref{Fig:3.9}(b) shows four $C$-independent
edge-disjoint paths in $G_n$ connecting vertices $1$ and $4$.
If we delete from $\overline{H_n}$ the vertices
$\{1,2,\ldots,18\}\setminus \{1,4\}$, then in the
resulting non-trivial connected component,
for every vertex $x\in V(G_n)\setminus \{1,2,\ldots,18\}$, there is a path
$P$ connecting the vertices $1$ and $4$, and passing through $x$;
combining two edge-disjoint paths of $G_n$ associated
with $P$ and the two edge-disjoint paths connecting
$1$ and $4$, shown in Fig.~\ref{Fig:3.9}(b),
we obtain four $C$-independent
edge-disjoint paths connecting the vertices $4$ and
$x$ (and, analogously, for the vertices $1$ and $x$).
Now, because of the symmetries of $\overline{H_n}$, it
remains to show that there are four $C$-independent
edge-disjoint paths connecting the vertex $4$ with each of the
vertices $15,16,17,18$; Fig.~\ref{Fig:3.9}(c) shows
the paths (since $n\ge 6$).

\medskip

\noindent
\emph{Case 2: $V(C) = \{2,3,6\}$.}

We have $N(C)=\{1,4,5,7,9,10\}$. If we delete from
$\overline{H_n}$ the vertices $1,2,\ldots,7$ and $9,10,13$, then
the obtained graph has only one connected component
$U$ with at least two vertices and $U$ is 2-connected. Hence any two vertices $x$ and $y$ in $\overline{H_n}$ that
are both in $U$ belong to a cycle of $U$ and then, by (C),
$G_n$ has four $C$-independent edge-disjoint paths
connecting $x$ and $y$. It remains to show that
for every vertex $x$ of $\overline{H_n}$ belonging to $U$, there are four $C$-independent edge-disjoint paths connecting $x$
and the vertex $13$. These four paths are shown in
Figs.~\ref{Fig:3.9}(d)--(f), depending on the choice of $x$.
We conclude that $G_n$ satisfies (C1)--(C9), hence $G_n$
is a PN-graph for every $n\ge6$.
\end{proof}

\begin{figure}
\centering
\includegraphics[width=0.9\textwidth]{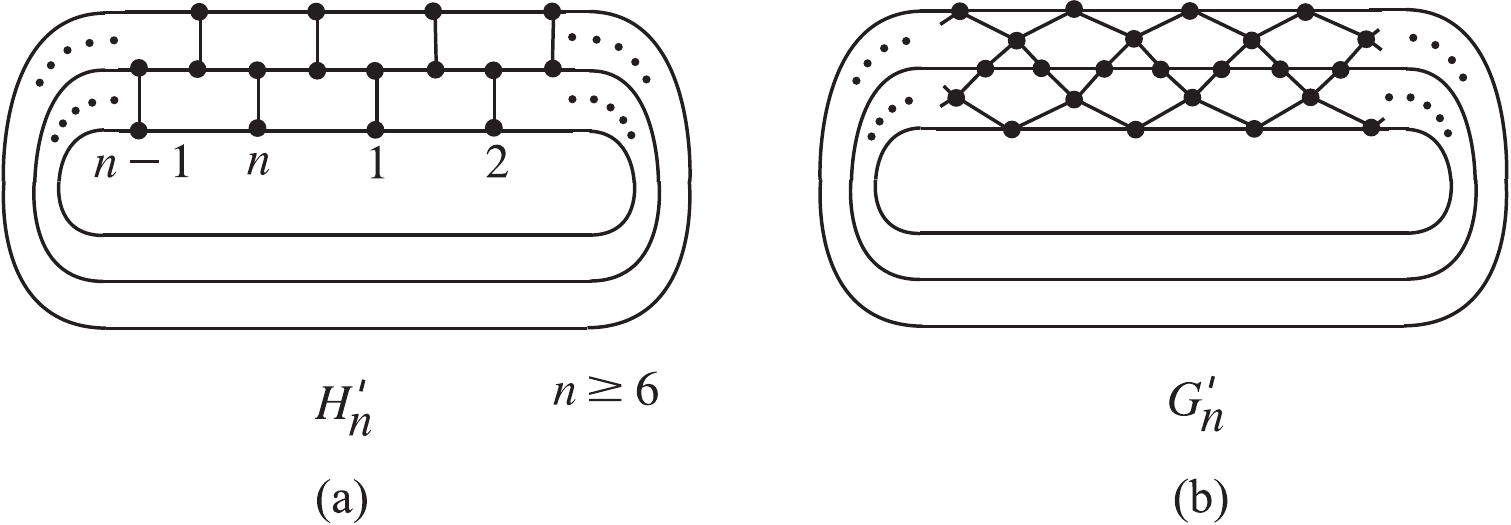}
\caption{The graph $H_n'\in \mathcal{H}$ and its extension $G_n'$.}
 \label{Fig:3.10}
  \end{figure}

Fig.~\ref{Fig:3.10} gives another example of an
extension $G_n'$ of a graph $H_n'\in \mathcal{H}$. By
Lemma~\ref{lem:3.4}, $G_n'$ satisfies (C1)--(C6) and
(C8). Since $G_n'$ has no 4-cycles, it
satisfies (C9). Using the symmetry of $\overline{H_n'}$, one can easily check that for every 3-cycle $C$ of
$G_n'$, if we delete from $\overline{H_n'}$ the vertices
$V(C)\cup N(C)$ of $G_n'$, then the obtained graph has
only one connected component $U$ with at least two
vertices and $U$ is 2-connected. Then, by (C), any
two vertices of $G_n'$ in $U$ are connected by four
$C$-independent edge-disjoint paths. Hence, $G_n'$
satisfies (C7) and is a PN-graph.

\section{MN-graphs based on PN-graphs}\label{Sec:4}

In this section we construct MN-graphs based on the
PN-graphs $G_n$ described in Section~\ref{Sec:3}.

For $m\geq 2$, denote by $S_m$, the graph shown in
Fig.~\ref{Fig:4.1}. The graph has $m+1$ disjoint cycles of
length $12m-2$ labelled by $B_0,B_1,\ldots,B_m$ as
shown in the figure. The vertices of $B_0$ are called
the \emph{central vertices} of $S_m$ and are labelled
by $1,2,\ldots,12m-2$ (see Fig.~\ref{Fig:4.1}). For
every central vertex $x\in\{1,2,\ldots,12m-2\}$,
denote by $x^*$ its ``opposite'' vertex $x+(6m-1)$ if
$x\in\{1,2,\ldots,6m-1\}$ and the vertex $x-(6m-1)$
if $x\in\{6m,6m+1,\ldots,12m-2\}$. In $S_m$, any pair
$\{x,x^*\}$ of central vertices is connected by a
\emph{central path} $P(x,x^*)$ of length $6m-3$ with
$6m-4$ two-valent vertices. There are exactly $6m-1$ central paths.

\begin{figure}[htb]
\centering
\includegraphics[width=0.7\textwidth]
{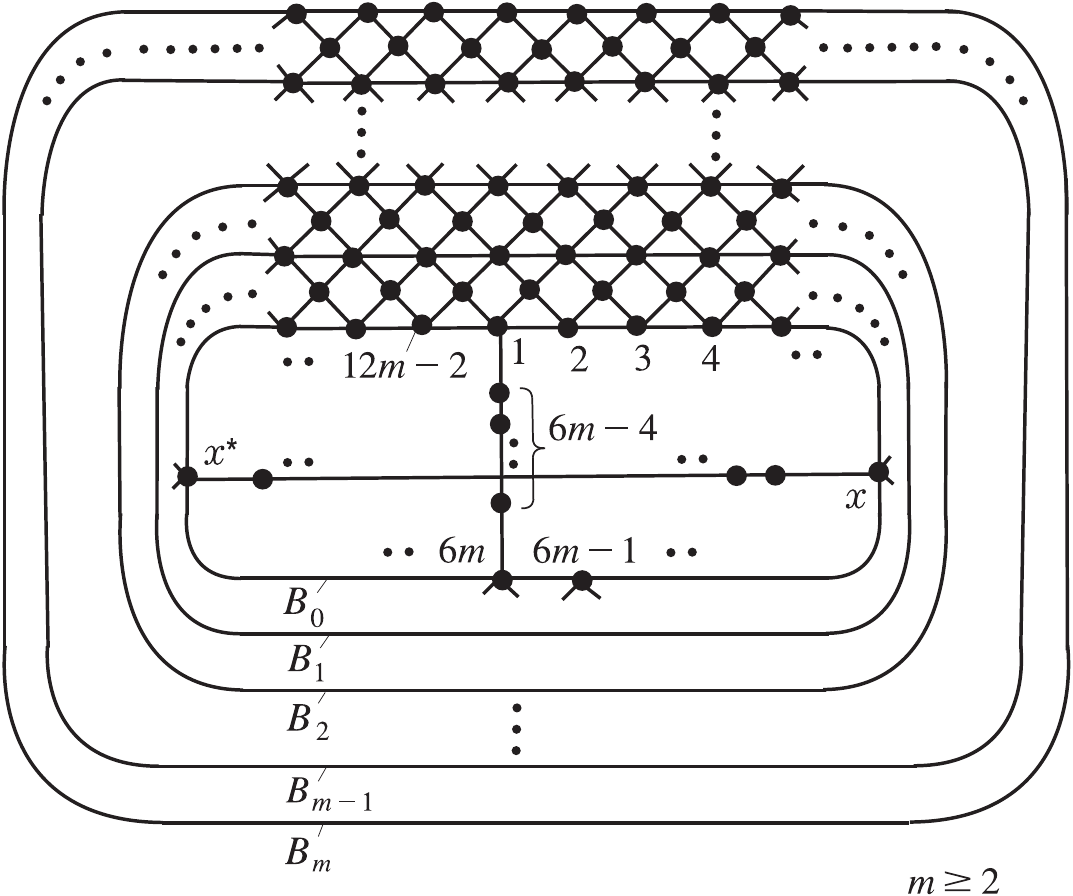} \caption{The graph $S_m$.}
 \label{Fig:4.1}
  \end{figure}

For any integers $m\geq 4$ and $n\geq 0$, denote by
$\Phi_m(n)$ the set of all $(12m-2)$-tuples
$(n_1,n_2,\ldots,n_{12m-2})$ of nonnegative integers
such that $n_1+n_2+\cdots+n_{12m-2}=n$. For every
$\lambda\in\Phi_m(n)$, denote by $S_m(\lambda)$ the
graph obtained from $S_m$ by replacing, for every central vertex
$x\in\{1,2,\ldots,12m-2\}$, the eight edges
marked by short crossings in Fig.~\ref{Fig:4.2}(a)
by $8(1+n_x)$ new edges marked by crossings
in Fig.~\ref{Fig:4.2}(b) (the value $x+1$ in that figure
is to be considered modulo $12m-2$). The graph $S_m(\lambda)$ has $m-2$\quad
$(12m-2)$-cycles $B_0,B_1,\ldots,B_{m-3}$ and three
$(12m-2+n)$-cycles $B_{m-2},B_{m-1},B_m$.

\begin{figure}
\centering
\includegraphics[width=0.7\textwidth]
{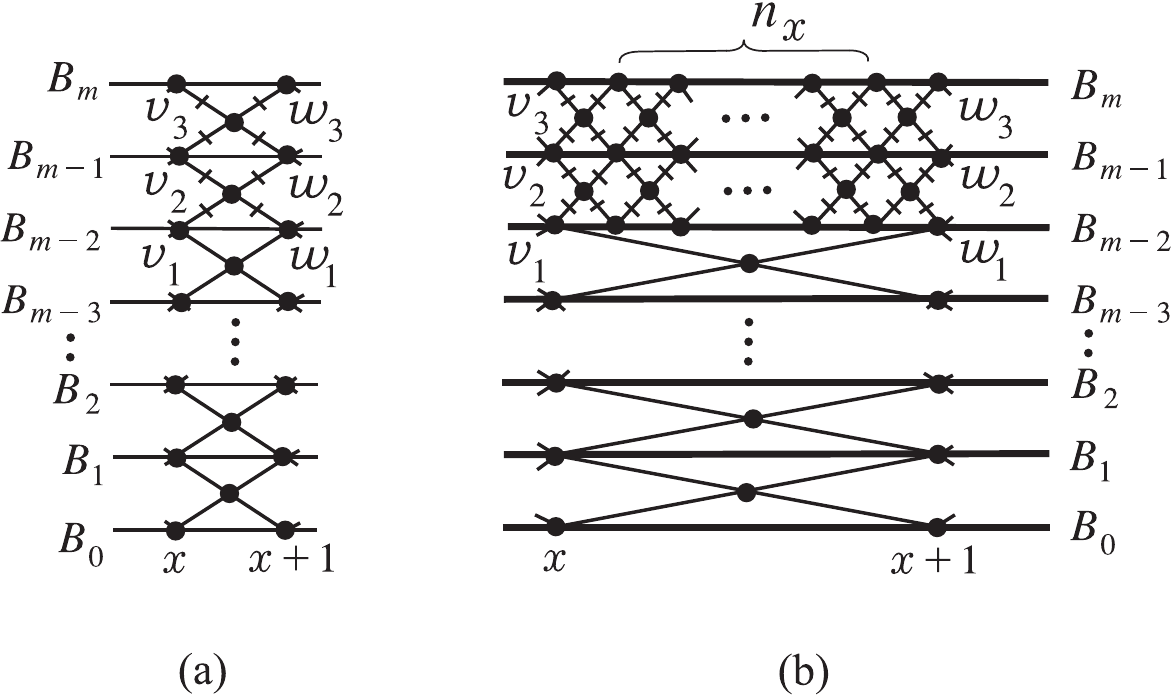} \caption{Obtaining the graph $S_m(\lambda)$.}
 \label{Fig:4.2}
  \end{figure}

We want to show that for every $m\geq 4$ and for
every $\lambda\in \Phi_m(n)$, $n\geq 0$, the graph
$S_m(\lambda)$ is an MN-graph.

\begin{lem}\label{lem:4.1}
For every $m\geq 4$, $n\geq 0$ and $\lambda\in\Phi_m(n)$,
the graph $S_m(\lambda)-e$ is\/ $1$-planar for every edge $e$.
\end{lem}

\begin{proof}
If we delete an edge of a central path,
then the remaining $6m-2$ central paths, each with
$6m-3$ edges, can be 1-immersed inside $B_0$ in
Fig.~\ref{Fig:4.1}. If we delete one of the edges
shown in Fig.~\ref{Fig:4.3}(a) by a thick line, then
the central path $P(x,x^*)$ can be drawn outside
$B_0$ with $6m-3$ crossing points as shown in the figure
and then the remaining $6m-2$ central paths can be
1-immersed inside $B_0$. If we delete one of the two
edges depicted in Fig.~\ref{Fig:4.3}(a) by a dotted
line, then Fig.~\ref{Fig:4.3}(b) shows how to place
the central vertex $x$ so that the path $P(x,x^*)$
can be drawn outside $B_0$ with $6m-3$ crossing
points (analogously to Fig.~\ref{Fig:4.3}(a)) and
then the remaining $6m-2$ central paths can be
1-immersed inside $B_0$. This exhibits all possibilities
for the edge $e$ (up to symmetries of $S_m$) and henceforth
completes the proof.
\end{proof}

\begin{figure}
\centering
\includegraphics[width=0.9\textwidth]{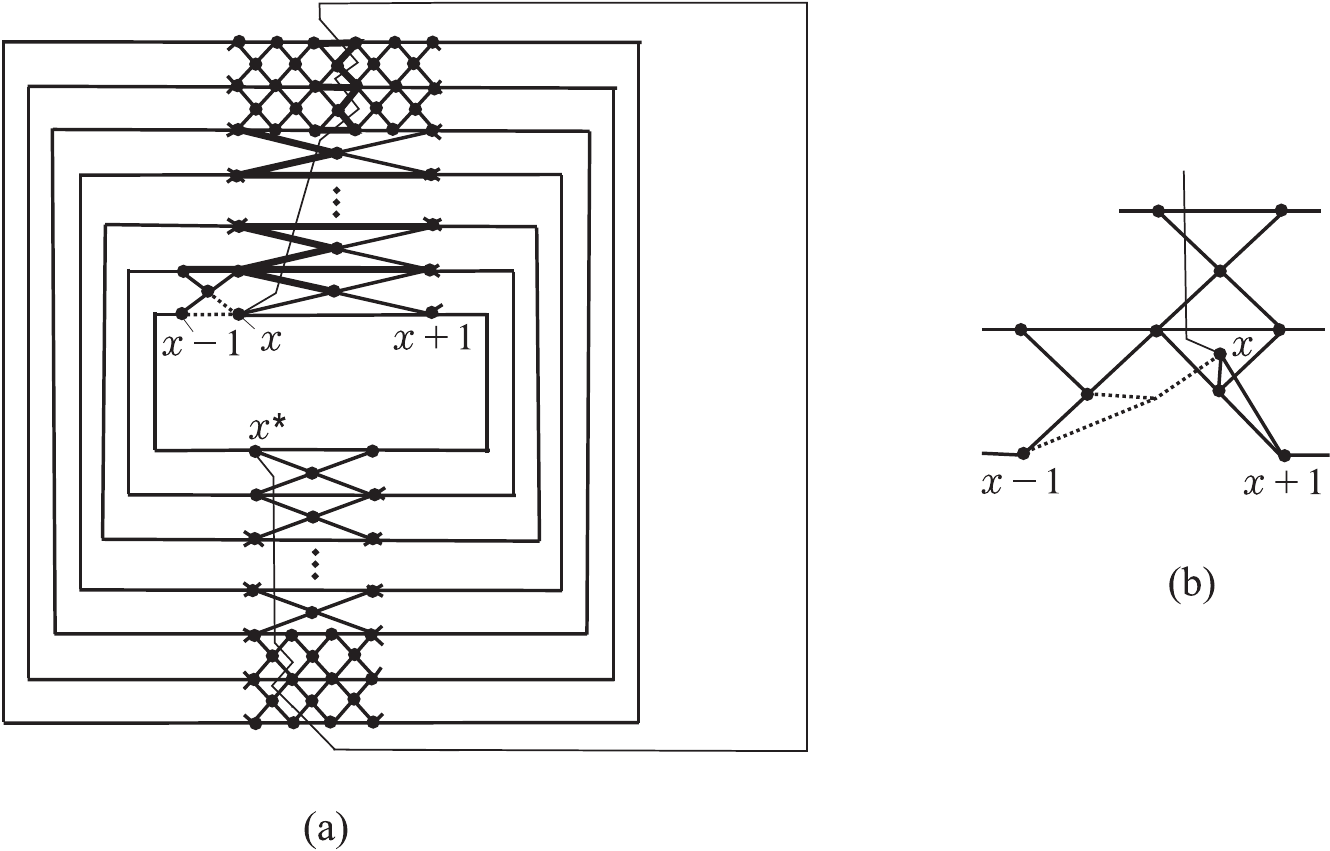}
\caption{The central path $P(x,x^*)$ immersed outside $B_0$.}
 \label{Fig:4.3}
  \end{figure}

Given a 1-immersion of a graph $G$ and an embedded cycle $C$,
we say that $G$ lies \emph{inside} (resp.\ \emph{outside}) $C$,
if the exterior (resp.\ interior) of $C$ does not contain vertices
and edges of $G$.

Denote by $J_{m-2}$ the graph obtained from the graph
$S_m$ in Fig.~\ref{Fig:4.1} by deleting the
2-valent vertices of all central paths and by deleting
all vertices lying outside the cycle $B_{m-2}$.

\begin{lem}\label{lem:4.2}
For every $m\geq 4$, $J_{m-2}$ is a PN-graph.
\end{lem}

\begin{proof} The graph $J_{m-2}$ contains $m-3$
subgraphs $L_1,L_2,\ldots,L_{m-3}$ isomorphic to the
PN-graph $G_{12m-2}$ such that for
$i=1,2,\ldots,m-1$, the graph $L_i$ contains the
cycles $B_{i-1}$, $B_i$, and $B_{i+1}$. Consider an
arbitrary 1-immersion $\varphi$ of $J_{m-2}$. Suppose
that in the plane embedding of the PN-graph $L_1$ in
$\varphi$, the cycle $B_2$ is the boundary cycle of
the outer $(12m-2)$-gonal face of the embedding. Then
the embedding of $L_1$ determines an embedding of the
subgraph of $L_2$ bounded by the cycles $B_1$ and
$B_2$. Since $L_2$ is a PN-graph, the subgraph of
$L_2$ bounded by $B_2$ and $B_3$ lies outside the
cycle $B_2$. Reasoning similarly, we obtain that for
$i=3,4,\ldots,m-3$, the subgraph of the PN-graph
$L_i$ bounded by $B_i$ and $B_{i+1}$ lies outside
$B_i$. As a result, $\varphi$ is a plane embedding of
$J_{m-2}$, hence $J_{m-2}$ is a PN-graph.
\end{proof}

 Denote by $\overline{S}_m(\lambda)$ the graph
 obtained from $S_m(\lambda)$, where $m\geq
 4$ and $\lambda\in \Phi_m(n)$, by deleting the
 2-valent vertices of all central paths.

 \begin{lem}\label{lem:4.3}
 For every $m\geq 4$, $n\geq 0$ and\/ $\lambda\in \Phi_m(n)$,
 $\overline{S}_m(\lambda)$ is a PN-graph.
\end{lem}

\begin{proof}
The graph $\overline{S}_m(\lambda)$
contains a subgraph $G$ isomorphic to the PN-graph
$G_{12m-2+n}$ and contains a subgraph $G'$
homeomorphic to the PN-graph $J_{m-2}$. The graph $G$
contains the cycles $B_{m-2}$, $B_{m-1}$, and $B_m$
of $\overline{S}_m(\lambda)$, and the graph $G'$
contains the cycles $B_0,B_1,\ldots,B_{m-2}$ of
$\overline{S}_m(\lambda)$ and is obtained from $J_{m-2}$
by subdividing the edges of the cycle $B_{m-2}$
(by using $n$ 2-valent vertices in total).

Consider, for a contradiction, a proper 1-immersion
$\varphi$ of $\overline{S}_m(\lambda)$. In $\varphi$,
the graph $G$ has a plane embedding and we shall investigate in which faces of the embedding of $G$ lie the
vertices of $G'$. We shall show that they all lie in the
face of $G$ bounded by the (subdivided) cycle $B_{m-2}$.

In the graph $\overline{S}_m(\lambda)$ the cycles $B_{m-2}$ and $B_i$, $i\in\{0,1,\ldots,m-3\}$ are connected by $24m-4$ edge-disjoint paths. This implies that
no 3- or 4-gonal face of $G$ contains all vertices of $B_i$
in its interior.

Any two vertices of $B_i$ are connected by
six edge-disjoint paths in $G'-B_{m-2}$. Therefore:

(a) No 3- or 4-gonal face of $G$ contains any vertex of the cycles $B_i$, $i=0,1,\ldots,m-3$, in its interior.

Suppose that a vertex $v$ of $G'$ does not belong to the cycles $B_i$, $i=1,2,\ldots,m-2$, and lies inside a 3- or 4-gonal face $F$ of $G$. By construction of $G'$, the vertex $v$ is adjacent to two vertices $w$ and $w'$ of some $B_j$, $j\in\{0,1,\ldots,m-3\}$. By (a), $w$ and $w'$ do not lie inside $F$, hence they lie, respectively, in faces $F_1$ and $F_2$ of $G$ adjacent to $F$. However, at least one of $F_1$ and $F_2$ is 3- or 4-gonal, contrary to (a). Therefore, no 3- or 4-gonal face of $G$ contains any vertex of $G'-B_{m-2}$. If there is a vertex of $G'-B_{m-2}$ outside $B_m$, then $G'$ has two adjacent vertices such that one of them is either a vertex of $B_{m-2}$ or lies inside $B_{m-2}$, and the other vertex lies outside $B_m$, a contradiction, since the edge joining the vertices crosses at least 5 edges of $G$. This implies that all vertices of $G'-B_{m-2}$ lie inside the face of $G$ bounded by $B_{m-2}$.
Hence $G'$ lies inside $B_{m-2}$ and has a proper 1-immersion in $\varphi$. If in this 1-immersion of $G'$ we ignore the 2-valent vertices
on the cycle $B_{m-2}$ of $\overline{S}_m(\lambda)$,
then we obtain a proper 1-immersion of the PN-graph
$J_{m-2}$, a contradiction.
\end{proof}

By the paths of
$\overline{S}_m(\lambda)$ associated with any central
vertex $x$ we mean the two paths shown in
Fig.~\ref{Fig:4.4}; one of them is depicted in thick
line and the other in dashed line. Every edge of $\overline{S}_m(\lambda)$ not belonging
to the cycles $B_0,B_1,\ldots,B_m$ is assigned a
\emph{type} $t\in\{1,2,\ldots,2m\}$ as shown in
Fig.~\ref{Fig:4.4} such that for $i=1,2,\ldots,m$, the
edges of type $2i-1$ and $2i$ are all edges lying
between the cycles $B_{i-1}$ and $B_i$, and the edges
of type $2i-1$ (resp.\ $2i$) are incident to vertices
of $B_{i-1}$ (resp.\ $B_i$).

\begin{figure}
\centering
\includegraphics[width=0.35\textwidth]{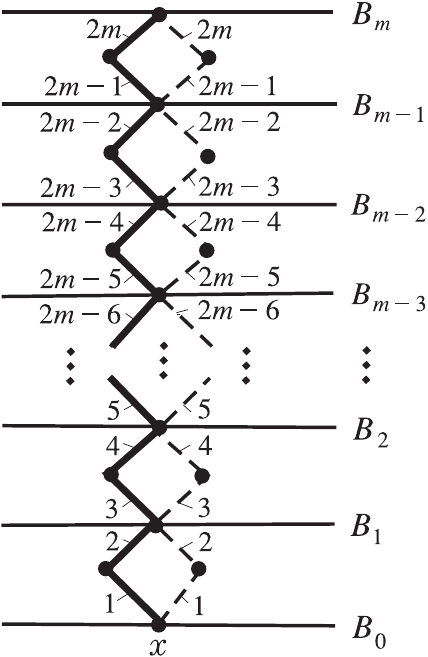}
\caption{The paths associated with a central vertex and the types of edges.}
 \label{Fig:4.4}
  \end{figure}

Suppose that there is a 1-immersion $\varphi$ of $S_m(\lambda)$. By Lemma \ref{lem:4.3}, $\overline{S}_m(\lambda)$ is a PN-graph.
Thus, $\varphi$ induces an embedding of this graph. We shall assume that
the outer face $F_0$ of this embedding is bounded by the cycle $B_m$.
We shall first show that $F_0$ is also a face of $\varphi$. To prove this,
it suffices to see that no central path can enter $F_0$.

Any central vertex $x$ is separated from $F_0$ by $3m-1$
edge-disjoint cycles: $m$ cycles $B_1,B_2,\ldots,B_m$ and $2m-1$ cycles $C_2,C_3,\ldots,C_{2m}$, where the cycle $C_i$ consists of all edges of type $i$
($i=2,3,\ldots,2m$). The central path $P=P(x,x^*)$ can have at most $6m-3$ crossing points, hence $P$ cannot enter $F_0$. If $P$ lies between $B_0$ and $B_m$ in $\varphi$, then it must cross $2(6m-2)$ paths associated either with $6m-2$ central vertices $x+1,x+2,\ldots,x^*-1$ or with $6m-2$ central vertices $x^*+1,x^*+2,\ldots,x-1$ (here we
{interpret all additions modulo $12m-2$), a contradiction. Hence, in $\varphi$ any central path either lies inside $B_0$ or crosses some edges of $\overline{S}_m(\lambda)$ but does not lie entirely between $B_0$ and $B_m$.

The main goal of this section is to show that $S_m(\lambda)$ has no
1-immersions (see Theorem~\ref{thm:4.1} in the sequel). Roughly speaking, the main idea of the proof is as
follows. Suppose, for a contradiction, that
$S_m(\lambda)$ has a 1-immersion. Every central path
can have at most $6m-3$ crossing points, hence, all $6m-1$
central paths can not be 1-immersed inside $B_0$.
Then there is a central path which crosses some edges of
$\overline{S}_m(\lambda)$. Let $P$ be a central path with maximum
number of such crossings. Since $P$ can have at most
$6m-3$ crossing points, some of the other $6m-2$
central paths do not cross $P$ and have to ``go
around" $P$ and, in doing so, one of the
paths has to cross more edges of
$\overline{S}_m(\lambda)$ than $P$ does, a contradiction.

Before proving Theorem~\ref{thm:4.1}, we need some
definitions and preliminary Lemmas~\ref{lem:4.4} and
\ref{lem:4.5}.

Consider a 1-immersion of $S_m(\lambda)$ (if it
exists). If a central path $P=P(x,x^*)$ does not lie
inside $B_0$, consider the sequence
$\delta_1,\delta_2,\ldots,\delta_r$ ($r\geq 2$),
where $\delta_1=x$ and $\delta_r=x^*$, obtained by
listing the intersection points of the path and $B_0$
when traversing the path from the vertex $x$ to the
vertex $x^*$ (here
$\delta_2,\delta_3,\ldots,\delta_{r-1}$ are crossing
points). By a \emph{piece} of $P$ we mean the segment
of $P$ from $\delta_i$ to
$\delta_{i+1}$ for some $i\in\{1,2,\ldots,r-1\}$;
denote the piece by $P(\delta_i,\delta_{i+1})$. A
piece of $P$ with an end point $x$ or $x^*$ is called
an \emph{end piece} of $P$ at the vertex $x$ or
$x^*$, respectively. An \emph{outer piece} of $P$ is
every piece of $P$ that is immersed outside $B_0$. Clearly,
either $P(\delta_1,\delta_2),P(\delta_3,\delta_4),
P(\delta_5,\delta_6),\ldots$ or
$P(\delta_2,\delta_3),P(\delta_4,\delta_5),
P(\delta_6,\delta_7),\ldots$ are all outer pieces of
$P$. The end points $\delta$ and $\delta'$ of an
outer piece $\Pi$ of $P$ partition $B_0$ into two
curves $A$ and $A'$ such that the curve $A$ lies inside
the closed curve consisting of $\Pi$  and $A'$ (see
Fig.~\ref{Fig:4.5}). The central vertices belonging
to $A$ and different from $\delta$ and $\delta'$ are
said to be \emph{bypassed} by $\Pi$ and $P$
(cf.\ Fig.~\ref{Fig:4.5}).

\begin{figure}
\centering
\includegraphics[width=0.2\textwidth]{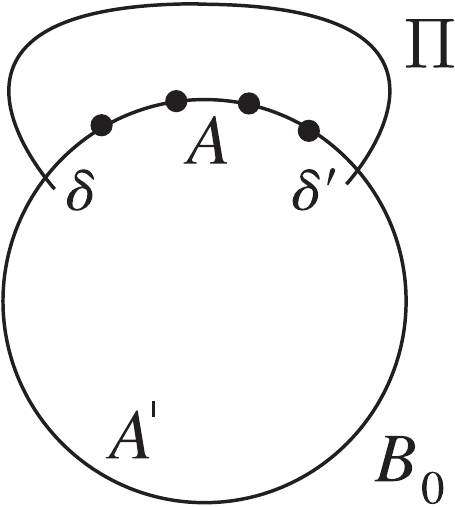}
\caption{The central vertices bypassed by an outer piece $\Pi$.}
 \label{Fig:4.5}
  \end{figure}

\begin{lem}\label{lem:4.4}
If\/ $P(x,x^*)$ bypasses neither a central vertex $y$
nor its opposite vertex $y^*$, and\/ $P(y,y^*)$ bypasses
neither $x$ nor $x^*$, then $P(x,x^*)$ crosses
$P(y,y^*)$.
\end{lem}

\begin{proof}
Suppose, for a contradiction, that
$P(x,x^*)$ does not cross $P(y,y^*)$. For every outer
piece of the two paths we can replace a curve of a
path containing the piece by a new curve lying inside
$B_0$ so that the path $P(x,x^*)$ (resp. $P(y,y^*)$)
becomes a new path $P'(x,x^*)$ (resp. $P'(y,y^*)$)
connecting the vertices $x$ and $x^*$ (resp. $y$ and
$y^*$) such that the two new paths lie inside $B_0$
and do not cross each other, a contradiction. How the
replacements can be done is shown in
Fig.~\ref{Fig:4.6}, where the new curves are depicted
in thick line. (Note that in Fig.~\ref{Fig:4.6}(b),
since $P(y,y^*)$ does not bypass $x$, the depicted
pieces bypassing $x$ belong to $P(x,x^*)$.)
\end{proof}

\begin{figure}
\centering
\includegraphics[width=0.35\textwidth]{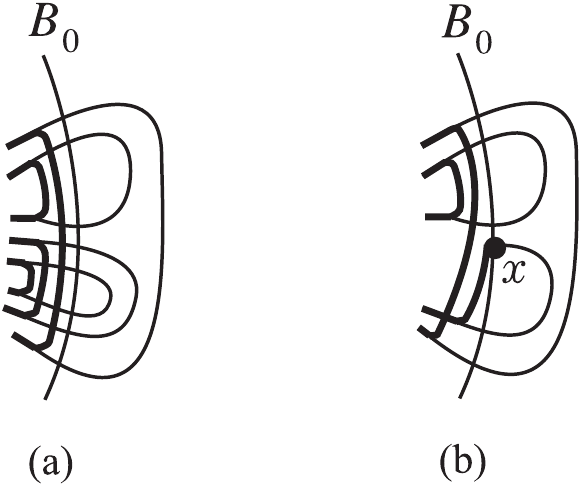}
\caption{Transforming paths $P(x,x^*)$ into paths $P'(x,x^*)$.}
 \label{Fig:4.6}
  \end{figure}


By a \emph{type} of an outer piece of a central path we mean the maximal
type of an edge of $\overline{S}_m(\lambda)$ crossed by the path.

For an outer piece $\Pi$ of a central path $P(x,x^*)$,
denote by $b(\Pi)$ the number of central vertices
bypassed by $\Pi$, and by $\Delta(\Pi)$ the number of
intersection points of
$\Pi$ and $\overline{S}_m(\lambda)$,
including the crossings at the end points of $\Pi$
(except if an end point is $x$ or $x^*$).

\begin{lem}\label{lem:4.5}
If\/ $\Pi$ is an outer piece of type $t$ of a central path $P$, then
\[
\Delta(\Pi)-b(\Pi)\geq 2t-\tau,
\]
where $\tau=1$ if\/ $\Pi$ is an end piece and $\tau=0$ otherwise.
\end{lem}

\begin{proof}
The piece $\Pi$ crossing an edge of
type $t$ has a point separated from the interior of
$B_0$ by $\lfloor\frac{t+1}{2}\rfloor+t$
edge-disjoint cycles:
$B_0,B_1,\ldots,B_{\lfloor\frac{t+1}{2}\rfloor-1}$,
and the cycles $C_1,C_2,\ldots,C_t$, where the cycle
$C_i$ consists of all edges of type $i$ ($i=1,2,\ldots,t$).
Thus, $\Pi$ crosses each of these cycles twice,
except that for an end piece, we may miss one crossing with $C_1$.
The piece $\Pi$ bypasses $b(\Pi)$ central
vertices, hence $\Pi$ crosses $2b(\Pi)$ paths
associated with those $b(\Pi)$ vertices. Hence we obtain two inequalities

\begin{equation}\label{eq:1}
  \Delta(\Pi)\geq 2\left\lfloor\tfrac{t+1}{2}\right\rfloor+2t-\tau,
\end{equation}
and
\begin{equation}\label{eq:2}
  \Delta(\Pi)\geq
  2\left\lfloor\tfrac{t+1}{2}\right\rfloor+2b(\Pi)-\tau.
\end{equation}

\noindent Now add (\ref{eq:1}) and (\ref{eq:2}), apply $2\lfloor\tfrac{t+1}{2}\rfloor\geq t$, then divide by $2$ and rearrange to get the result.
\end{proof}

By the \emph{type} of a central path not lying (entirely) inside
$B_0$ we mean the maximal type of the outer pieces of
the path. If $t$ different central paths bypass a
central vertex $x$, then all the paths cross edges of
the same path $T$ associated with $x$ and since the
edges of $T$ have pairwise different types, we obtain
that one of the central paths crosses an edge of type
at least $t$. Hence we have:

\medskip

(D) If $t$ different central paths bypass
    the same central vertex, then one of the paths
    has type at least $t$.

\medskip

\begin{thm}\label{thm:4.1}
For every $m\geq4$ and $\lambda\in \Phi_m(n)$,
the graph $S_m(\lambda)$ is not\/ $1$-planar.
\end{thm}

\begin{proof}
Consider, for a contradiction, a
1-immersion $\varphi$ of $S_m(\lambda)$ and a path
$P=P(x,x^*)$ of maximal type $t>0$.

As above, let $\Delta(P)$ be the number of crossing points
of $P$ and $\overline{S}_m(\lambda)$, and let $b(P)$ be the number of distinct central vertices bypassed by $P$ and different from $x$ and $x^*$.

There are at least $6m-2-b(P)$ different pairs $\{y,y^*\}$ ($\{y,y^*\}\neq\{x,x^*\}$) of central vertices such
that $P$ does not bypass $y$ and $y^*$; denote by
$\mathcal{P}$ the set of the corresponding (at least  $6m-2-b(P)$) paths
$P(y,y^*)$. If $P(x,x^*)$ does not bypass $y$ and
$y^*$, then $P(y,y^*)$ either does not bypass $x$ and
$x^*$ (in this case $P(y,y^*)$ crosses $P(x,x^*)$
by Lemma~\ref{lem:4.4}) or bypasses at least one
of the vertices $x$ and $x^*$. Hence, we have

\begin{equation}\label{eq:3}
6m-2-b(P)\leq\beta+\gamma+\varepsilon,
\end{equation}
where: $\beta$ is the number of
paths of $\mathcal{P}$ that cross $P$ and do not
bypass $x$ or $x^*$; $\gamma$ is the number of
paths of $\mathcal{P}$ that bypass $x$ or $x^*$
and do not cross $P$; $\varepsilon$ is the number of
paths of $\mathcal{P}$ that cross $P$ and bypass $x$
or $x^*$. We are interested in the number
$\gamma+\varepsilon$ of paths of $\mathcal{P}$ that
bypass $x$ or $x^*$.

The path $P$ has at most $6m-3$ crossing points,
hence
\[
6m-3-\Delta(P)\geq\beta+\varepsilon
\]
and, by (\ref{eq:3}), we obtain
\[
\gamma\geq 6m-2-b(P)-(\beta+\varepsilon)\geq
\Delta(P)-b(P)+1,
\]
whence
\begin{equation}\label{eq:4}
\gamma+\varepsilon\geq\Delta(P)-b(P)+1+\varepsilon.
\end{equation}

Let $\Pi_1,\Pi_2,\ldots,\Pi_\ell$ ($\ell\geq 1$) be
all outer pieces of $P$, and let $\Pi_1$ be of the
maximal type $t$. We have
$\Delta(P)=\sum^\ell_{i=1}\Delta(\Pi_i)$ and
$b(P)\leq\sum^\ell_{i=1}b(\Pi_i)$ (the vertices $x$
and $x^*$ can be bypassed by $P$, and some central vertices can be bypassed by $P$ more than once). By Lemma~\ref{lem:4.5},
$\Delta(\Pi_i)-b(\Pi_i)\geq 0$ for every
$i=1,2,\ldots,\ell$. Hence, by
(\ref{eq:4}) and Lemma~\ref{lem:4.5}, we obtain
\begin{eqnarray}
 \label{eq:5}
 \gamma+\varepsilon &\geq& \sum^\ell_{i=1}\Delta(\Pi_i)
   -\sum^\ell_{i=1}b(\Pi_i)+1+\varepsilon \nonumber \\
 &\geq&
 \Delta(\Pi_1)-b(\Pi_1)+1+\varepsilon \geq (2t+1)
   -\tau+\varepsilon,
\end{eqnarray}
where $\tau=1$ if $\Pi_1$ is an end piece and
$\tau=0$ otherwise. If $\Pi_1$ is not an end
piece or $\varepsilon\geq 1$, then, by
(\ref{eq:5}), $\gamma+\varepsilon\geq 2t+1$, hence
one of the vertices $x$ and $x^*$ is bypassed by at
least $t+1$ paths of $\mathcal{P}$. Now, by (D),
one of the $t+1$ paths has type at least $t+1$, a contradiction. Now suppose that $\Pi_1$ is an end piece at the vertex
$x$ and $\varepsilon=0$. Then every path of
$\mathcal{P}$ either crosses $P$ or bypasses $x$ or
$x^*$, and at least $2t$ paths of $\mathcal{P}$
bypass $x$ or $x^*$. If no one of the $2t$ paths
bypasses $x$, then all the $2t\geq t+1$ paths bypass
$x^*$ and, by (D), one of the paths has type at
least $t+1$. If one of the $2t$ paths, say, $P'$,
bypasses $x$, then $P'$ has an outer piece $\Pi'$
that bypasses $x$ and does not cross $\Pi_1$ (since
$\varepsilon=0$, $P'$ does not cross $P$). The
piece $\Pi_1$ has type $t$ and is an end piece at
$x$, hence $\Pi'$ has type at least $t+1$, a
contradiction.
\end{proof}

We have shown that every graph $S_m(\lambda)$, where
$m\geq 4$ and $\lambda\in \Phi_m(n)$, is an MN-graph.
These graphs have order $(5m-1)(12m-2)+5n$. Clearly,
graphs $S_{m_1}(\lambda_1)$ and $S_{m_2}(\lambda_2)$,
where $\lambda_1\in\Phi_{m_1}(n_1)$ and
$\lambda_2\in\Phi_{m_2}(n_2)$, are nonisomorphic for
$m_1\neq m_2$ and for $m_1=m_2$ and $n_1\neq n_2$.

\begin{clm}\label{clm:2}
For any integers\/ $m\geq 4$ and\/ $n\geq 0$, there are
at least $\frac{1}{(24m-4)}{n+12m-3 \choose 12m-3}$
nonisomorphic MN-graphs $S_m(\lambda)$, where
$\lambda\in \Phi_m(n)$.
\end{clm}

\begin{proof}
The automorphism group of $S_m$ is the dihedral group $D_{12m-2}$
of order $24m-4$. Now the claim follows by recalling a well-known
fact that $|\Phi_m(n)|={n+12m-3 \choose 12m-3}$.
\end{proof}

\section{Testing 1-immersibility is hard}
\label{sect:NPC}

In this section we prove that testing 1-immersibility is NP-hard.
This shows that it is extremely unlikely that there exists a nice
classification of MN-graphs.

\begin{thm}
\label{thm:NPC1}
It is NP-complete to decide if a given input graph is\/ $1$-immersible.
\end{thm}

Since 1-immersions can be represented combinatorially, it is clear that 1-immersability is in NP. To prove its completeness, we shall make a reduction from a known NP-complete problem, that of 3-colorability of planar graphs of maximum degree at most four \cite{GJS}.

The rest of this section is devoted to the proof of Theorem \ref{thm:NPC1}.

Let $G$ be a given plane graph of maximum degree 4
whose 3-colorability is to be tested. We shall show how to construct, in polynomial time, a related graph $\overline{G}$ such that $\overline{G}$ is 1-immersible if and only if $G$ is 3-colorable. We may assume that $G$ has no vertices of degree less than three (since degree 1 and 2 vertices may be deleted without affecting 3-colorability).

To construct $\overline{G}$, we will use as building blocks graphs which have a unique 1-immersion. These building blocks are connected with each other by edges to form a graph which also has a unique 1-immersion. Then we add some additional paths to obtain $\overline{G}$.

We say that a 1-planar graph $G$ has a \emph{unique\/ $1$-immersion\/} if, whenever two edges $e$ and $f$ cross each other in some 1-immersion, then they cross each other in every 1-immersion of $G$, and secondly, if $G^\bullet$ is the planar graph obtained from $G$ by replacing each pair of crossing edges $e=ab$ and $f=cd$ by a new vertex of degree four joined to $a,b,c,d$, then $G^\bullet$ is 3-connected (and thus has combinatorially unique embedding in the plane -- the one obtained from 1-immersions of $G$).

It was proved in \cite{K} that for every $n\geq 6$, the graph with $4n$ vertices and $13n$ edges shown in Fig.~\ref{Fig.5.1}(a) has a unique 1-immersion. (To be precise, the paper \cite{K} considers the graph for even values of $n\geq 6$ only, but one can check that the proof does not depend on whether $n\geq 6$ is even or odd.) We call the graph a \emph{U-graph}. Fig.~\ref{Fig.5.1}(b) shows a designation of the U-graph used in what follows. In the 1-immersion of the U-graph shown in Fig.~\ref{Fig.5.1}, the vertices $1,2,3,\ldots,n-1,n$ which lie on the boundary of the outer face of the spanning embedding (the boundary is called the \emph{outer boundary cycle} of the 1-immersed U-graph) are called the \emph{boundary vertices} of the U-graph in the 1-immersion. If a graph has a U-graph as a subgraph, then the U-graph is called the \emph{U-subgraph} of the graph.

\begin{figure}
\centering
\includegraphics[width=0.90\textwidth]{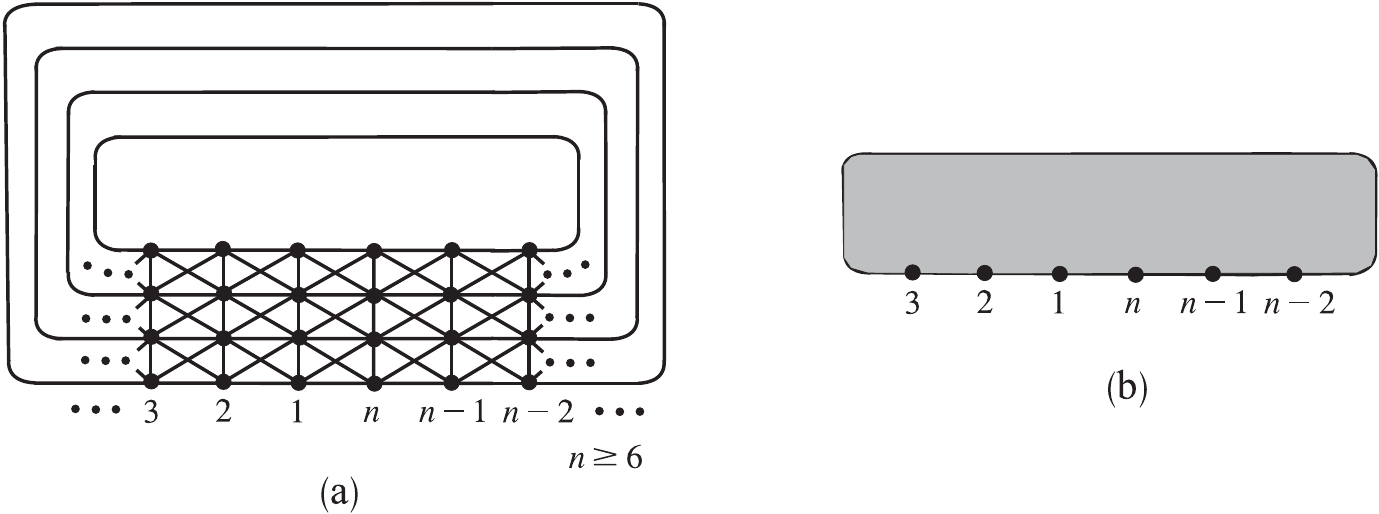}
\caption{The U-graph.}
 \label{Fig.5.1}
  \end{figure}

Take two 1-immersed U-graphs $U_1$ and $U_2$ such that each of them has the outer boundary cycle of length at least 7, and construct the 1-immersed graphs shown in Figs.~\ref{Fig.5.2}(a) and (b), respectively, where by $1,2,\ldots,7$ we denote seven consecutive vertices on the outer boundary cycle of each of the 1-immersed graphs. We say that in Fig.~\ref{Fig.5.2}(a) (resp. (b)) the U-graphs $U_1$ and $U_2$ are connected by a $(1)$-grid (resp. $(2)$-grid). The vertices labeled $1,2,\ldots,7$ are the \emph{basic vertices} of the grid and for $i=1,2,\ldots,7$, the $h$-path connecting the vertices labeled $i$ of the $(h)$-grid, $h\in\{1,2\}$, is called the \emph{basic path} of the grid connecting these vertices. Let us denote the $i$th basic path by $P_i$. The paths $P_{i-1}$ and $P_i$, $i=2,3,\ldots,7$, are \emph{neighboring basic paths} of the grid. For two basic paths $P=P_i$ and $P'=P_j$, $1\leq i<j\leq 7$, denote by $C(P,P')$ the cycle of the graph in Fig.~\ref{Fig.5.2} consisting of the two paths and of the edges $(i,i+1),(i+1,i+2),\ldots,(j-1,j)$ of the two graphs $U_1$ and $U_2$.

\begin{figure}
\centering
\includegraphics[width=0.70\textwidth]{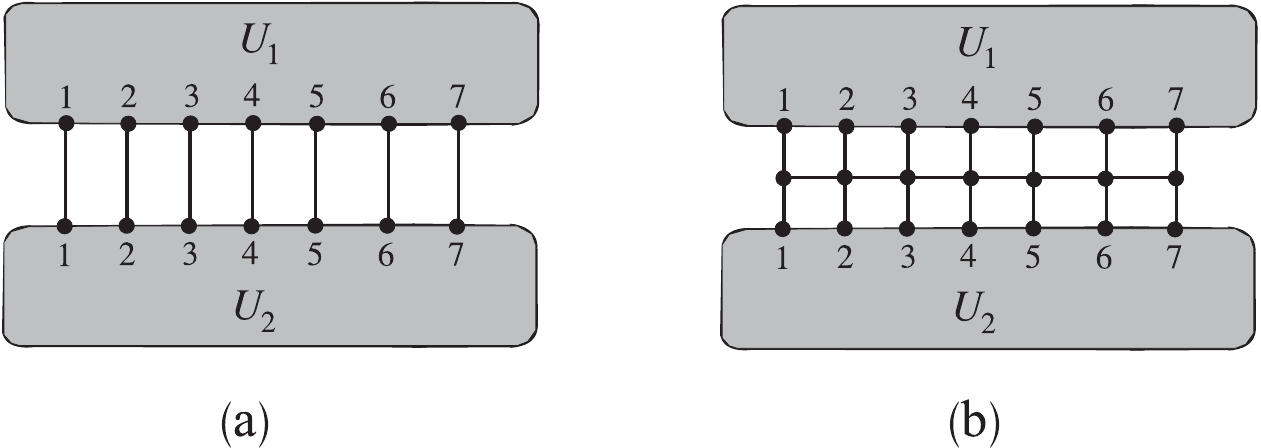}
\caption{Two U-graphs connected by a grid.}
 \label{Fig.5.2}
  \end{figure}

By a \emph{U-supergraph} we mean every graph obtained in the following way. Consider a plane connected graph $H$. Now, for every vertex $v\in V(H)$, take a 1-immersed U-graph $U(v)$ of order at least $28\cdot\deg(v)$ and for any two adjacent vertices $u$ and $w$ of the graph, connect $U(u)$ and $U(w)$ by a $(1)$- or $(2)$-grid as shown in Fig.~\ref{Fig.5.3} such that any two distinct grids have no basic vertices in common. We obtain a 1-immersed U-supergraph.

\begin{figure}
\centering
\includegraphics[width=0.46\textwidth]{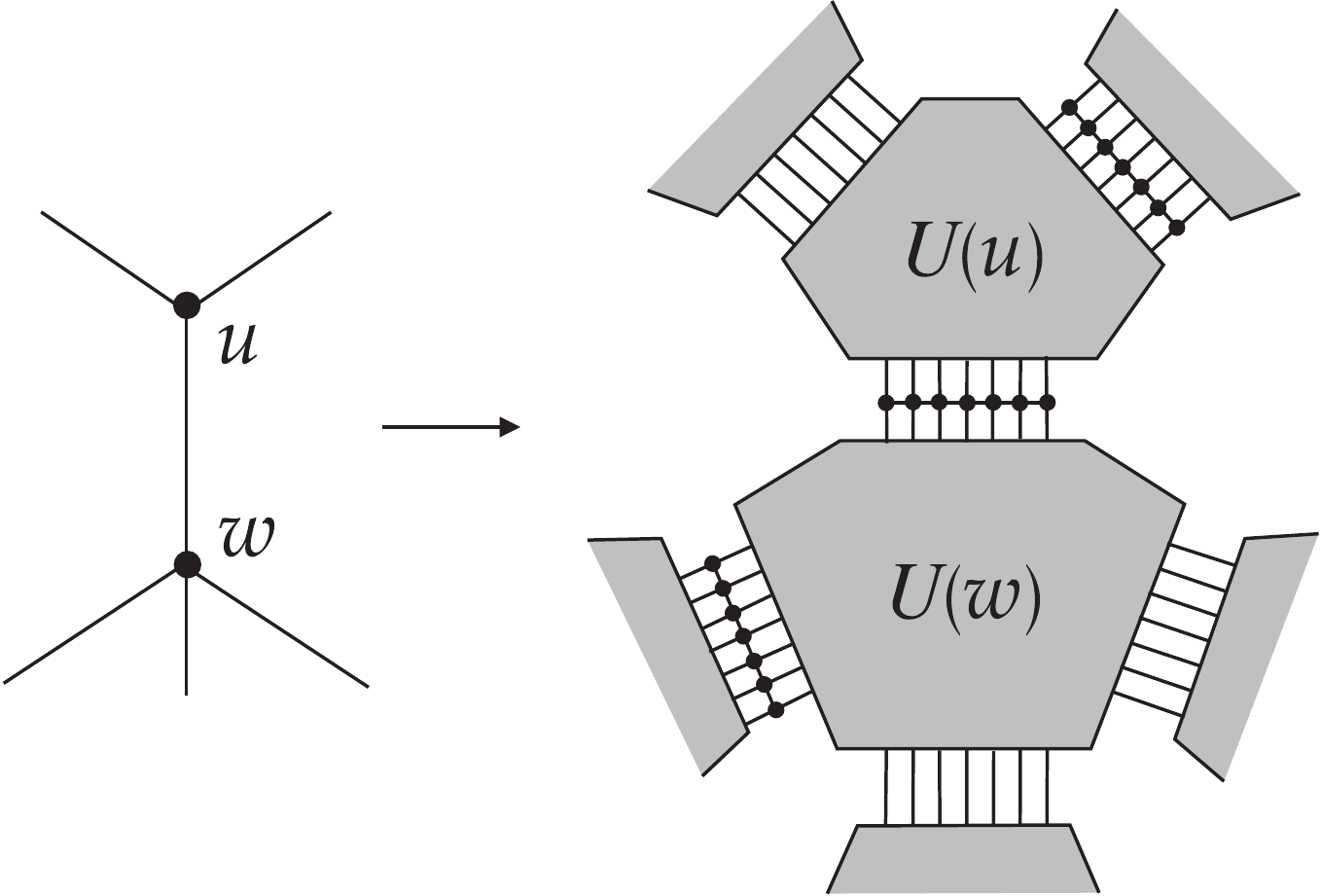}
\caption{Constructing a U-supergraph.}
 \label{Fig.5.3}
  \end{figure}

\begin{thm}
\label{thm:1}
Every U-supergraph $M$ has a unique 1-immersion.
\end{thm}

\begin{proof}
It suffices to show the following:
\begin{itemize}
    \item [(a)]  The graph consisting of two U-graphs connected by an $(h)$-grid, $h=1,2$, has a unique 1-immersion.
    \item [(b)] In every 1-immersion $\varphi$ of $M$, the edges of distinct grids do not intersect.
\end{itemize}
Note that $M$ contains no subgraph which can be 1-immersed inside the boundary cycle of a 1-immersed U-subgraph of $M$ in a 1-immersion of $M$  as shown in Fig.~\ref{Fig.5.4}  in dashed line. Hence, in every 1-immersion of $M$, the boundary edges of the U-subgraphs of $M$ are not crossed.

\begin{figure}
\centering
\includegraphics[width=0.22\textwidth]{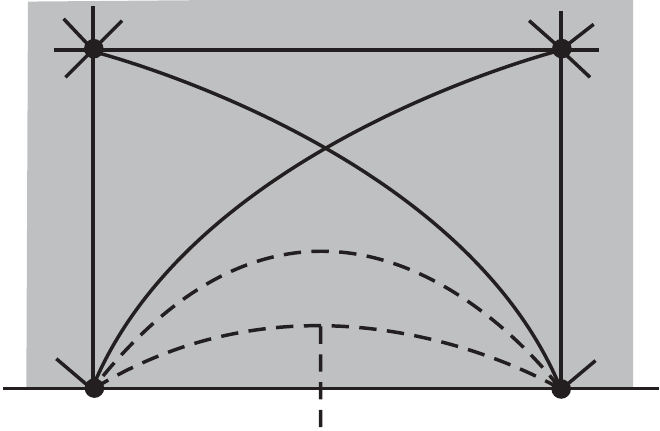}
\caption{A 1-immersion of a subgraph.}
 \label{Fig.5.4}
  \end{figure}

We prove (a) and (b) in the following way. We consider a 1-immersed subgraph $W$ of $M$ (cf.\ Fig.~\ref{Fig.5.2}) consisting of two U-graphs $U_1$ and $U_2$ connected by an $(h)$-grid $\Gamma$, $h\in\{1,2\}$, and we show that in every 1-immersion $\varphi$ of $M$, the graph $W$ has the same 1-immersion and the edges of $\Gamma$ are not crossed by edges of other grids.

Suppose, for a contradiction, that $U_1$ and $U_2$ are 1-immersed under $\varphi$ as shown in Fig.~\ref{Fig.5.5}(a) (the situation described in the figure arises when $U_1$ is drawn clockwise and $U_2$ counter-clockwise (or vice versa)). Clearly, there are two basic paths $P_i$ and $P_j$ of $\Gamma$, $1\leq i<j\leq 7$, which do not intersect. Then the cycle shown in Fig.~\ref{Fig.5.5}(a) in thick line is embedded in the plane, a contradiction, since the cycle is crossed by 5 other basic paths of $\Gamma$, but the cycle has only $2h\leq 4$ edges that can be crossed by other edges. Hence, $U_1$ and $U_2$ are 1-immersed as shown in Fig.~\ref{Fig.5.2}.

\begin{figure}
\centering
\includegraphics[width=1.00\textwidth]{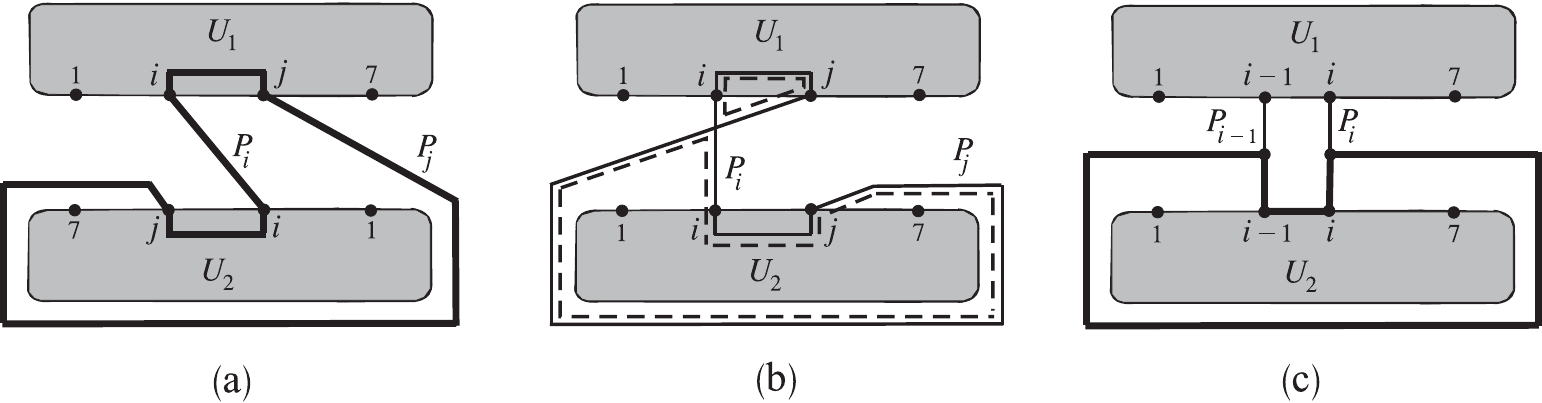}
\caption{Cycles of two adjacent 1-immersed U-subgraphs.}
 \label{Fig.5.5}
  \end{figure}

 Suppose that in $\varphi$, a basic path $P_i$ of $\Gamma$ crosses a basic path $Q$ of some grid of $M$ exactly once. If $Q=P_j$ is a basic path of $\Gamma$, $j\not=i$ (see Fig.~\ref{Fig.5.5}(b)), then the closed curves $C_1$ and $C_2$ shown in Fig.~\ref{Fig.5.5}(b) by dashed cycles, are embedded in the plane and each of the other five basic paths of $\Gamma$ crosses an edge of $C_1$ or $C_2$, a contradiction, since $C_1$ and $C_2$ have $2h-2\leq 2$ edges in total which can be crossed by other edges. If $Q$ is a basic path of a grid $\Gamma'$ different from $\Gamma$, then there is a basic path $P_j$, $j\not= i$, of $\Gamma$ such that $P_j$ is not crossed by $P_i$ and $Q$. Hence, the cycle $C(P_i,P_j)$ is embedded and $Q$ crosses the edges of the cycle exactly once. Then for every other basic path $Q'$ of $\Gamma'$, the cycle $C(Q,Q')$ crosses $C(P_i,P_j)$ at least twice and $Q'$ crosses $P_i$ or $P_j$ (the edges of different U-subgraphs do not intersect). We have that the 7 basic paths of $\Gamma'$ cross $P_i$ and $P_j$, a contradiction. Hence, in $\varphi$, if two basic paths intersect, then they intersect twice. In particular, only basic paths of (2)-grids can intersect.

 Now we claim the following:

 (E) If in $\varphi$ two neighboring basic paths $P_{i-1}$ and $P_i$ of the $(2)$-grid $\Gamma$ do not intersect, then the edge $e$ joining the middle vertices of $P_{i-1}$ and $P_i$ lies inside the embedded cycle $C(P_{i-1},P_i)$.

 Indeed, if $e$ lies outside $C(P_{i-1},P_i)$, then the 4-cycle shown in Fig.~\ref{Fig.5.5}(c) in thick line is crossed by 5 basic paths $P_r$, $r\not=i-1,i$, a contradiction.

 Suppose that in $\varphi$, a basic path of $\Gamma$ crosses a basic path of a grid $\Gamma'$ twice (that is, $\Gamma$ and $\Gamma'$ are $(2)$-grids). Since the number of basic paths of $\Gamma$ is odd (namely, 7), it can not be that every basic path of $\Gamma$ crosses some other basic path of $\Gamma$ twice. Then there is a basic path of $\Gamma$ which is not crossed by other basic paths of $\Gamma$. Hence, if $\Gamma=\Gamma'$ (resp. $\Gamma\not=\Gamma'$), then there are two neighboring basic paths $P_{i-1}$ and $P_i$ of $\Gamma$ such that one of them, say, $P_{i-1}$, is crossed twice by some basic path $Q$ of $\Gamma$ (resp. $\Gamma'$), and the other basic path $P_i\not=Q$ is not crossed by $Q$. Then the cycle $C(P_{i-1},P_i)$ is embedded. By (E), the edge $e$ joining the middle vertices of $P_{i-1}$ and $P_i$ lies inside $C(P_{i-1},P_i)$. Denote by $C_1$ and $C_2$ the two embedded adjacent 4-cycles each of which consists of $e$ and edges of the 6-cycle $C(P_{i-1},P_i)$. The middle vertex of $Q$ lies outside $C(P_{i-1},P_i)$ and the two end vertices of $Q$ lie inside $C_1$ and $C_2$, respectively. The end vertices of $Q$ belong to two U-subgraphs connected by $\Gamma'$. Since the edges of the two U-subgraphs do not cross the edges of $C_1$ and $C_2$, we obtain that one of the U-subgraphs lies inside $C_1$ and the other lies inside $C_2$, a contradiction, since the two U-subgraphs are connected by at least four basic paths different from $Q$, $P_{i-1}$, and $P_i$. Hence, no basic path of $\Gamma$ crosses some other basic path twice.

 We conclude that the basic paths of the grids connecting U-subgraphs do not intersect.

 Now it remains to show that if $\Gamma$ is a $(2)$-grid, then the edges joining the middle vertices of the basic paths of $\Gamma$ are not crossed. Consider any two neighboring basic paths $P_{i-1}$ and $P_i$ of $\Gamma$. The cycle $C(P_{i-1},P_i)$ is embedded and, by (E), the edge $e$ joining the middle vertices of $P_{i-1}$ and $P_i$ lies inside $C(P_{i-1},P_i)$. It is easy to see that for every edge $e'$ of $G$ not belonging to U-subgraphs and different from $e$ and the edges of $C(P_{i-1},P_i)$, in the graph $G-e'$ the end vertices of $e'$ are connected by a path which consists of edges of U-subgraphs and basic paths of grids and which does not pass through the vertices of $C(P_{i-1},P_i)$. Now, if the edge $e$ is crossed by some other edge $e'$, then $e'$ is not an edge of a U-subgraph, the end vertices of $e'$ lie inside the cycles $C_1$ and $C_2$, respectively (where $C_1$ and $C_2$ are defined as in the preceding paragraph) whose edges are not crossed by edges of U-subgraphs and basic paths, a contradiction.

 Therefore, the edges of $M$ do not intersect and that the graph $W$ has a unique 1-immersion. This completes the proof of the theorem.
\end{proof}

Now, given a plane graph $G$ every vertex of which has degree 3 or 4, we construct a graph $\overline{G}$ such that $G$ is 3-colorable if and only if $\overline{G}$ is 1-immersible. To obtain $\overline{G}$, we proceed as follows. First we construct a subgraph $G^{(1)}$ \textbf{of} $\overline{G}$ such that $G^{(1)}$ has a unique 1-immersion. The graph $G^{(1)}$ is obtained from a U-supergraph $W$ by adding some additional vertices and edges. By inspection of the subsequent figures which illustrate the construction of $G^{(1)}$ and its 1-immersion, the reader will easily identify the additional vertices and edges:
they do not belong to U-subgraphs and grids. Then one can easily check that given the 1-immersion of $W$, the additional vertices and edges can be placed in the plane in a unique way to obtain a 1-immersion of $G^{(1)}$, hence $G^{(1)}$ has a unique 1-immersion also. Now, given the unique 1-immersion of $G^{(1)}$, to construct $\overline{G}$ we place some new additional paths ``between" 1-immersed U-subgraphs of $G^{(1)}$. Notice that due to circle inversion one may assume that in a 1-planar drawing of $\overline{G}$ each U-subgraph is drawn in such a way that the outer boundary cycle is containing all other vertices of the U-subgraph inside the region it bounds.

\begin{figure}
\centering
\includegraphics[width=0.90\textwidth]{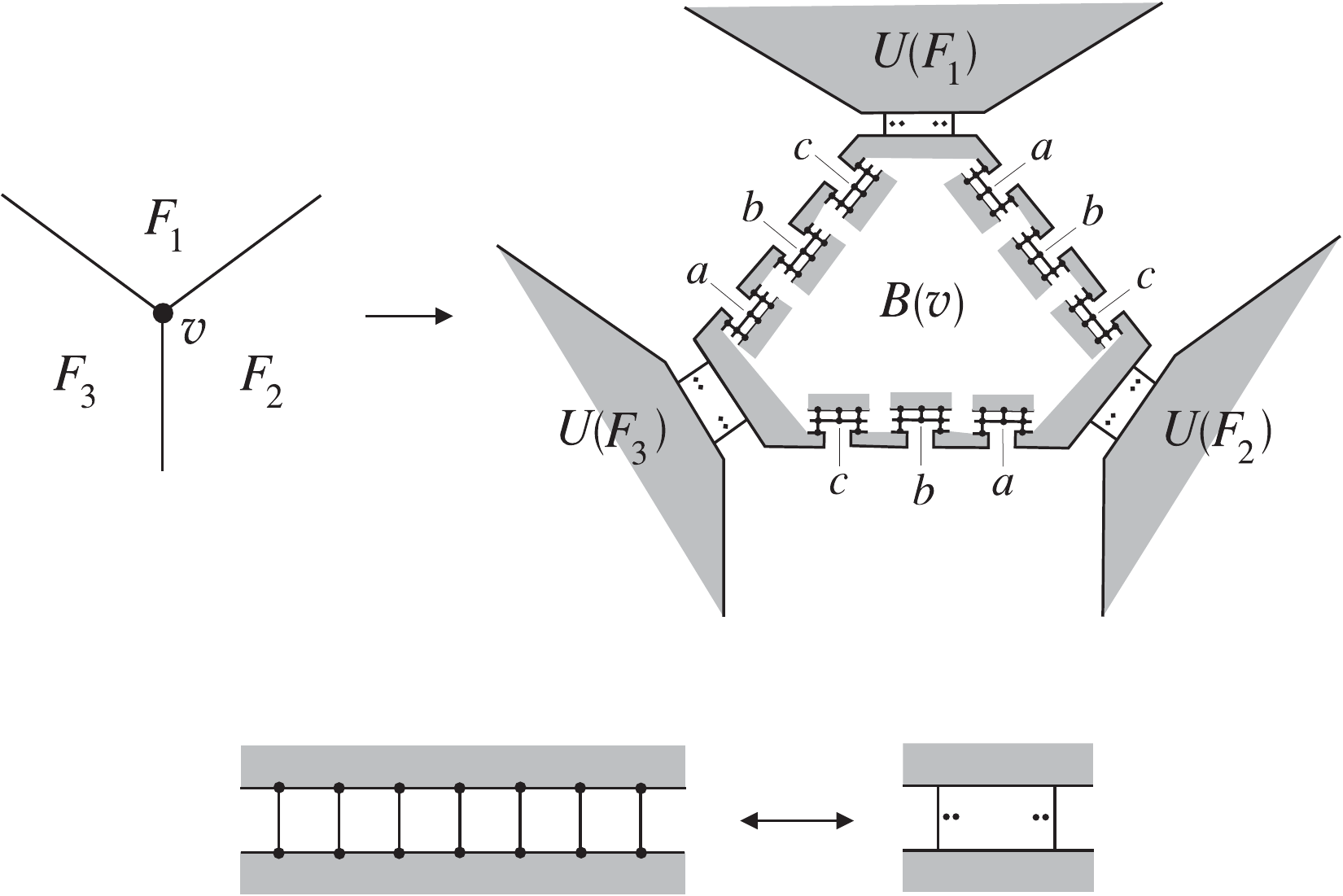}
\caption{Constructing the graph $G^{(1)}$.}
 \label{Fig.5.6}
\end{figure}

The graph $G^{(1)}$ is obtained from the plane graph $G$ if we replace every face $F$ of the embedding of $G$ by a U-graph $U(F)$ and replace every vertex $v$ by a \emph{vertex-block} $B(v)$ as shown at the top of Fig.~\ref{Fig.5.6}. At the bottom of Fig.~\ref{Fig.5.6} we show the designation of a (1)-grid used at the top of the figure and at what follows. The vertex-block $B(v)$ has a unique 1-immersion and is obtained from a U-supergraph by adding some additional vertices and edges. Fig.~\ref{Fig.5.6} shows schematically the boundary of $B(v)$ and Fig.~\ref{Fig.5.9}   shows $B(v)$ in more detail. For a $k$-valent vertex $v$ of $G$, $3\leq k\leq 4$, the vertex-block $B(v)$ has $3k$ \emph{boundary vertices} labeled clockwise as $a,b,c,a,b,c,\ldots,a,b,c$; these vertices do not belong to U-subgraphs of $B(v)$. In Figs.~\ref{Fig.5.6} and \ref{Fig.5.9} we only show the case of a 3-valent vertex $v$; for a 4-valent vertex the construction is analogous -- there are three more boundary vertices labeled $a,b,c$, respectively.

We say that vertex-blocks $B(v)$ and $B(w)$ are adjacent if $v$ and $w$ are adjacent vertices of $G$.

\begin{figure}[htb]
\centering
\includegraphics[width=0.90\textwidth]{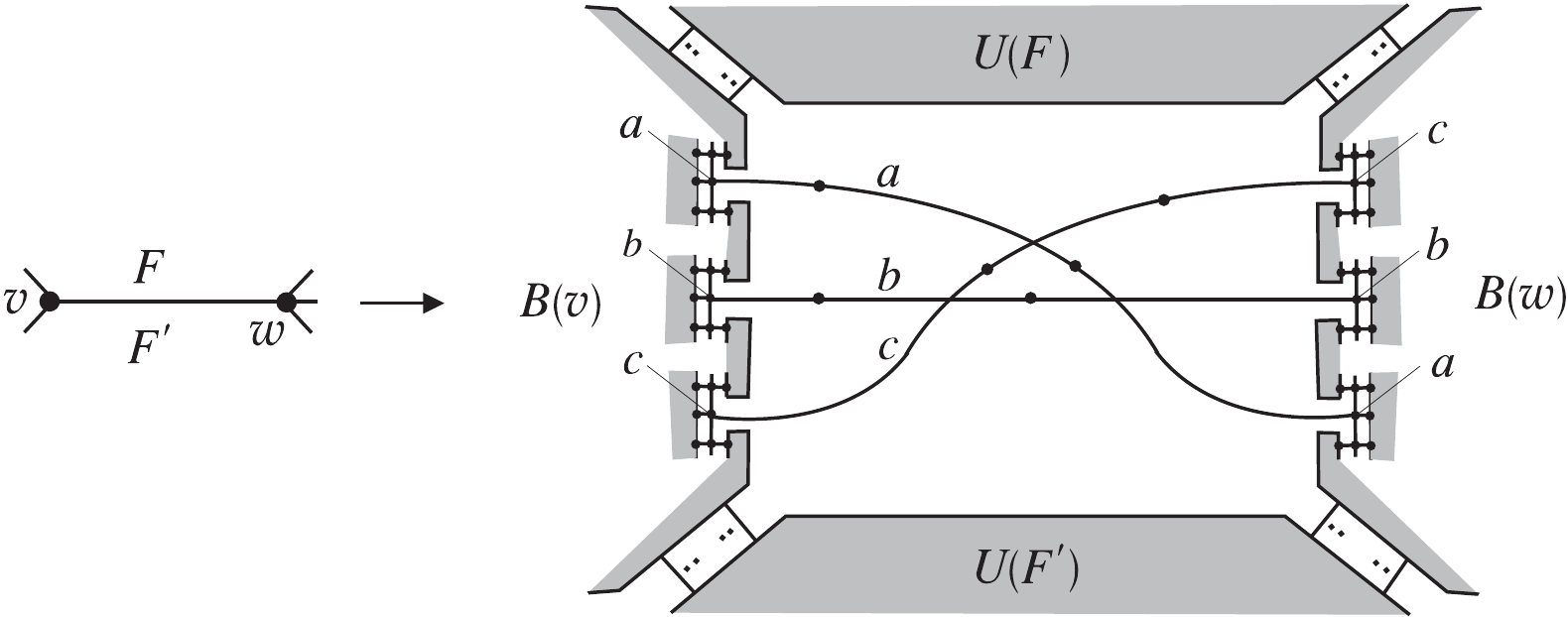}
\caption{The pending paths connecting adjacent vertex-blocks.}
 \label{Fig.5.7}
  \end{figure}

The graph $\overline{G}$ is obtained from $G^{(1)}$ if we take a collection of additional disjoint paths of length $\geq 1$ (they are called the \emph{pending paths}) and identify the end vertices of every path with two vertices, respectively, of $G^{(1)}$. The graph $G^{(1)}$ has a unique 1-immersion and the edges of the U-subgraphs of $G^{(1)}$ can not be crossed by the pending paths, hence the 1-immersed $G^{(1)}$ restricts the ways in which the pending paths can be placed in the plane to obtain a 1-immersion of $\overline{G}$.
Every pending path connects either boundary vertices of adjacent vertex-blocks or vertices of the same vertex-block.

The graph $\overline{G}$ is such that for any two adjacent vertex-blocks, there are exactly three pending paths connecting the vertices of the vertex-blocks. The paths have length 3 and are shown in Fig.~\ref{Fig.5.7}; we say that these pending paths are incident with the two vertex-blocks. Each of the three pending paths connects the boundary vertices labeled by the same letter: $a$, $b$, or $c$. For $h\in\{a,b,c\}$, the pending path connecting vertices labeled $h$ is called the \emph{$(h)$-path\/} connecting the two vertex-blocks. In Fig.~\ref{Fig.5.7} the $(h)$-path, $h\in\{a,b,c\}$, is labeled by the letter $h$.

 Denote by $G^{(2)}$ the graph obtained from $G^{(1)}$ if we add all triples of pending paths connecting vertex-blocks $B(v), B(w)$, for all edges $vw\in E(G)$.

\begin{figure}[htb]
\centering
\includegraphics[width=0.70\textwidth]{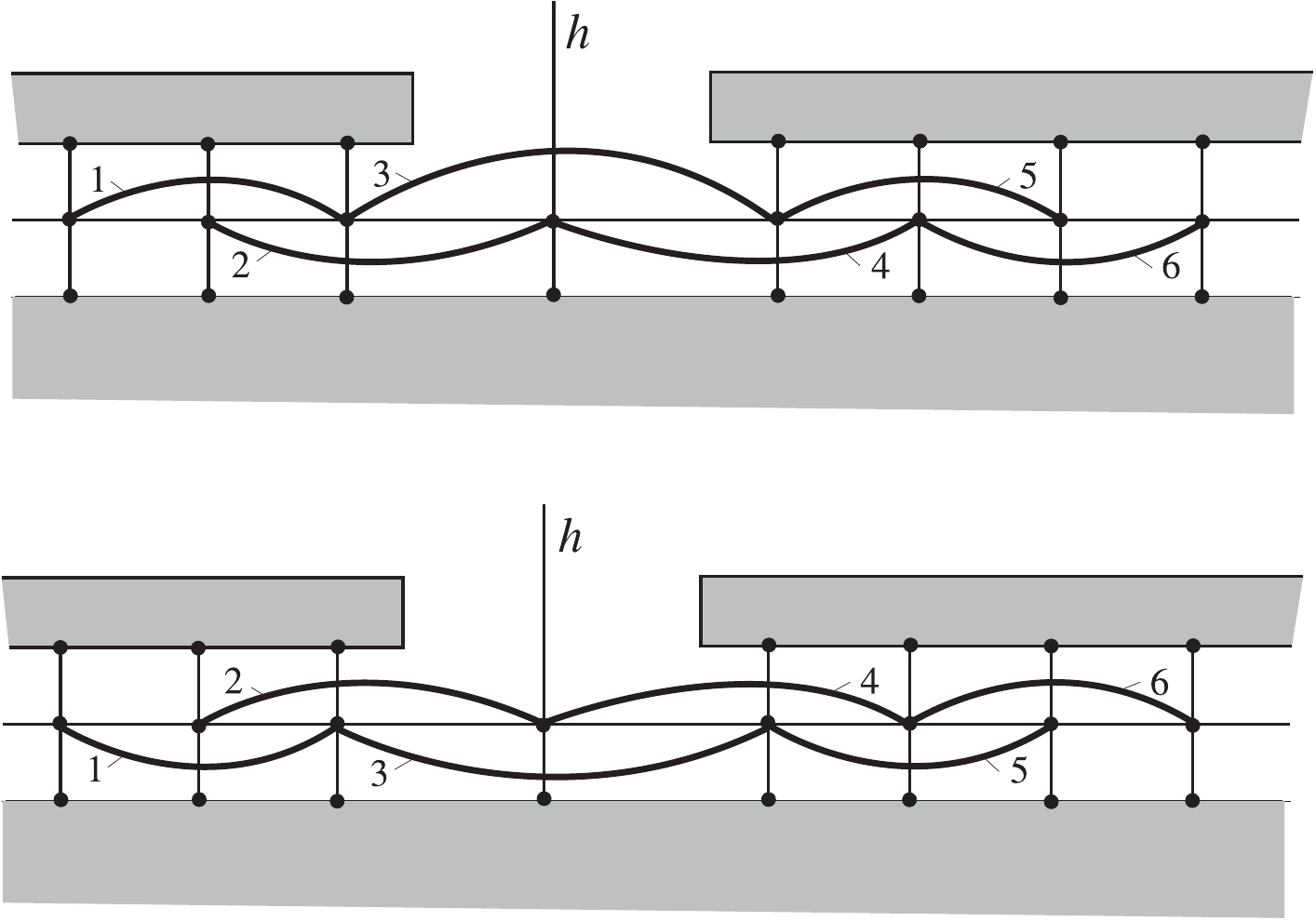}
\caption{Pending paths of an $h$-family.}
 \label{Fig.5.8}
  \end{figure}

The graph $\overline{G}$ also satisfies the properties stated below. The pending paths connecting vertices of the same vertex-block $B(v)$ are divided into three families called, respectively, the $a\,$-, $b\,$-, and \emph{$c$-families} of $B(v)$. Given the 1-immersion of $G^{(1)}$, for every $h\in\{a,b,c\}$, the $h$-family of $B(v)$ has the following properties:
 \begin{itemize}
   \item [(i)] Every path $P$ of the $h$-family admits exactly two embeddings in the plane such that we obtain a 1-immersion of $G^{(2)}\cup P$.
   \item [(ii)] The $h$-family consists of paths $P_1,P_2,\ldots,P_n$ such that the graph $G^{(2)}\cup P_1\cup P_2\cup\cdots\cup P_n$ has exactly two 1-immersions. In the two 1-immersions, every path $P_i$ uses its two embeddings. In one of the 1-immersions, paths of the $h$-family cross all $(h)$-paths incident with $B(v)$. In the other 1-immersion, the paths of the $h$-family do not cross any $(h)$-path incident with $B(v)$.
       {\sloppy

       }
 \end{itemize}

 Fig.~\ref{Fig.5.8} shows fragments of the two 1-immersions of the union of $G^{(2)}$ and the pending paths of an $h$-family. In the figure, each of the depicted (in thick line) six edges of the family, they are labeled by $1,2,\ldots,6$, respectively, uses its two embeddings in one of the two 1-immersions.

 If in a 1-immersion of $\overline{G}$, paths of an $h$-family of $B(v)$, $h\in\{a,b,c\}$, cross $(h)$-paths incident with $B(v)$, then we say that the $h$-family of $B(v)$ is \emph{activated} in the 1-immersion of $\overline{G}$.

\begin{figure}[htb]
\centering
\includegraphics[width=0.85\textwidth]{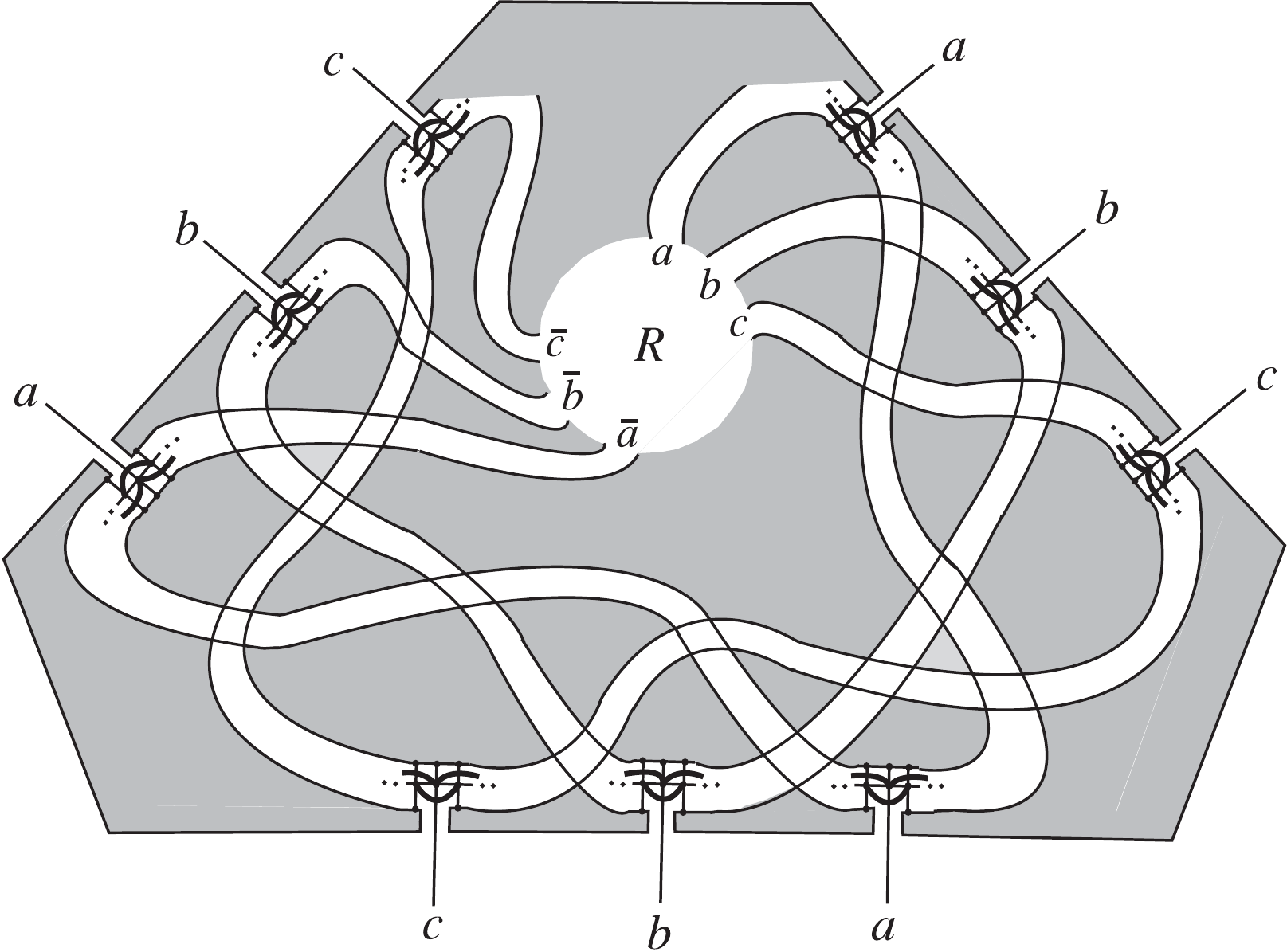}
\caption{A vertex-block and the activated $h$-families.}
 \label{Fig.5.9}
  \end{figure}

 Figs.~\ref{Fig.5.9} and \ref{Fig.5.10} show a vertex-block $B(v)$ and the $h$-families of the vertex-block ($h=a,b,c$) in the case where $v$ is 3-valent (the generalization for a 4-valent vertex $v$ is straightforward). The pending paths of the three $h$-families are shown by thick lines and the three families are activated. To avoid cluttering a figure, Fig.~\ref{Fig.5.9} contains a fragment denoted by $R$ which is given in more detail in Fig.~\ref{Fig.5.10}. Recall that the grey areas in Figs.~\ref{Fig.5.9} and \ref{Fig.5.10} represent U-graphs.

\begin{figure}[t!]
\centering
\includegraphics[width=0.7\textwidth]{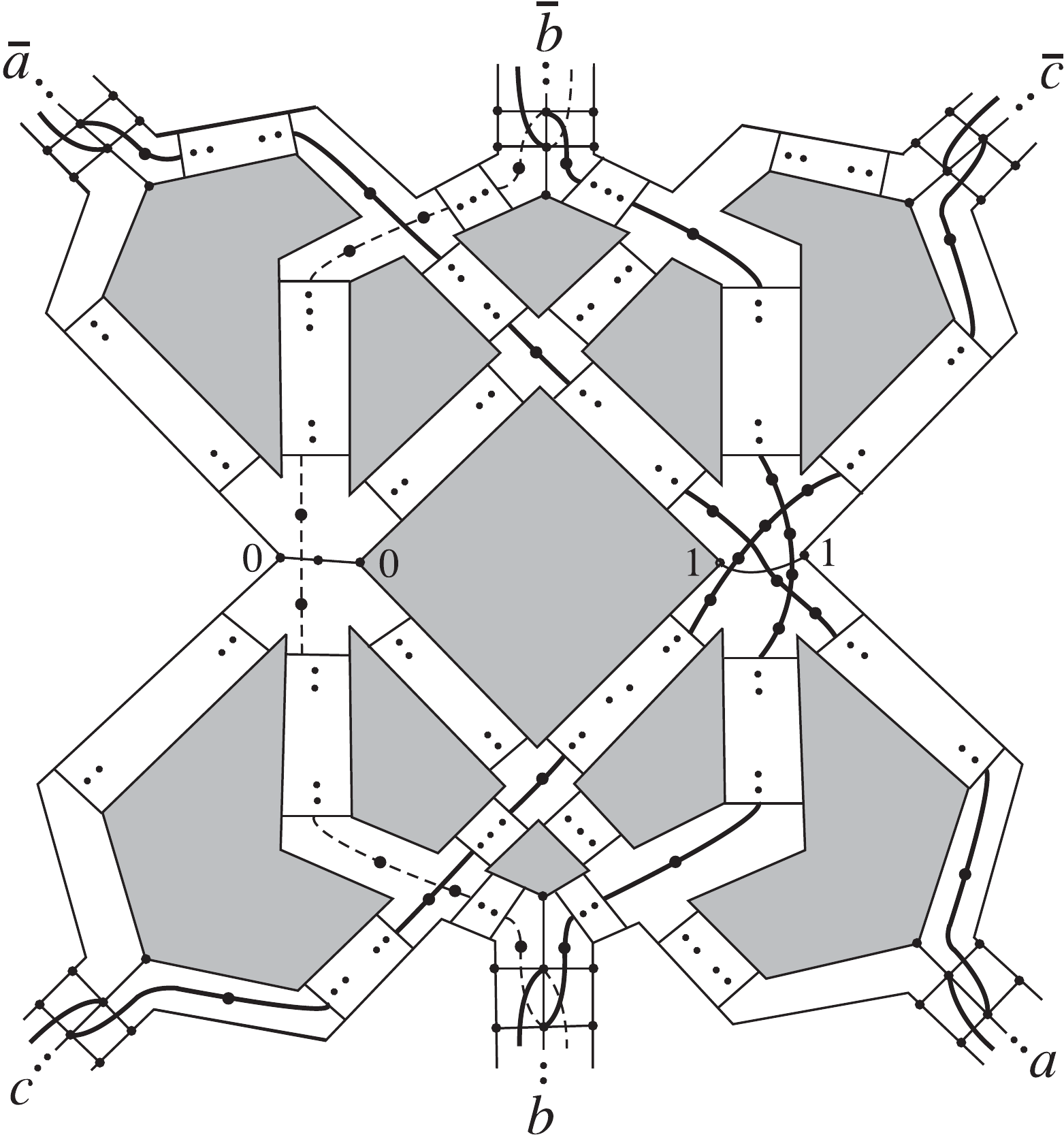}
\caption{The fragment $R$ of the vertex-block in Fig.~\ref{Fig.5.9}.}
 \label{Fig.5.10}
  \end{figure}

In Figs.~\ref{Fig.5.9} and \ref{Fig.5.10} we use designations of some fragments of $B(v)$; the designations are given at the left of Fig.~\ref{Fig.5.11}  and the corresponding fragments are given at the right of Fig.~\ref{Fig.5.11} (that is, the connections between the grey areas in Figs.~\ref{Fig.5.9} and \ref{Fig.5.10} consist of seven basic paths).  The reader can easily check that for every pending path $P$ of the three families, there are exactly two ways to embed the path so that we obtain a 1-immersion of $G^{(2)}\cup P$. The vertex-block $B(v)$ contains a 2-path connecting the vertices labeled 0 in Fig.~\ref{Fig.5.10} and a 1-path connecting the vertices labeled 1 in Fig.~\ref{Fig.5.10}; we call the paths the (0)- and (1)-blocking paths, respectively). For every $h\in\{a,b,c\}$, exactly one pending path of the $h$-family of $B(v)$ crosses a blocking path: the pending path has length 33, crosses the (1)-blocking (resp.\ (0)-blocking) path when the $h$-family is activated (resp.\ not activated), and the pending path in each of its two embeddings crosses exactly one pending path of each of the other two families. Fig.~\ref{Fig.5.10} shows the two embeddings of the pending 33-path of the $b$-family (one of them is in thick line, the other, when the family is not activated, is in dashed line). Note that Fig.~\ref{Fig.5.10} shows something that is not a 1-immersion, since all three families of paths are activated, and the $(1)$-blocking path is crossed three times.

\begin{figure}[t!]
\centering
\includegraphics[width=0.78\textwidth]{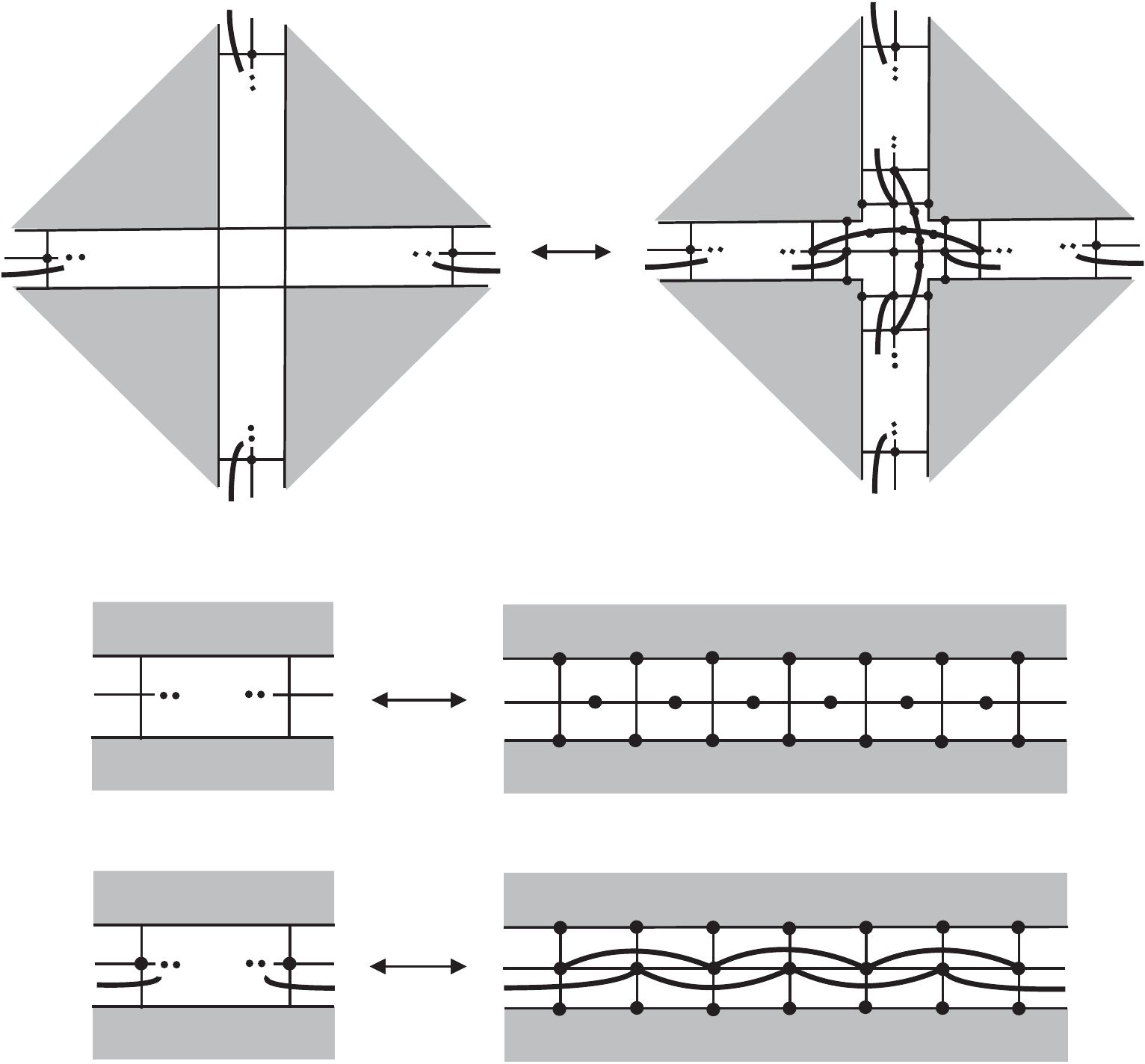}
\caption{The designations of fragments of the vertex-block $B(v)$.}
 \label{Fig.5.11}
  \end{figure}

Denote by $G^{(2)}_v$ the union of $G^{(2)}$ and the paths of all three $h$-families of $B(v)$. Now the reader can check that $B(v)$ and the $h$-families of $B(v)$ are constructed in such a way that the following holds:

\medskip

  (F) In every 1-immersion of $G^{(2)}_v$ (and, hence, of $\overline{G}$) exactly one $h$-family of $B(v)$ is activated, and for each $h\in\{a,b,c\}$, there is a 1-immersion of $G^{(2)}_v$ in which the $h$-family of $B(v)$ is activated.

\medskip

By construction of $\overline{G}$, if $\overline{G}$ has a 1-immersion, then in the 1-immersion for every vertex $v$ of $G$, exactly one $h$-family ($h\in\{a,b,c\}$) of $B(v)$ is activated and taking Fig.~\ref{Fig.5.7} into account we obtain that in the 1-immersion the $h$-families of the vertex-blocks adjacent to $B(v)$ are not activated.

Now take a 1-immersion of $\overline{G}$ (if it exists) and assign every vertex $v$ of $G$ a color $h\in\{a,b,c\}$ such that the $h$-family of $B(v)$ is activated in the 1-immersion of $\overline{G}$. We obtain a proper 3-coloring of $G$ with colors $\{a,b,c\}$.

Take a proper 3-coloring of $G$ (if it exists) with colors $\{a,b,c\}$ and for every vertex $v$ of $G$, if $h(v)$ is the color of $v$, take the  $h(v)$-family of $B(v)$ to be activated and the other two families not to be activated. By the construction of $\overline{G}$, and by the mentioned properties of 1-immersions of its subgraphs $B_v^{(2)}$, it follows that we obtain a 1-immersion of $\overline{G}$.

When constructing $\overline{G}$, we choose the order of every U-subgraph such that every boundary vertex of the U-subgraph is incident with an edge not belonging to the U-subgraph. This implies that for every face $F$ of size $k$ of the plane embedding of $G$, the number of edges in the U-graph $U(F)$ is bounded by a constant multiple of $k$. Similarly, for each $v\in V(G)$, the union of $B(v)$ and its three $h$-families has constant size. Therefore, the whole construction of $\overline{G}$ can be carried over in linear time. This completes the proof of Theorem~\ref{thm:NPC1}.

\section{$k$-planarity testing for multigraphs}
\label{Sec:6}

A graph drawn in the plane is $k$-\emph{immersed} in the plane ($k\geq 1$) if any edge is crossed by at most $k$ other edges (and any pair of crossing edges cross only once). A graph is $k$-\emph{planar} if it can be $k$-immersed into the plane.

It appears that we can slightly modify the proof of Theorem~\ref{thm:NPC1} so as to obtain a proof that $k$-planarity testing ($k\geq 2$) for multigraphs is NP-complete. Below we give only a sketch of the proof, the reader can easily fill in the missing details.

Denote by $G(k)$, $\overline{G}(k)$, and $G^{(1)}(k)$, respectively, the multigraphs obtained from the graphs $G$, $\overline{G}$, and $G^{(1)}$ if we replace every edge by $k$ parallel edges. For an edge $e$ of the multigraphs denote by $H(e)$ the set consisting of $e$ and all other $k-1$ edges parallel to $e$. Denote by $\varphi$ the unique plane 1-immersion of $G^{(1)}$, and by $\varphi_k$ the plane $k$-immersion of $G^{(1)}(k)$ obtained from $\varphi$ if we replace every edge of $G^{(1)}$ by $k$ parallel edges.

\begin{lem}
\label{lem:6.1}
The multigraph $G^{(1)}(k)$, $k\geq 2$, has a unique plane $k$-immersion.
\end{lem}
\begin{proof}
We consider an arbitrary plane $k$-immersion $\psi$ of $G^{(1)}(k)$ and show that $\psi$ is $\varphi_k$.

\begin{figure}[t!]
\centering
\includegraphics[width=0.78\textwidth]{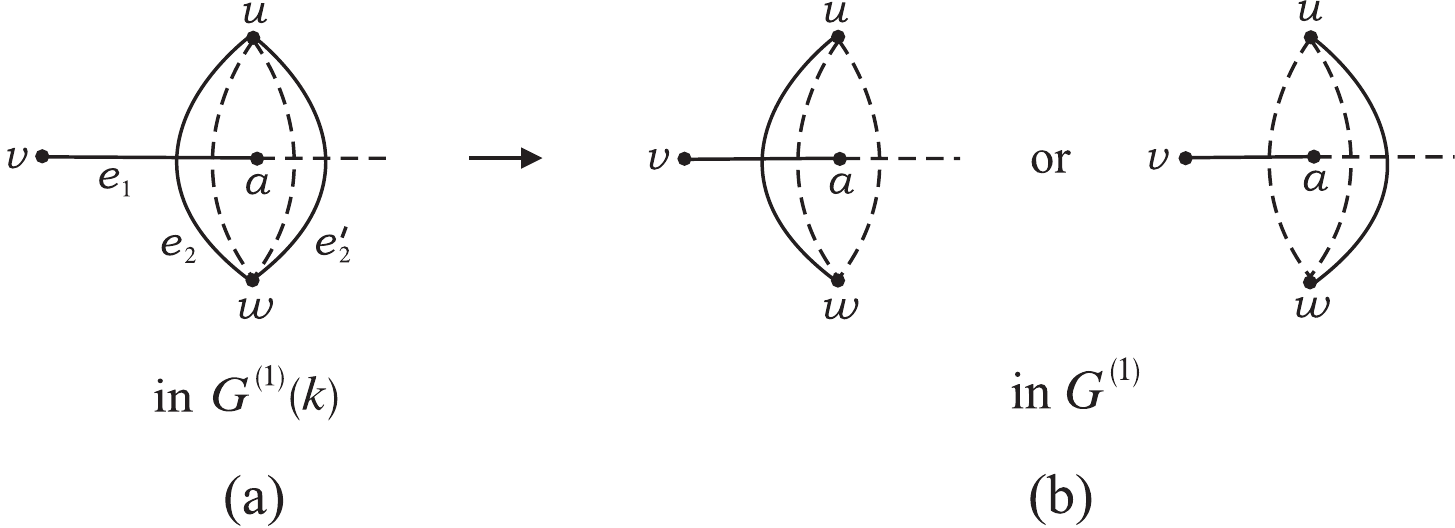}
\caption{Different plane 1-immersions of $G^{(1)}$.}
 \label{Fig.6.1}
  \end{figure}

First we show that if edges $e_1$ and $e_2$ of $G^{(1)}(k)$ cross in $\psi$, then each edge of $H(e_1)$ intersects every edge of $H(e_2)$. Suppose, for a contradiction, that an edge $e'_2$ of $H(e_2)$ does not intersect $e_1$ (see Fig.~\ref{Fig.6.1}(a)). Consider the 2-cell $D$ whose boundary consists of the edges $e_2$ and $e'_2$. Since $e_2$ and $e'_2$ can have at most $2k$ crossings in total, there are at most two vertices lying outside $D$ that are adjacent to vertices inside $D$. This means (see Fig.~\ref{Fig.6.1}(b)) that $G^{(1)}$ has two different plane 1-immersions (the edge of $G^{(1)}$ joining $u$ and $w$ has different positions in the two different plane 1-immersions), a contradiction. Hence, each edge of $H(e_1)$ intersects every edge of $H(e_2)$. Delete $k-1$ edges from every $k$ parallel edges. We obtain a plane 1-immersion of $G^{(1)}$, that is, $\varphi$. Hence $\psi$ is $\varphi_k$.
\end{proof}

The graph $\overline{G}(k)$ is obtained from $G^{(1)}(k)$ if we add the pending paths of $\overline{G}$ where every edge is replaced by $k$ parallel edges. Now, considering a pending path of an $h$-family, we have (see Fig.~\ref{Fig.6.2}, where each thick edge represents $k$ parallel edges) that if $e'\in H(e)$, then each of the edges $e$ and $e'$ is already crossed by $k$ edges of $H(e_1)$ and $H(e_2)$, respectively, thus the edges of the pending path incident with the vertex $v$ can not cross edges $e$ and $e'$, a contradiction. Hence all edges of $H(e)$ cross either $H(e_1)$ or $H(e_2)$. As a result, the $h$-families and the other pending paths of $\overline{G}(k)$ ``behave" in the same way as in $\overline{G}$. We conclude that $\overline{G}$ has a plane 1-immersion if and only if $\overline{G}(k)$ has a plane $k$-immersion. Since $\overline{G}$ has a plane 1-immersion if and only if $G$ has a proper 3-coloring, we get that $k$-planarity testing for multigraphs is NP-complete.

\begin{figure}[t!]
\centering
\includegraphics[width=0.6\textwidth]{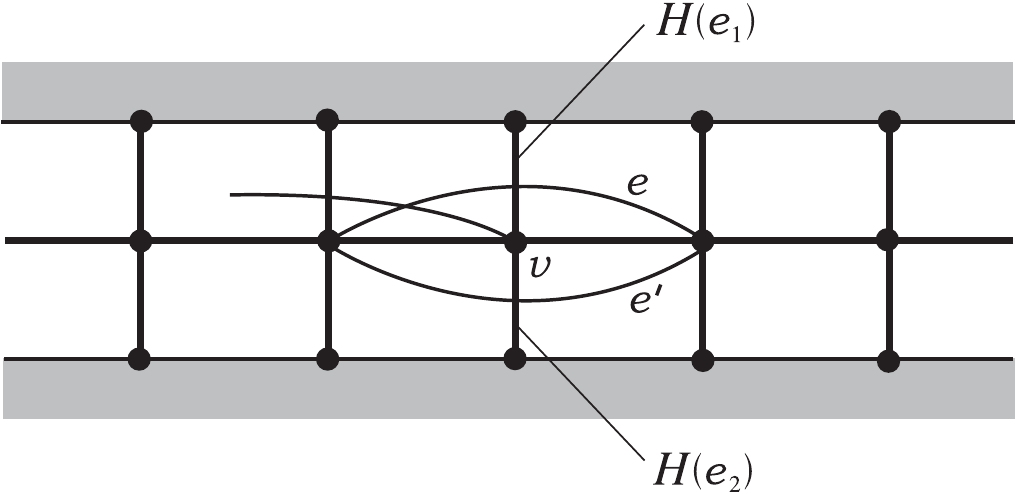}
\caption{Edges of an $h$-family in $\overline{G}(k)$.}
 \label{Fig.6.2}
  \end{figure}

If we restrict ourselves to simple graphs only, then to have a proof analogous to the proof of Theorem~\ref{thm:NPC1} we need simple graphs that have a unique plane $k$-immersion ($k\geq 2$), but the construction of such graphs seems to be nontrivial and does not readily follow from the construction of U-graphs in Sect.~\ref{sect:NPC}.

\subsection*{Acknowledgement.} The authors are grateful to anonymous referee for pointing out that our proof in Section \ref{sect:NPC} might be used to derive a corresponding result for $k$-immersions.


\begin{thebibliography}{99}

\small
\setlength\itemsep{0pt}

\bibitem{B1} O.V. Borodin, Solution of Ringel's problem
about vertex bound colouring of planar graphs and
colouring of 1-planar graphs, Metody Discret. Analiz.
41 (1984) 12--26. 

\bibitem{B2} O.V. Borodin, A new proof of the 6-color theorem,
J. Graph Th.~19 (1995) 507--521.

\bibitem{BKRS} O.V. Borodin, A.V. Kostochka, A.
Raspaud, E. Sopena, Acyclic colouring of 1-planar
graphs, Discrete Analysis and Operations Researcher 6
(1999) 20--35.

\bibitem{BSW} R. Bodendiek, H. Schumacher, and K. Wagner, Bemerkungen zu einen Sechsfarbenproblem von G. Ringel, Abh. Math. Sem. Univ. Hamburg 53 (1983) 41--52.

\bibitem{Ch1} Zhi-Zhong Chen, Approximation algorithms
for independent sets in map graphs, Journal of
Algorithms 41 (2001) 20--40.

\bibitem{Ch2} Zhi-Zhong Chen, New bounds on the number
of edges in a $k$-map graph, Lecture Notes in
Computer Science 3106 (2004) 319--328.

\bibitem{CGP} Zhi-Zhong Chen, Enory Grigni, Christos H. Papadimitriou,
Planar map graphs,
Proceedings of the thirtieth annual ACM symposium on Theory of computing,
May 24--26, 1998, Dallas, Texas, pp.~514--523.

\bibitem{CGP2} Zhi-Zhong Chen, M. Grigni, and C.~H. Papadimitriou,
Map graphs, J. ACM 49 (2002) 127--138.

\bibitem{CK} Zhi-Zhong Chen and M. Kouno, A linear-time
algorithm for 7-coloring 1-plane graphs, Algorithmica
43 (2005) 147--177.

\bibitem{FM} I. Fabrici and T. Madaras, The structure
of 1-planar graphs, Discrete Math. 307 (2007)
854--865.

\bibitem{GJS} M. R. Garey, D. S. Johnson, and L. Stockmeyer,
Some simplified NP-complete graph problems,
Theor.\ Comp.\ Sci. 1 (1976) 237--267.

\bibitem{Hl} P. Hlin\v{e}n\'y,
Crossing number is hard for cubic graphs,
J. Combin.\ Theory, Ser.~B 96 (2006) 455--471.

\bibitem{K} V.P. Korzhik, Minimal non-1-planar graphs,
 Discrete Math. 308 (2008) 1319--1327.

\bibitem{GD08} V.P. Korzhik and Bojan Mohar, Minimal obstructions for 1-immersions and hardness of 1-planarity testing, GD 2008, Lecture Notes in Computer Science 5417 (2009) 302-312.

\bibitem{R} G. Ringel, Ein Sechsfarbenproblem auf der
Kugel, Abh. Sem. Univ. Hamburg 29 (1965) 107--117.

\bibitem{S} H. Schumacher, Zur Struktur 1-planarer Graphen,
 Math. Nachr. 125 (1986) 291--300.

\end{thebibliography}
\end{document}